\documentclass[11pt, leqno]{amsart} 
\usepackage{amssymb,amscd,amsfonts,amsbsy}
\usepackage{latexsym}
\usepackage{exscale}
\usepackage{amsmath,amsthm,amsfonts}
\usepackage{mathrsfs}
\usepackage{xcolor} 
\usepackage[colorlinks=true, linkcolor=blue, citecolor=red, urlcolor=red, 
]{hyperref} 
\usepackage{esint} 
\usepackage{stmaryrd}
\usepackage{pifont}

\usepackage[utf8]{inputenc}

\calclayout

\newcommand{\black}{\color{black}}

\allowdisplaybreaks

\makeatletter
\@namedef{subjclassname@2020}{%
  \textup{2020} Mathematics Subject Classification}
\makeatother

\theoremstyle{plain}
\newtheorem{theorem}[equation]{Theorem}

\newtheorem{lemma}[equation]{Lemma}

\theoremstyle{definition}
\newtheorem{definition}[equation]{Definition}

\numberwithin{equation}{section} 

\def\C{\mathbb{C}}
\def\N{\mathbb{N}}
\def\D{\mathcal{D}}

\def\F{\mathscr{F}}

\def\G{\mathscr{G}}

\def\I{\mathbb{I}}

\def\R{\mathbb{R}}
\def\Rn{\mathbb{R}^n}

\def\Z{\mathbb{Z}}

\def\loc{\operatorname{loc}}

\def\supp{\operatorname{supp}}
\def\BMO{\operatorname{BMO}}
\def\CMO{\operatorname{CMO}}

\def\rs{\operatorname{rs}}
\def\d{\operatorname{d}}
\def\rd{\operatorname{rd}}
\def\ird{\operatorname{ird}}
\def\ch{\operatorname{ch}}
\def\sk{\operatorname{sk}}

\renewcommand{\emptyset}{\text{\textup{\O}}}
\newcommand{\tinyemptyset}{\mbox{\tiny \textup{\O}}}

\begin{document}

\title{A compact extension of Journ\'{e}'s $T1$ theorem on product spaces}

\author{Mingming Cao}
\address{Mingming Cao\\
Instituto de Ciencias Matem\'aticas CSIC-UAM-UC3M-UCM\\
Con\-se\-jo Superior de Investigaciones Cient{\'\i}ficas\\
C/ Nicol\'as Cabrera, 13-15\\
E-28049 Ma\-drid, Spain} \email{mingming.cao@icmat.es}

\author{K\^{o}z\^{o} Yabuta}
\address{K\^{o}z\^{o} Yabuta\\
Research Center for Mathematics and Data Science\\
Kwansei Gakuin University\\
Gakuen 2-1, Sanda 669-1337\\
Japan} \email{kyabuta3@kwansei.ac.jp}

\author{Dachun Yang}
\address{Dachun Yang\\
Laboratory of Mathematics and Complex Systems (Ministry of Education of China) \\
School of Mathematical Sciences\\
Beijing Normal University\\
Beijing 100875\\
The People's Republic of China} \email{dcyang@bnu.edu.cn}

\thanks{The first author acknowledges financial support from Spanish Ministry of Science and Innovation through the Ram\'{o}n y Cajal  2021 (RYC2021-032600-I), through the ``Severo Ochoa Programme for Centres of Excellence in R\&D'' (CEX2019-000904-S), and through PID2019-107914GB-I00, and from the Spanish National Research Council through the ``Ayuda extraordinaria a Centros de Excelencia Severo Ochoa'' (20205CEX001). The third author was supported by the National Natural Science Foundation of China (Grant Nos. 12371093 and 12071197), the National Key Research and Development Program of China (Grant No. 2020YFA0712900).
}

\date{February 15, 2024}




\begin{abstract}
We prove a compact version of the $T1$ theorem for bi-parameter singular integrals. That is, if a bi-parameter singular integral operator $T$ admits the compact full and partial kernel representations, and satisfies the weak compactness property, the diagonal $\CMO$ condition, and the product $\CMO$ condition, then $T$ can be extended to a compact operator on $L^p(w)$ for all $1<p<\infty$ and $w \in A_p(\R^{n_1} \times \R^{n_2})$. Even in the unweighted setting, it is the first time to give a compact extension of Journ\'{e}'s $T1$ theorem on product spaces.
\end{abstract}

\maketitle

\section{Introduction}
\subsection{Motivation and the main theorem}
A central theme in the theory of singular integrals is to establish the boundedness of operators on various function spaces. In the 1980s, David and Journ\'{e} \cite{DJ} proved the original $T1$ theorem, which gave a characterization of the $L^2$ boundedness for Calder\'{o}n--Zygmund operators based on testing conditions. More specifically, they showed that a singular integral operator $T$ associated with a Calder\'{o}n--Zygmund kernel is bounded on $L^2(\Rn)$ if and only if it satisfies the weak boundedness property and $T1, T^*1 \in \BMO(\Rn)$. It is a deep result of crucial importance because many examples of operators, such as the Calder\'{o}n commutators and pseudo-differential operators, can be covered by this result. Besides, the $L^2$ boundedness is usually fundamental to the further analysis, such as achieving $L^p$ boundedness, the endpoint weak (1, 1) estimate, and some weighted estimates.

The theory of multi-parameter singular integrals originated in the work of Fefferman and Stein \cite{FS}, who carefully treated bi-parameter operators of convolution type. Soon after, using the vector-valued Calder\'{o}n--Zygmund theory, Journ\'{e} \cite{Jou} first established a bi-parameter $T1$ theorem for a class of general bi-parameter singular integral operators which are not necessarily of convolution type. Recently, Pott and Villarroya \cite{PV} introduced a new class of bi-parameter singular integral operators, where the vector-valued formulations are replaced by mixed type conditions directly assumed on operators. This machinery was refined by Martikainen  \cite{Mar}, who proved a bi-parameter representation of singular integrals by dyadic shifts, and then obtained, as a pleasant byproduct, a bi-parameter $T1$ theorem. It was observed in \cite{Grau} that the full and partial kernel assumptions in \cite{Mar} are actually equivalent to the vector-valued formulations of Journ\'{e} \cite{Jou}. For more work about the multi-parameter theory and applications, see \cite{CF80, CF85, Fef87, Fef88, FL, MPTT}.

Regarding the compactness of singular integrals, Villarroya \cite{Vil} first presented a new $T1$ theorem to characterize the compactness for a large class of singular integral operators, in the spirit of David and Journ\'{e} \cite{DJ}. The study of the compactness of singular integrals has found useful applications in the field of partial differential equations and geometric measure theory. For example, Fabes et al. \cite{FJR} applied the compactness of layer potentials to solve the Dirichlet and Neumann boundary value problems, while Hofmann et al. \cite[Section 4.6]{HMT} characterized regular SKT domains via the compactness of layer potentials and commutators.

However, the $T1$ theorem to deduce compactness of multi-parameter singular integrals has been open in the last several years. The purpose of this paper is to develop a general theory of the compactness for a class of bi-parameter singular integrals. To this end, the first problem is how to impose appropriate hypotheses on operators in the bi-parameter case because the general bi-parameter singular integrals (cf. Definition \ref{def:SIO}) are not compact on $L^p(\R^{n_1} \times \R^{n_2})$. For example, consider
\begin{align*}
B := \{f \in L^2(\R \times \R): \|f\|_{L^2(\R \times \R)} < 1\} 
\quad\text{ and }\quad 
\mathcal{H} := H \otimes H, 
\end{align*}
where $H$ denotes the Hilbert transform on the real line.  It is not hard to check that $\mathcal{H}(B)=B$. Since $\overline{B}$ is not a compact set in $L^2(\R \times \R)$, we conclude that $\mathcal{H}$ is not compact on $L^2(\R \times \R)$. In view of the work \cite{Mar, Vil}, we will use mixed type conditions involving the size and the H\"{o}lder estimates of the kernel, a weak compactness property, and $\CMO$ conditions. Having formulated proper assumptions on operators, we have to deal with the proof. The proof in the one-parameter setting \cite{Vil} seems too complicated to be generalized to the bi-parameter case because we have to deal with the much added complexity of the bi-parameter situation, especially some interesting mixed type phenomena. Fortunately, dyadic analysis provides us an effective approach.  Furthermore, the multilinear extension will be treated in our forthcoming papers \cite{CLSY, CY}, and in the multi-parameter setting we establish a compact $T1$ theorem for singular integral operators associated with Zygmund dilations \cite{CCLLYZ}.

To set the stage, let us give the definition of the compactness of operators. Given normed spaces $\mathscr{X}$ and $\mathscr{Y}$, a linear operator $T: \mathscr{X} \to \mathscr{Y}$ is said to be \emph{compact} if for all bounded sets $A \subset \mathscr{X}$, the set $T(A)$ is relatively compact in $\mathscr{Y}$, i.e. $\overline{T(A)}$ is compact in $\mathscr{Y}$. Equivalently, $T$ is compact if for all bounded sequences $\{x_n\} \subset \mathscr{X}$, the sequence $\{T(x_n)\}$ has a convergent subsequence in $\mathscr{Y}$. More definitions and notation are given in Sections \ref{sec:SIO} and \ref{sec:pre}.

Now we state our main result as follows.
\begin{theorem}\label{thm:compact}
Let $T$ be a bi-parameter singular integral operator (cf. Definition \ref{def:SIO}). Assume that $T$ satisfies the following hypotheses: 
\begin{list}{\rm (\theenumi)}{\usecounter{enumi}\leftmargin=1.2cm \labelwidth=1cm \itemsep=0.2cm \topsep=.2cm \renewcommand{\theenumi}{H\arabic{enumi}}}

\item\label{list-1} $T$ admits the compact full kernel representation (cf. Definition \ref{def:full}),

\item\label{list-2} $T$ admits the compact partial kernel representation (cf. Definition \ref{def:partial}),

\item\label{list-3} $T$ satisfies the weak compactness property (cf. Definition \ref{def:WCP}),

\item\label{list-4} $T$ satisfies the diagonal $\CMO$ condition (cf. Definition \ref{def:diag-CMO}),

\item\label{list-5} $T$ satisfies the product $\CMO$ condition (cf. Definition \ref{def:prod-BMO}).
\end{list}
Then $T$ is compact on $L^p(w)$ for all $p \in (1, \infty)$ and $w \in A_p(\R^{n_1} \times \R^{n_2})$. 
\end{theorem}

Our main theorem contains several novelties. First, we introduce a general framework for operators and kernels in order to establish the compactness of bi-parameter singular integrals. Even in the unweighted case, we, for the first time, not only extend the result in \cite{Vil} to the case of bi-parameter singular integrals, but also improve Journ\'{e}'s $T1$ theorem on product spaces (cf. \cite{Jou}) to the corresponding compact version. A main technique to achieve this is the weighted interpolation for compact operators (cf. Theorem \ref{thm:inter}), for which the idea is essentially contained in the Rubio de Francia extrapolation for compact operators in papers \cite{COY, HL}. Indeed, inspired by \cite{COY, HL}, to show Theorem \ref{thm:compact}, we just need to establish the weighted boundedness of $T$ on $L^r(v)$ for all $r \in (1, \infty)$ and $v \in A_r(\R^{n_1} \times \R^{n_2})$, and the compactness of $T$ on $L^2(\R^{n_1} \times \R^{n_2})$. The latter is the core of our proof because the former is somehow involved in the recent literature (see e.g. \cite{HPW, LMV}). Second, applying the method of dyadic analysis, our proof is totally different from the one in \cite{Vil}, where the wavelet decomposition was used. As we will see below, back to the one-parameter situation, our argument is much more natural and simpler than that in \cite{Vil}. Third, in the bi-parameter scenario, our result of $L^2(\R^{n_1} \times \R^{n_2})$ compactness (see \eqref{reduction-1} below) refines \cite[Corollary 2.5]{Mar} and our proof does not need a representation theorem, probabilistic arguments, or goodness of cubes. It is worth mentioning that the proof of \eqref{reduction-1} does not depend on \cite[Corollary 2.5]{Mar}, and that the representation theorem in \cite{Mar} is just used to obtain the $L^r(v)$ boundedness of $T$ aforementioned. Finally, when dealing with the setting $I_1 \subsetneq J_1$ or/and $I_2 \subsetneq J_2$, we split the pair $\langle T (h_{I_1} \otimes h_{I_2}), h_{J_1} \otimes h_{J_2} \rangle$ into paraproducts and global terms (see \eqref{STH} or \eqref{TH4}). The product $\CMO$ condition is used to show the compactness of full paraproducts, while global terms are handled by distinguishing that $I_i$ is close to the boundary of $J_i$ (denoted by $I_i \in \mathcal{B}_i^{k_i}(J_i)$) or $I_i$ is far away from the boundary of $J_i$ (denoted by $I_i \in \mathcal{G}_i^{k_i}(J_i)$). If $I_i \in \mathcal{B}_i^{k_i}(J_i)$, we have the decay factor $2^{-k_i \frac{n_i}{2}}$ and $\# \mathcal{B}_i^{k_i}(J_i) \lesssim 2^{k_i(\frac{n_i}{2} - \frac{1}{4})}$; if $I_i \in \mathcal{G}_i^{k_i}(J_i)$, we obtain the decay factor $2^{-k_i (\frac{n_i}{2} + \frac{\delta_i}{2})}$ and $\# \mathcal{G}_i^{k_i}(J_i) \lesssim 2^{k_i \frac{n_i}{2}}$. More details are given in Sections  \ref{sec:SI}, \ref{sec:NI}, and \ref{sec:Inside}.

\subsection{Bi-parameter singular integral operators}\label{sec:SIO}
Once we have stated our main result, we present the precise definitions of all the previous concepts so to have a rigorous idea of what we are handling. First of all, we proceed to define bi-parameter singular integral operators and compact full and partial kernel representations.

\begin{definition}\label{def:FF}
Let $\mathscr{F}$ consist of all triples $(F_1, F_2, F_3)$ of bounded functions $F_1, F_2, F_3: [0, \infty) \to [0, \infty)$ satisfying
\begin{align*}
\lim_{t \to 0} F_1(t)
=\lim_{t \to \infty} F_2(t)
= \lim_{t \to \infty} F_3(t)
=0.
\end{align*}
For $i=1, 2$, let $\mathscr{F}_i$ be the collection of all bounded functions $F_i: \mathcal{Q}_{n_i} \to [0, \infty)$ satisfying
\begin{align*}
\lim_{\ell(I_i) \to 0} F_i(I_i)
= \lim_{\ell(I_i) \to \infty} F_i(I_i)
= \lim_{|c_{I_i}| \to \infty} F_i(I_i)
=0, 
\end{align*}
where $\mathcal{Q}_{n_i}$ is the family of all cubes in $\R^{n_i}$, $i=1, 2$.
\end{definition}

Throughout this article, we always assume that $\delta_1, \delta_2 \in (0, 1]$ are fixed numbers. For convenience, we use $\ell^{\infty}$ metrics on $\R^{n_1}$ and $\R^{n_2}$.

\begin{definition}\label{def:full}
A linear operator $T$ admits a \emph{compact full kernel representation} if the following hold.
If $f = f_1 \otimes f_2$ and $g = g_1 \otimes g_2$ with $f_1, g_1 :\R^{n_1} \rightarrow \C$, $f_2, g_2 :\R^{n_2} \rightarrow \C$,
$\supp(f_1) \cap \supp(g_1) = \emptyset$, and $\supp(f_2) \cap \supp(g_2) = \emptyset$, then
\begin{align*}
\langle Tf, g \rangle
= \int_{\R^{n_1+n_2}} \int_{\R^{n_1+n_2}} K(x,y) f(y) g(x) \, dx \, dy,
\end{align*}
where the kernel $K: (\R^{n_1+n_2} \times \R^{n_1+n_2}) \setminus \big\{(x,y) \in \R^{n_1+n_2} \times \R^{n_1+n_2}: x_1=y_1 \text{ or } x_2=y_2\big\} \rightarrow \C$ satisfies
\begin{list}{\rm (\theenumi)}{\usecounter{enumi}\leftmargin=1.2cm \labelwidth=1cm \itemsep=0.2cm \topsep=.2cm \renewcommand{\theenumi}{\arabic{enumi}}}

\item\label{full-1} the size condition
\begin{align*}
|K(x, y)| \leq \frac{F_1(x_1, y_1)}{|x_1-y_1|^{n_1}} \frac{F_2(x_2, y_2)}{|x_2-y_2|^{n_2}}.
\end{align*}

\item\label{full-2} the H\"{o}lder conditions
\begin{align*}
|&K(x, y) - K((x_1, x'_2), y) - K((x'_1, x_2), y) + K(x', y)|
\\
&\leq \bigg(\frac{|x_1-x'_1|}{|x_1-y_1|} \bigg)^{\delta_1} \frac{F_1(x_1, y_1)}{|x_1 - y_1|^{n_1}}
\bigg(\frac{|x_2-x'_2|}{|x_2-y_2|}\bigg)^{\delta_2} \frac{F_2(x_2, y_2)}{|x_2 - y_2|^{n_2}}
\end{align*}
whenever $|x_1-x'_1| \leq |x_1-y_1|/2$ and $|x_2-x'_2| \leq |x_2-y_2|/2$,
\begin{align*}
|&K(x, y) - K((x_1, x'_2), y) - K(x, (y'_1, y_2)) + K((x_1, x'_2), (y'_1, y_2))|
\\
&\leq \bigg(\frac{|y_1-y'_1|}{|x_1-y_1|}\bigg)^{\delta_1} \frac{F_1(x_1, y_1)}{|x_1 - y_1|^{n_1}}
\bigg(\frac{|x_2-x'_2|}{|x_2-y_2|}\bigg)^{\delta_2} \frac{F_2(x_2, y_2)}{|x_2-y_2|^{n_2}}
\end{align*}
whenever $|y_1-y'_1| \leq |x_1-y_1|/2$ and $|x_2-x'_2| \leq |x_2-y_2|/2$, 
\begin{align*}
|&K(x, y) - K(x, (y_1, y'_2)) - K((x'_1, x_2), y) + K((x'_1, x_2), (y_1, y'_2))|
\\
&\leq \bigg(\frac{|x_1-x'_1|}{|x_1-y_1|}\bigg)^{\delta_1} \frac{F_1(x_1, y_1)}{|x_1 -y_1|^{n_1}}
\bigg(\frac{|y_2-y'_2|}{|x_2-y_2|}\bigg)^{\delta_2} \frac{F_2(x_2, y_2)}{|x_2 - y_2|^{n_2}}
\end{align*}
whenever $|x_1-x'_1| \leq |x_1-y_1|/2$ and $|y_2-y'_2| \leq |x_2-y_2|/2$, and 
\begin{align*}
|&K(x, y) - K(x, (y_1, y'_2)) - K(x, (y'_1, y_2)) + K(x, y')|
\\
&\leq \bigg(\frac{|y_1-y'_1|}{|x_1-y_1|} \bigg)^{\delta_1} \frac{F_1(x_1, y_1)}{|x_1 - y_1|^{n_1}}
\bigg(\frac{|y_2-y'_2|}{|x_2-y_2|} \bigg)^{\delta_2} \frac{F_2(x_2, y_2)}{|x_2 - y_2|^{n_2}}
\end{align*}
whenever $|y_1-y'_1| \leq |x_1-y_1|/2$ and $|y_2-y'_2| \leq |x_2-y_2|/2$. 

\item\label{full-3} the mixed size-H\"{o}lder conditions
\begin{align*}
|K(x, y) - K((x'_1, x_2), y)|
\leq \bigg(\frac{|x_1-x_1'|}{|x_1-y_1|}\bigg)^{\delta_1} \frac{F_1(x_1, y_1)}{|x_1 -y_1|^{n_1}}
\frac{F_2(x_2, y_2)}{|x_2-y_2|^{n_2}}
\end{align*}
whenever $|x_1 - x'_1| \leq |x_1-y_1|/2$, 
\begin{align*}
|K(x, y) - K((x_1, x'_2), y)|
\leq \frac{F_1(x_1, y_1)}{|x_1-y_1|^{n_1}}
\bigg(\frac{|x_2 - x'_2|}{|x_2-y_2|}\bigg)^{\delta_2} \frac{F_2(x_2, y_2)}{|x_2 - y_2|^{n_2}}
\end{align*}
whenever $|x_2 - x'_2| \leq |x_2-y_2|/2$,
\begin{align*}
|K(x, y) - K(x, (y'_1, y_2))|
\leq \bigg(\frac{|y_1-y_1'|}{|x_1-y_1|}\bigg)^{\delta_1} \frac{F_1(x_1, y_1)}{|x_1 - y_1|^{n_1}}
\frac{F_2(x_2, y_2)}{|x_2-y_2|^{n_2}}
\end{align*}
whenever $|y_1 - y'_1| \leq |x_1-y_1|/2$, and
\begin{align*}
|K(x, y) - K(x, (y_1, y'_2))|
\leq \frac{F_1(x_1, y_1)}{|x_1-y_1|^{n_1}}
\bigg(\frac{|y_2-y'_2|}{|x_2-y_2|}\bigg)^{\delta_2} \frac{F_2(x_2, y_2)}{|x_2 - y_2|^{n_2}}
\end{align*}
whenever $|y_2 - y'_2| \leq |x_2-y_2|/2$.

\item\label{full-4} the function $F_i$ above is given by
\begin{align*}
F_i(x_i, y_i) := F_{i, 1}(|x_i - y_i|) F_{i, 2}(|x_i - y_i|) F_{i, 3}(|x_i + y_i|),
\end{align*}
where $(F_{i, 1}, F_{i, 2}, F_{i, 3}) \in \mathscr{F}$, $i=1, 2$.
\end{list}

We say that a linear operator $T$ admits a \emph{full kernel representation} if both $F_1$ and $F_2$ in \eqref{full-1}--\eqref{full-3} above are replaced by a uniform constant $C \ge 1$.
\end{definition}

\begin{definition}\label{def:partial}
A linear operator $T$ admits a \emph{compact partial kernel representation on the first parameter} if the following hold.
If $f = f_1 \otimes f_2$ and $g = g_1 \otimes g_2$ with $\supp(f_1) \cap \supp(g_1) = \emptyset$, then
\begin{align*}
\langle Tf, g \rangle
= \int_{\R^{n_1}} \int_{\R^{n_1}} K_{f_2, g_2}(x_1, y_1) f_1(y_1) g_1(x_1) \, dx_1 \, dy_1,
\end{align*}
where the kernel $K_{f_2, g_2}: (\R^{n_1} \times \R^{n_1}) \setminus \big\{(x_1, y_1) \in \R^{n_1} \times \R^{n_1}: x_1 = y_1 \big\} \rightarrow \C$ satisfies
\begin{list}{\rm (\theenumi)}{\usecounter{enumi}\leftmargin=1.2cm \labelwidth=1cm \itemsep=0.2cm \topsep=.2cm \renewcommand{\theenumi}{\arabic{enumi}}}

\item[(i)]\label{partial-1} the size condition
\begin{align*}
|K_{f_2, g_2}(x_1,y_1)|
\leq C(f_2, g_2) \frac{F_1(x_1, y_1)}{|x_1-y_1|^{n_1}}.
\end{align*}

\item[(ii)]\label{partial-2} the H\"{o}lder conditions
\begin{align*}
|K_{f_2, g_2}(x_1, y_1) - K_{f_2, g_2}(x'_1, y_1)|
\leq C(f_2, g_2) \bigg(\frac{|x_1-x'_1|}{|x_1-y_1|}\bigg)^{\delta_1} \frac{F_1(x_1, y_1)}{|x_1 - y_1|^{n_1}}
\end{align*}
whenever $|x_1-x'_1| \leq |x_1-y_1|/2$, and
\begin{align*}
|K_{f_2, g_2}(x_1, y_1) - K_{f_2, g_2}(x_1, y'_1)|
\leq C(f_2, g_2) \bigg(\frac{|y_1-y'_1|}{|x_1-y_1|}\bigg)^{\delta_1} \frac{F_1(x_1, y_1)}{|x_1 - y_1|^{n_1}}
\end{align*}
whenever $|y_1-y'_1| \leq |x_1-y_1|/2$.

\item[(iii)]\label{partial-3} the function $F_1$ above is given by
\begin{align*}
F_1(x_1, y_1) := F_{1, 1}(|x_1 - y_1|) F_{1, 2}(|x_1 - y_1|) F_{1, 3}(|x_1 + y_1|),
\end{align*}
where $(F_{1, 1}, F_{1, 2}, F_{1, 3}) \in \mathscr{F}$.

\item[(iv)]\label{partial-4} the bound $C(f_2, g_2)$ above verifies
\begin{align*}
C(\mathbf{1}_{I_2}, \mathbf{1}_{I_2})
+ C(\mathbf{1}_{I_2}, a_{I_2})
+ C(a_{I_2}, \mathbf{1}_{I_2})
\le F_2(I_2) \, |I_2|,
\end{align*}
for all cubes $I_2 \subset \R^{n_2}$ and all functions $a_{I_2}$ satisfying $\supp(a_{I_2}) \subset I_2$, $|a_{I_2}| \le 1$, and $\int a_{I_2} =0$, where $F_2 \in \mathscr{F}_2$.
\end{list}

Analogously, one can define a \emph{compact partial kernel representation on the second parameter} when $\supp(f_2) \cap \supp(g_2) = \emptyset$.
\end{definition}

\begin{definition}\label{def:SIO}
let $T$ be a linear operator. 
\begin{itemize}
\item We say that $T$ admits the \emph{compact partial kernel representation} if it admits a compact partial kernel representation on both the first and second parameters.

\item We say that a linear operator $T$ admits the \emph{partial kernel representation} if both $F_1$ and $F_2$ in compact partial kernel representations are replaced by a uniform constant $C \ge 1$.

\item If $T$ admits the full and partial kernel representations, we call $T$ a \emph{bi-parameter singular integral operator}.
\end{itemize} 
\end{definition}

Next, let us given compactness and cancellation assumptions for bi-parameter singular integral operators.

\begin{definition}\label{def:WCP}
We say that $T$ satisfies the \emph{weak compactness property} if
\begin{align*}
|\langle T(\mathbf{1}_{I_1} \otimes \mathbf{1}_{I_2}), \mathbf{1}_{I_1} \otimes \mathbf{1}_{I_2} \rangle|
\leq F_1(I_1) |I_1| \, F_2(I_2) |I_2|,
\end{align*}
for all cubes $I_1 \subset \R^{n_1}$ and $I_2 \subset \R^{n_2}$, where $F_1 \in \mathscr{F}_1$ and $F_2 \in \mathscr{F}_2$. We say that $T$ satisfies the \emph{weak boundedness property} if both $F_1$ and $F_2$ above are replaced by a uniform constant $C \ge 1$.
\end{definition}

\begin{definition}\label{def:diag-CMO}
We say that $T$ satisfies the \emph{diagonal $\CMO$ condition} if for every $i=1, 2$, for every cube $I_i \subset \R^{n_i}$, and for every function $a_{I_i}$ such that $\supp(a_{I_i}) \subset I_i$, $|a_{I_i}| \le 1$, and $\int_{\R^{n_i}} a_{I_i} dx_i =0$, there holds
\begin{align*}
&|\langle T(\mathbf{1}_{I_1} \otimes \mathbf{1}_{I_2}), \mathbf{1}_{I_1} \otimes a_{I_2} \rangle|
+ |\langle T(\mathbf{1}_{I_1} \otimes \mathbf{1}_{I_2}), a_{I_1} \otimes \mathbf{1}_{I_2} \rangle|
\\ 
&\quad + |\langle T(a_{I_1} \otimes \mathbf{1}_{I_2}), \mathbf{1}_{I_1} \otimes \mathbf{1}_{I_2} \rangle|
+ |\langle T(\mathbf{1}_{I_1} \otimes a_{I_2}), \mathbf{1}_{I_1} \otimes \mathbf{1}_{I_2} \rangle| 
\\
&\leq F_1(I_1) |I_1| \, F_2(I_2) |I_2|,
\end{align*}
where $F_1 \in \mathscr{F}_1$ and $F_2 \in \mathscr{F}_2$. We say that $T$ satisfies the \emph{diagonal $\BMO$ condition} if both $F_1$ and $F_2$ above are replaced by a uniform constant $C \ge 1$.
\end{definition}

Finally, let us define the product $\BMO$ space. Let $\D_0$ be the standard dyadic grid on $\Rn$:
\[
\D_0 := \big\{2^{-k}([0, 1)^n + m): k \in \Z, \, m \in \Z^n \big\}.
\]
Let $\Omega :=(\{0, 1\}^n)^{\Z}$ and let $\mathbb{P}$ be the natural probability measure on $\Omega$: each component $\omega_j$ has an equal probability $2^{-n}$ of taking any of the $2^n$ values in $\{0, 1\}^n$, and all components are independent of each other. Given $\omega=(\omega_j)_{j \in \Z} \in \Omega$, the \emph{random dyadic grid} $\D_{\omega}$ on $\Rn$ is defined by
\begin{align*}
\D_{\omega} := \bigg\{Q+\omega := Q + \sum_{j: 2^{-j}< \ell(Q)} 2^{-j} \omega_j : Q \in \D_0\bigg\}.
\end{align*}
By a \emph{dyadic grid} $\D$ we mean that $\D=\D_{\omega}$ for some $\omega \in \Omega$.

Given a dyadic grid $\D_{n_i}$ on $\R^{n_i}$, $i=1, 2$, we denote the collection of dyadic rectangles by $\D=\D_{n_1} \times \D_{n_2}$. We say that a locally integrable function $b$ belongs to the \emph{dyadic product $\BMO$ space} $\BMO_{\D}$ if
\begin{align*}
\|b\|_{\BMO_{\D}}
:= \sup_{\Omega} \bigg(\frac{1}{|\Omega|} 
\sum_{\substack{I_1 \times I_2 \in \D \\ I_1 \times I_2 \subset \Omega}}
|\langle b, h_{I_1} \otimes h_{I_2} \rangle|^2\bigg)^{\frac12}
< \infty,
\end{align*}
where the supremum is taken over all sets $\Omega \subset \R^{n_1} \times \R^{n_2}$ such that $|\Omega|<\infty$ and for every $x \in \Omega$ there exists an $I_1 \times I_2 \in \D$ so that $x \in I_1 \times I_2 \subset \Omega$. The \emph{product $\BMO$ space} $\BMO(\R^{n_1} \times \R^{n_2})$ is defined as the collection of all locally integrable functions $b$ satisfying 
\begin{align*}
\|b\|_{\BMO(\R^{n_1} \times \R^{n_2})} 
:= \sup_{\D} \|b\|_{\BMO_{\D}} < \infty, 
\end{align*}
where the supremum over all dyadic grids $\D=\D_{n_1} \times \D_{n_2}$.

We say that $b$ belongs to the \emph{dyadic product $\CMO$ space $\CMO_{\D}$} if $b \in \BMO_{\D}$ and satisfies 
\begin{align*}
\lim_{N \to \infty} \|P_N^{\perp} b\|_{\BMO_{\D}} = 0,  
\end{align*}
where 
\begin{align*}
P_N^{\perp} b
:= \sum_{\substack{J_1 \not\in \D_{n_1}(N) \\ \text{or } J_2 \not\in \D_{n_2}(N)}}
\langle b, h_{J_1} \otimes h_{J_2} \rangle \, h_{J_1} \otimes h_{J_2}.
\end{align*}
The notation $\D_{n_1}(N)$ and $\D_{n_2}(N)$ are defined in Section \ref{sec:notation}. The \emph{product $\CMO$ space $\CMO(\R^{n_1} \times \R^{n_2})$} is defined as the collection of all functions $b \in \BMO(\R^{n_1} \times \R^{n_2})$ satisfying 
\begin{align*}
\lim_{N \to \infty} \sup_{\D} \|P_N^{\perp} b\|_{\BMO_{\D}} = 0. 
\end{align*}

Given a linear operator $T$ acting on functions defined on the product space $\R^{n_1} \times \R^{n_2}$, we define its \emph{adjoint} $T^*$ and its \emph{partial adjoints} $T^*_1$ and $T^*_2$ by
\begin{align*}
\langle T(f_1 \otimes f_2), g_1 \otimes g_2 \rangle
&=\langle T^*(g_1 \otimes g_2), f_1 \otimes f_2 \rangle
\\
&=\langle T^*_1(g_1 \otimes f_2), f_1 \otimes g_2 \rangle
=\langle T^*_2(f_1 \otimes g_2), g_1 \otimes f_2 \rangle.
\end{align*}

\begin{definition}\label{def:prod-BMO}
A linear operator $T$ satisfies the \emph{product $\BMO$ condition} if
\begin{align*}
b \in \BMO(\R^{n_1} \times \R^{n_2}) \quad\text{ for all } \, b \in \{T1, T^*1, T^*_1 1, T^*_2 1\}.
\end{align*}
We say that $T$ satisfies the \emph{product $\CMO$ condition} if
\begin{align*}
b \in \CMO(\R^{n_1} \times \R^{n_2}) \quad\text{ for all } \, b \in \{T1, T^*1, T^*_1 1, T^*_2 1\}.
\end{align*}
\end{definition}

This paper is organized as follows. Section \ref{sec:pre} contains some preliminaries, definitions, and technical results to show our main  theorem. Section \ref{sec:reduction} is devoted to presenting crucial reductions, while auxiliary results are presented in Section \ref{sec:aux}. Finally, in Section \ref{sec:proof}, we give main estimates for matrix elements involved in the expansion of $\langle Tf, g \rangle$.

\section{Preliminaries}\label{sec:pre}

\subsection{Notation}\label{sec:notation}
\begin{itemize}
\item For convenience, we use $\ell^{\infty}$ metrics on $\R^{n_1}$ and $\R^{n_2}$.

\item Given a cube $I \subset \R^{n_i}$, $i=1, 2$, we denote its center by $c_I$ and its edge length by $\ell(I)$. For any $\lambda>0$, we denote by $\lambda I$ the cube with the center $c_I$ and edge length $\lambda \ell(I)$. The measure of $I$ is simply denoted by $|I|$ no matter what dimension we are in.

\item Write $\I_i=[-\frac12, \frac12]^{n_i}$ and $\lambda \I_i=[-\frac{\lambda}{2}, \frac{\lambda}{2}]^{n_i}$, $i=1, 2$.

\item Let $\D_{n_i}$ denote a general dyadic grid in $\R^{n_i}$, $i=1, 2$.

\item For every $i=1, 2$ and $I \in \D_{n_i}$, set $\ch(I):=\{I' \in \D_{n_i}: I' \subset I, \ell(I')=\ell(I)/2\}$ and let $I^{(k)}$ denote the unique dyadic cube $J \in \D_{n_i}$ so that $I \subset J$ and $\ell(J)=2^k \ell(I)$.

\item Given cubes $I, J \subset \R^{n_i}$, $i=1, 2$, we denote by $I \Cup J$ the cube with minimal edge length containing $I \cup J$.  If there is more than one cube satisfying this condition, we will simply select one.

\item Given cubes $I, J \subset \R^{n_i}$, $i=1, 2$, we respectively define the relative size and the relative distance between them by
\begin{align*}
\rs(I, J) := \frac{\min\{\ell(I), \, \ell(J)\}}{\max\{\ell(I), \, \ell(J)\}}
\quad\text{ and }\quad
\rd(I, J) := \frac{\d(I, J)}{\max\{\ell(I), \, \ell(J)\}}.
\end{align*}
Moreover, the inner relative distance between $I$ and $J$ is defined by
\begin{align*}
\ird(I, J) :=1+ \frac{\d(I, J)}{\min\{\ell(I), \, \ell(J)\}}. 
\end{align*}

\item For each $i=1, 2$ and $N \in \N_+$, we define $\D_{n_i}(N)$ as the family of all $I \in \D_{n_i}$ with $2^{-N} \leq  \ell(I) \leq 2^N$ and $\rd(I, 2^N \I_i) \leq N$.

\item For each $i=1, 2$, $I \not\in \D_{n_i}(N)$ means $I \in \D_{n_i} \setminus \D_{n_i}(N)$.

\item Given a cube $I \subset \R^{n_i}$ and $f \in L^1_{\loc}(\R^{n_i})$, $i=1, 2$, we write $\langle f \rangle_I := \frac{1}{|I|} \int_I f \, dx$.

\item We shall use $a\lesssim b$ and $a \simeq b$ to mean, respectively, that $a\leq C b$ and $0<c\leq a/b\leq C$, where the constants $c$ and $C$ are harmless positive constants, not necessarily the same at each occurrence, which depend only on dimension and the constants appearing in the hypotheses of theorems.
\end{itemize}

\subsection{Haar functions}
Let $h_I$ be an $L^2(\Rn)$ normalized \emph{Haar function} related to $I \in \mathcal{D}_n$, where $\mathcal{D}_n$ is a dyadic grid on $\Rn$. This means the following. Given $I = I_1 \times \cdots \times I_n$, $h_I$ is one of the $2^n$ functions $h_I^{\eta}$, $\eta=(\eta_1, \ldots, \eta_n) \in \{0, 1\}^n$, defined by
\begin{align*}
h_I^{\eta} = h_{I_1}^{\eta_1} \otimes \cdots \otimes h_{I_n}^{\eta_n},
\end{align*}
where $h_{I_i}^0 = |I_i|^{-\frac12} \mathbf{1}_{I_i}$ and $h_{I_i}^1 = |I_i|^{-\frac12}(\mathbf{1}_{I_i^-} - \mathbf{1}_{I_i^+})$ for every $i = 1, \ldots, n$. Here $I_i^-$ and $I_i^+$ are the left and right halves of the interval $I_i$ respectively. Note that for any $\eta \neq 0$, the Haar function is cancellative : $\int_{\Rn} h_I^{\eta} \, dx = 0$. We usually suppress the presence of $\eta$ and simply write $h_I$ for some $h_I^{\eta}$, $\eta \neq 0$.

All the cancellative Haar functions on $\Rn$ form an orthonormal basis of $L^2(\Rn)$. Thus, for each $g \in L^2(\Rn)$, we may write
\begin{align*}
g = \sum_{I \in \D_n} \sum_{\eta \in \{0, 1 \}^n \setminus \{0\}} \langle g, h_I^\eta \rangle h_I^\eta.
\end{align*}
However, we suppress the finite $\eta$ summation and just write $g = \sum_{I \in \D_n} \langle g, h_I \rangle h_I$. We may expand a function $f \in L^2(\R^{n_1} \times \R^{n_2})$ using the corresponding product basis:
\begin{align}\label{expand}
f = \sum_{\substack{I_1 \in \D_{n_1} \\ I_2 \in \D_{n_2}}}
\langle f, h_{I_1} \otimes h_{I_2} \rangle \, h_{I_1} \otimes h_{I_2}.
\end{align}

\subsection{Weights}
A measurable function $w$ on $\R^{n_1} \times \R^{n_2}$ is called a \emph{weight} if it is a.e. positive and locally integrable. Given $p \in (1, \infty)$, we say that a weight $w$ on $\R^{n_1} \times \R^{n_2}$ belongs to the \emph{bi-parameter weight class} $A_p(\R^{n_1} \times \R^{n_2})$ if
\begin{align*}
[w]_{A_p(\R^{n_1} \times \R^{n_2})}
:= \sup_{R} \bigg(\fint_R w \, dx \bigg) \bigg(\fint_R w^{-\frac{1}{p-1}} dx \bigg)^{p-1} < \infty,
\end{align*}
where the supremum is taken over rectangles $R=I_1 \times I_2$ with $I_i \subset \R^{n_i}$ being a cube, $i=1, 2$.

The basic bi-parameter weighted theory can be found in \cite[Chapter IV]{GR}. Let us present an interpolation theorem for compact operators below, which is a particular case of \cite[Theorem 9]{CK}.

\begin{theorem}\label{thm:inter}
Let $1<p_i<\infty$ and $w_i \in A_{p_i}(\R^{n_1} \times \R^{n_2})$, $i=0, 1$. Assume that $T$ is a linear operator such that $T$ is bounded on $L^{p_0}(w_0)$ and compact on $L^{p_1}(w_1)$. Then $T$ is compact on $L^p(w)$, where $\frac1p = \frac{1-\theta}{p_0} + \frac{\theta}{p_1}$, $w^{\frac1p} = w_0^{\frac{1-\theta}{p_0}}  w_1^{\frac{\theta}{p_1}}$, and $0<\theta<1$.
\end{theorem}

\subsection{Technical lemmas} 

We next give two useful lemmas.
\begin{lemma}\label{lem:improve}
Given $(F_{i, 1}, F_{i, 2}, F_{i, 3}) \in \F$, let
\begin{align*}
F_i(x_i, y_i) := F_{i, 1}(|x_i - y_i|) F_{i, 2}(|x_i - y_i|) F_{i, 3}(|x_i + y_i|),  \quad i=1, 2.
\end{align*}
\begin{list}{\rm (\theenumi)}{\usecounter{enumi}\leftmargin=1.2cm \labelwidth=1cm \itemsep=0.2cm \topsep=.2cm \renewcommand{\theenumi}{\roman{enumi}}}

\item\label{size-imp} For each $i=1, 2$, there exists a triple $(F'_{i, 1}, F'_{i, 2}, F'_{i, 3}) \in \F$ such that $F'_{i, 1}$ is monotone increasing, $F'_{i, 2}$ and $F'_{i, 3}$ are monotone decreasing, and
\begin{align}\label{FF-imp-1}
F_i(x_i, y_i)
\le F'_i(x_i, y_i)
:= F'_{i, 1}(|x_i-y_i|) F'_{i, 2}(|x_i-y_i|) F'_{i, 3} \bigg(1 + \frac{|x_i+y_i|}{1+|x_i-y_i|}\bigg).
\end{align}

\item\label{Holder-imp}
Assume that $F_{1, 1}$ and $F_{2, 1}$ are monotone increasing. Then for each $i=1, 2$, there exist $\delta'_i \in (0, \delta_i)$ and $(F'_{i, 1}, F'_{i, 2}, F'_{i, 3}) \in \F$ such that $F'_{i, 1}$ is monotone increasing, $F'_{i, 2}$ and $F'_{i, 3}$ are monotone decreasing, and
\begin{align}\label{FF-imp-2}
F_i(x_i, y_i) \frac{|y_i-y'_i|^{\delta_i}}{|x_i-y_i|^{n_i+\delta_i}}
\le F'_i(x_i, y_i) \frac{|y_i-y'_i|^{\delta'_i}}{|x_i-y_i|^{n_i+\delta'_i}},
\end{align}
whenever $|y_i-y'_i| \leq |x_i-y_i|/2$, where
\begin{align}\label{def:Fxy}
F'_i(x_i, y_i)
:= F'_{i, 1}(|y_i-y'_i|) F'_{i, 2}(|x_i-y_i|) F'_{i, 3} \bigg(1 + \frac{|x_i+y_i|}{1+|x_i-y_i|}\bigg).
\end{align}
\end{list}
\end{lemma}

\begin{proof}
By the property of $(F_{i, 1}, F_{i, 2}, F_{i, 3}) \in \F$, we have
\begin{align}\label{imp-1}
\lim_{|x_i-y_i| \to 0} F_i(x_i, y_i)
= \lim_{|x_i-y_i| \to \infty} F_i(x_i, y_i)
= \lim_{|x_i+y_i| \to \infty} F_i(x_i, y_i)
=0.
\end{align}
Note that $|x_i+y_i| \ge \frac{|x_i+y_i|}{1+|x_i-y_i|}$. This along with \eqref{imp-1} gives
\begin{align}\label{imp-2}
\lim_{|x_i-y_i| \to 0} F_i(x_i, y_i)
= \lim_{|x_i-y_i| \to \infty} F_i(x_i, y_i)
= \lim_{1+\frac{|x_i+y_i|}{1+|x_i-y_i|} \to \infty} F_i(x_i, y_i)
=0.
\end{align}
Define
\begin{align*}
F'_{i, 1}(t) &:= \sup_{|x_i-y_i| \le t} F_i(x_i, y_i)^{\frac13}, 
\\
F'_{i, 2}(t) &:= \sup_{|x_i-y_i| \ge t} F_i(x_i, y_i)^{\frac13},
\\
F'_{i, 3}(t) &:= \sup_{1+\frac{|x_i+y_i|}{1+|x_i-y_i|} \ge t} F_i(x_i, y_i)^{\frac13}.
\end{align*}
Then it is easy to see that $F'_{i, 1}$ is monotone increasing while $F'_{i, 2}$ and $F'_{i, 3}$ are monotone decreasing. Moreover, it follows from \eqref{imp-2} that $(F'_{i, 1}, F'_{i, 2}, F'_{i, 3}) \in \F$,
\begin{align*}
F'_{i, j}(|x_i-y_i|) \ge F_i(x_i, y_i)^{\frac13}, \ j=1,2, \text{ and }
F'_{i, 3} \bigg(1+\frac{|x_i+y_i|}{1+|x_i-y_i|} \bigg) \ge F_i(x_i, y_i)^{\frac13},
\end{align*}
which immediately imply \eqref{FF-imp-1}.

To show part \eqref{Holder-imp}, we choose $\delta'_i \in (0, \delta_i)$ and write $\varepsilon_i := \delta_i - \delta'_i$ and
\begin{align}\label{def:Fxyy}
F_i(x_i, y_i, y'_i)
:= F_i(x_i, y_i) \frac{|y_i-y'_i|^{\varepsilon_i}}{|x_i-y_i|^{\varepsilon_i}}
\quad\text{ for all } \, |y_i-y'_i| \leq |x_i-y_i|/2.
\end{align}
We claim that
\begin{align}\label{imp-3}
\lim_{|y_i-y'_i| \to 0} F_i(x_i, y_i, y'_i)
= \lim_{|x_i-y_i| \to \infty} F_i(x_i, y_i, y'_i)
= \lim_{1+\frac{|x_i+y_i|}{1+|x_i-y_i|} \to \infty} F_i(x_i, y_i, y'_i)
=0.
\end{align}
Indeed, the last two limits follow from \eqref{imp-2}. To see the first one, we assume that there exists a sequence $\{(x_k, y_k, y'_k)\}_{k \in \N}$ such that $|y_k-y'_k|<\frac12 |x_k-y_k|$ with $\lim_{k \to \infty} |y_k-y'_k|=0$, but
\begin{align}\label{infF}
\inf_k F_i(x_k, y_k, y'_k)>0.
\end{align}
If $\liminf_{k \to \infty} \frac{|y_k-y'_k|}{|x_k-y_k|}=0$, then by definition, $\liminf_{k \to \infty} F_i(x_k, y_k, y'_k)=0$, which contradicts \eqref{infF}. If $\liminf_{k \to \infty} \frac{|y_k-y'_k|}{|x_k-y_k|} > 0$, then there exist a constant $C_0>0$ and a subsequence (which we relabel) $\{(x_k, y_k, y'_k)\}_{k \in \N}$ such that $|y_k-y'_k| \ge C_0 |x_k-y_k|$. Since $F_{i, 1}$ is monotone increasing and $\lim_{k \to \infty} |y_k-y'_k|=0$, one has
\begin{align*}
0 \le F_i(x_k, y_k, y'_k)
\le F_i(x_k, y_k)
\lesssim F_{i, 1}(|x_k-y_k|) \le F_{i, 1}(C_0^{-1} |y_k-y'_k|) \to 0,
\end{align*}
as $k \to \infty$. Thus, $\lim_{k \to \infty} F_i(x_k, y_k, y'_k) =0$, which contradicts \eqref{infF}. Therefore, the first limit in \eqref{imp-3} holds.

Now we define
\begin{align*}
F'_{i, 1}(t) &:= \sup_{|y_i-y'_i| \le t} F_i(x_i, y_i, y'_i)^{\frac13},
\\ 
F'_{i, 2}(t) &:= \sup_{|x_i-y_i| \ge t} F_i(x_i, y_i, y'_i)^{\frac13},
\\
F'_{i, 3}(t) &:= \sup_{1+\frac{|x_i+y_i|}{1+|x_i-y_i|} \ge t} F_i(x_i, y_i, y'_i)^{\frac13}.
\end{align*}
Then one can check that $F'_{i, 1}$ is monotone increasing while $F'_{i, 2}$ and $F'_{i, 3}$ are monotone decreasing. By \eqref{imp-3}, we see that $(F'_{i, 1}, F'_{i, 2}, F'_{i, 3}) \in \F$,
\begin{align*}
F'_{i, 1}(|y_i-y'_i|) \ge F_i(x_i, y_i, y'_i)^{\frac13}, \qquad 
F'_{i, 2}(|x_i-y_i|) \ge F_i(x_i, y_i, y'_i)^{\frac13}, 
\end{align*}
and 
\begin{align*}
F'_{i, 3} \bigg(1+\frac{|x_i+y_i|}{1+|x_i-y_i|} \bigg) \ge F_i(x_i, y_i, y'_i)^{\frac13}.
\end{align*}
As a consequence, this, \eqref{def:Fxy}, and \eqref{def:Fxyy} imply that
\begin{align*}
F_i(x_i, y_i) \frac{|y_i-y'_i|^{\delta_i}}{|x_i-y_i|^{n_i+\delta_i}}
= F_i(x_i, y_i, y'_i) \frac{|y_i-y'_i|^{\delta'_i}}{|x_i-y_i|^{n_i+\delta'_i}}
\le F'_i(x_i, y_i) \frac{|y_i-y'_i|^{\delta'_i}}{|x_i-y_i|^{n_i+\delta'_i}},
\end{align*}
which shows \eqref{FF-imp-2} and hence completes the proof of the lemma.
\end{proof}

\begin{lemma}\label{lem:diag}
Fix $i=1, 2$. Let $r>1$ and $1<s<1+\frac{1}{n_i}$. Assume that a bounded and decreasing function $F_{i, 2}: [0, \infty) \to [0, \infty)$ satisfies  $\lim_{t \to \infty} F_{i, 2}(t)=0$. Then there exists a bounded and decreasing function $\widetilde{F}_{i, 2}:[0, \infty) \to [0, \infty)$ satisfying $\lim_{t \to \infty} \widetilde{F}_{i, 2}(t)=0$ such that for all cubes $I_i, J_i \in \R^{n_i}$ with $\ell(I_i) = \ell(J_i)$, $I_i \cap J_i = \emptyset$, and $\d(I_i, J_i)=0$,
\begin{align}\label{eq:diag-1}
\mathscr{I}_1
&:= \bigg(\fint_{I_i} \fint_{J_i} F_{i, 2}(|x_i-y_i|)^r \, dx_i \, dy_i \bigg)^{\frac1r}
\lesssim \widetilde{F}_{i, 2}(\ell(I_i)),
\end{align}
and
\begin{align}\label{eq:diag-2}
\mathscr{I}_2
&:= \bigg(\fint_{I_i} \fint_{J_i} \frac{1}{|x_i-y_i|^{s n_i}} \, dx_i \, dy_i  \bigg)^{\frac1s}
\lesssim |I_i|^{-1}.
\end{align}
\end{lemma}

\begin{proof}
Given $t>0$ and $x_i \in J_i$, we let
\begin{align*}
J_i(t) := \{x_i \in J_i: \d(x_i, I_i) = t\}
\quad\text{ and }\quad
I_i(x_i, t) := \{y_i \in I_i: |y_i-x_i|=t \}.
\end{align*}
Let $\sigma_i$ be the Lebesgue measure in $\R^{n_i-1}$. Observe that
\begin{align}
\label{surface-1} & \sigma_i(J_i(t)) \le \ell(J_i)^{n_i-1} \quad\text{if}\quad 0<t < \ell(J_i),
\\
\label{surface-2} &\sigma_i(I_i(x_i, t)) \lesssim t^{n_i-1} \quad\text{if}\quad \d(x_i, I_i) < t < \d(x_i, I_i) + \ell(I_i).
\end{align}
Since $F_{i, 2}$ is monotone decreasing, we use \eqref{surface-1} to obtain
\begin{align*}
\mathscr{I}_1^r
&\le \fint_{J_i} F_{i, 2}(\d(x_i, I_i))^r \, dx_i
=\frac{1}{|J_i|} \int_0^{\ell(J_i)} \int_{J_i(t)} F_{i, 2}(t)^r \, d\sigma_i \, dt
\\
&\lesssim \frac{\ell(J_i)^{n_i-1}}{|J_i|} \int_0^{\ell(J_i)} F_{i, 2}(t)^r  \, dt
=\sum_{k=0}^{\infty} \ell(J_i)^{-1} \int_{2^{-k-1} \ell(J_i)}^{2^{-k} \ell(J_i)} F_{i, 2}(t)^r \, dt
\\
&\lesssim \sum_{k=0}^{\infty} 2^{-k} F_{i, 2}(2^{-k} \ell(J_i))^r
\le \bigg(\sum_{k=0}^{\infty} 2^{-k/r} F_{i, 2}(2^{-k} \ell(J_i)) \bigg)^r.
\end{align*}
Choosing
\begin{align*}
\widetilde{F}_{i, 2}(t)
:= \sum_{k=0}^{\infty} 2^{-k/r} F_{i, 2}(2^{-k}t),
\end{align*}
we see that $\widetilde{F}_{i, 2}$ is bounded and decreasing, and satisfies $\lim_{t \to \infty} \widetilde{F}_{i, 2}(t)=0$. Thus, \eqref{eq:diag-1} holds. Similarly, it follows from \eqref{surface-2} that
\begin{align*}
\mathscr{I}_2^s &= \frac{1}{|I_i|}
\fint_{J_i} \bigg(\int_{\d(x_i, I_i)}^{\d(x_i, I_i) + \ell(I_i)} \int_{I_i(x_i, t)} t^{-sn_i} \, d\sigma_i \, dt \bigg) dx_i
\\
& \lesssim \frac{1}{|I_i|} \fint_{J_i} \bigg(\int_{\d(x_i, I_i)}^{\infty} t^{-n_i(s-1)} \frac{dt}{t} \bigg) dx_i
\\
&\lesssim \frac{1}{|I_i|} \fint_{J_i} \d(x_i, I_i)^{-n_i(s-1)} \, dx_i
\\
&= \frac{1}{|I_i| |J_i|} \int_0^{\ell(J_i)} \int_{J_i(t)} t^{-n_i(s-1)} d\sigma_i \, dt
\\
&\lesssim |I_i|^{-1} \ell(J_i)^{-1} \int_0^{\ell(J_i)} t^{-n_i(s-1)} dt
\\
&\lesssim |I_i|^{-1} \ell(J_i)^{-n_i(s-1)}
=|I_i|^{-1} |J_i|^{1-s}
=|I_i|^{-s}.
\end{align*}
This shows \eqref{eq:diag-2}.
\end{proof}

\section{Reductions of the proof}\label{sec:reduction}
Let $T$ be a bi-parameter singular integral operator satisfying the hypotheses \eqref{list-1}--\eqref{list-5}. Obviously, $T$ satisfies the weak boundedness property, the diagonal $\BMO$ condition, and the product $\BMO$ condition. Thus, \cite[Theorem 2.4]{Mar} holds, which together with \cite[Proposition 7.6]{HPW} implies that
\begin{align}\label{T-Lp}
\text{$T$ is bounded on $L^r(v)$ for all $r \in (1, \infty)$ and $v \in A_r(\R^{n_1} \times \R^{n_2})$}.
\end{align}
We are going to prove that
\begin{align}\label{reduction-1}
\text{$T$ is compact on $L^2(\R^{n_1} \times \R^{n_2})$.}
\end{align}
Assuming \eqref{reduction-1} momentarily, we can conclude the proof of Theorem \ref{thm:compact} as follows. Let $p \in (1, \infty)$, $w \in A_p(\R^{n_1} \times \R^{n_2})$, $p_1=2$, and $w_1 \equiv 1$. We claim that there exist $\theta \in (0, 1)$, $p_0=p_0(\theta) \in (1, \infty)$, and $w_0=w_0(\theta) \in A_{p_0}(\R^{n_1} \times \R^{n_2})$ such that
\begin{align}\label{wpp}
\frac1p=\frac{1-\theta}{p_0} + \frac{\theta}{p_1}
\quad\text{ and }\quad
w^{\frac1p} = w_0^{\frac{1-\theta}{p_0}} w_1^{\frac{\theta}{p_1}}.
\end{align}
Indeed, \eqref{wpp} can be shown as in the proof of \cite[Lemma 4.4]{HL} (or \cite[Lemma 4.1]{COY} for the multilinear generalization). Although the latter is given for cubes, in view of the reverse H\"{o}lder  inequality for $A_p(\R^{n_1} \times \R^{n_2})$ weights from \cite[p. 458]{GR}, the argument still holds for rectangles. The details are left to the reader.

Observe that \eqref{T-Lp} applied to $w_0 \in A_{p_0}(\R^{n_1} \times \R^{n_2})$ gives that
\begin{align}\label{TL-1}
\text{$T$ is bounded on $L^{p_0}(w_0)$}.
\end{align}
Moreover, \eqref{reduction-1} can be rewritten as
\begin{align}\label{TL-2}
\text{$T$ is compact on $L^{p_1}(w_1)$.}
\end{align}
Consequently,  it follows from Theorem \ref{thm:inter} and \eqref{wpp}--\eqref{TL-2} that $T$ is compact on $L^p(w)$. This shows Theorem \ref{thm:compact}.

It remains to show \eqref{reduction-1}. For convenience, set $\D_i := \D_{n_i}$, $i=1, 2$. Given $N \in \N_+$, we define the \emph{projection operator} as
\begin{align*}
P_N f
:= \sum_{\substack{J_1 \in \D_1(N) \\ J_2 \in \D_2(N)}}
\langle f, h_{J_1} \otimes h_{J_2} \rangle \, h_{J_1} \otimes h_{J_2},
\end{align*}
and its \emph{orthogonal operator} as
\begin{align*}
P_N^{\perp} f
:= \sum_{\substack{J_1 \not\in \D_1(N) \\ \text{or } J_2 \not\in \D_2(N)}}
\langle f, h_{J_1} \otimes h_{J_2} \rangle \, h_{J_1} \otimes h_{J_2}.
\end{align*}
Note that $P_N$ is a finite-dimensional operator, which along with \cite[Proposition 15.1]{FHHMZ} implies that $P_N$ is a compact operator. In view of the fact that $P_N^{\perp} \circ T = T - P_N \circ T$, the estimate \eqref{reduction-1} is reduced to showing that
\begin{align}\label{reduction-2}
\lim_{N \to \infty} \|P_N^{\perp} \circ T\|_{L^2 \to L^2} =0.
\end{align}

Let $f, g \in L^2(\R^{n_1} \times \R^{n_2})$, By definition and \eqref{expand}, we rewrite
\begin{align*}
\langle P_{2N}^{\perp} (Tf), g \rangle
=\sum_{\substack{I_1 \in \D_1 \\ I_2 \in \D_2}} 
\sum_{\substack{J_1 \not\in \D_1(2N) \\ \text{or } J_2 \not\in \D_2(2N)}}
\mathscr{G}_{J_1 J_2}^{I_1 I_2} \, f_{I_1 I_2} \, g_{J_1 J_2},
\end{align*}
where
\begin{align*}
\mathscr{G}_{J_1 J_2}^{I_1 I_2} := \langle T (h_{I_1} \otimes h_{I_2}), h_{J_1} \otimes h_{J_2} \rangle, \
f_{I_1 I_2} := \langle f, h_{I_1} \otimes h_{I_2} \rangle, \text{\ and\ }
g_{J_1 J_2} := \langle g, h_{J_1} \otimes h_{J_2}  \rangle.
\end{align*}
In what follows, we only focus on the sum where $\ell(I_1) \le \ell(J_1)$ and $\ell(I_2) \le \ell(J_2)$ because other three cases can be shown with minor modifications. For $i=1, 2$, according to relative distance between $I_i$ and $J_i$, we split $\ell(I_i) \le \ell(J_i)$ into the following three cases:
\begin{itemize}

\item $\ell(I_i) \le \ell(J_i)$ and $\rd(I_i, J_i) \ge 1$;

\item $\ell(I_i) \le \ell(J_i)$, $\rd(I_i, J_i) < 1$, and $I_i \cap J_i = \emptyset$;

\item $I_i \subset J_i$.

\end{itemize}
Considering symmetry,  to obtain \eqref{reduction-2}, it suffices to demonstrate that given $\varepsilon>0$, there exists an $N_0=N_0(\varepsilon) >1$ such that
\begin{align}\label{reduction-3}
\max_{1 \le j \le 6} |\mathscr{I}_j^N|
\lesssim \varepsilon \|f\|_{L^2} \|g\|_{L^2} \, \,
\quad\text{ for all } N \ge N_0,
\end{align}
where the implicit constant is independent of $\varepsilon$, $N$, $f$, and $g$, 
\begin{align*}
\mathscr{I}_1^N
&:= \sum_{\substack{J_1 \not\in \D_1(2N) \\ \text{or } J_2 \not\in \D_2(2N)}}
\sum_{\substack{I_1 \in \D_1 \\ \ell(I_1) \le \ell(J_1) \\ \rd(I_1, J_1) \ge 1}}
\sum_{\substack{I_2 \in \D_2 \\ \ell(I_2) \le \ell(J_2) \\ \rd(I_2, J_2) \ge 1}}
\G_{J_1 J_2}^{I_1 I_2} \, f_{I_1 I_2} \, g_{J_1 J_2},
\\
\mathscr{I}_2^N
&:= \sum_{\substack{J_1 \not\in \D_1(2N) \\ \text{or } J_2 \not\in \D_2(2N)}}
\sum_{\substack{I_1 \in \D_1 \\ \ell(I_1) \le \ell(J_1) \\ \rd(I_1, J_1) \ge 1}}
\sum_{\substack{I_2 \in \D_2 \\ \ell(I_2) \le \ell(J_2) \\ \rd(I_2, J_2) < 1  \\ I_2 \cap J_2 = \tinyemptyset}}
\G_{J_1 J_2}^{I_1 I_2} \, f_{I_1 I_2} \, g_{J_1 J_2},
\\
\mathscr{I}_3^N
&:= \sum_{\substack{J_1 \not\in \D_1(2N) \\ \text{or } J_2 \not\in \D_2(2N)}}
\sum_{\substack{I_1 \in \D_1 \\ \ell(I_1) \le \ell(J_1) \\ \rd(I_1, J_1) \ge 1}}
\sum_{\substack{I_2 \in \D_2 \\ I_2 \subset J_2}}
\G_{J_1 J_2}^{I_1 I_2} \, f_{I_1 I_2} \, g_{J_1 J_2},
\\ 
\mathscr{I}_4^N
&:= \sum_{\substack{J_1 \not\in \D_1(2N) \\ \text{or } J_2 \not\in \D_2(2N)}}
\sum_{\substack{I_1 \in \D_1 \\ \ell(I_1) \le \ell(J_1) \\ \rd(I_1, J_1) < 1  \\ I_1 \cap J_1 = \tinyemptyset}}
\sum_{\substack{I_2 \in \D_2 \\ \ell(I_2) \le \ell(J_2) \\ \rd(I_2, J_2) < 1  \\ I_2 \cap J_2 = \tinyemptyset}}
\G_{J_1 J_2}^{I_1 I_2} \, f_{I_1 I_2} \, g_{J_1 J_2},
\\
\mathscr{I}_5^N
&:= \sum_{\substack{J_1 \not\in \D_1(2N) \\ \text{or } J_2 \not\in \D_2(2N)}}
\sum_{\substack{I_1 \in \D_1 \\ \ell(I_1) \le \ell(J_1) \\ \rd(I_1, J_1) < 1  \\ I_1 \cap J_1 = \tinyemptyset}}
\sum_{\substack{I_2 \in \D_2 \\ I_2 \subset J_2}}
\G_{J_1 J_2}^{I_1 I_2} \, f_{I_1 I_2} \, g_{J_1 J_2},
\end{align*} 
and
\begin{align*}
\mathscr{I}_6^N
:= \sum_{\substack{J_1 \not\in \D_1(2N) \\ \text{or } J_2 \not\in \D_2(2N)}}
\sum_{\substack{I_1 \in \D_1 \\ I_1 \subset J_1}}
\sum_{\substack{I_2 \in \D_2 \\ I_2 \subset J_2}}
\G_{J_1 J_2}^{I_1 I_2} \, f_{I_1 I_2} \, g_{J_1 J_2}.
\end{align*}

\section{Auxiliary results}\label{sec:aux}
The goal of this section is to present some auxiliary results in order to demonstrate \eqref{reduction-3}. 
In what follows, we always assume that $I_i, J_i \in \D_i := \D_{n_i}$ with $\ell(I_i) \le \ell(J_i)$, $i=1, 2$. For each $i=1, 2$, let $F_i \in \F_i$ and $(F_{i, 1}, F_{i, 2}, F_{i, 3}) \in \F$ satisfy that $F_{i, 1}$ is increasing,  $F_{i, 2}$ and $F_{i, 3}$ are decreasing. Define 
\begin{align}
F_i(I_i, J_i) 
&:= F_{i, 1}(\ell(I_i)) F_{i, 2}(\ell(I_i \Cup J_i)) F_{i, 3}(\rd(I_i \Cup J_i, \I_i)), 
\\
\widetilde{F}_i(I_i, J_i)
&:= F_{i, 1}(\ell(I_i)) \widetilde{F}_{i, 2}(\ell(I_i)) \widetilde{F}_{i, 3}(I_i \Cup J_i), 
\\
\widetilde{F}_i(I_i)
&:= F_{i, 1}(\ell(I_i)) \widetilde{F}_{i, 2}(\ell(I_i)) \widetilde{F}_{i, 3}(I_i), 
\\
\widehat{F}_i(I_i) 
&:= \widetilde{F}_i(I_i) + F_i(I_i) + \sum_{I'_i \in \ch(I_i)} F_i(I'_i),  
\end{align}
where 
\begin{align*}
\widetilde{F}_{i, 2}(t)
:= \sum_{k=0}^{\infty} 2^{-k \theta_i} F_{i, 2}(2^{-k} t)
\quad\text{ and }\quad 
\widetilde{F}_{i, 3}(I_i)
:= \sum_{k=0}^{\infty} 2^{-k \delta_i} F_{i, 3}(\rd(2^k I_i, \I_i)).
\end{align*}
The auxiliary parameters $\theta_1, \theta_2 \in (0, 1)$ are harmless and small enough. 

Since any dilation of functions in $\F$, $\F_1$, and $\F_2$ still belongs to the original space, we will often omit all universal constants appearing in the argument involving these functions.

\subsection{Limiting results} 
We first give some estimates for quantities above.  

\begin{lemma}\label{lem:FE}
For $i=1, 2$, define $F_i(I_i, J_i)$, $\widetilde{F}_i(I_i, J_i)$, $\widetilde{F}_i(I_i)$, and $\widehat{F}_i(I_i)$ as above. Then the following statements hold: 
\begin{list}{\rm (\theenumi)}{\usecounter{enumi}\leftmargin=1.2cm \labelwidth=1cm \itemsep=0.2cm \topsep=.2cm \renewcommand{\theenumi}{\alph{enumi}}}

\item\label{list:FE-1} For any $\varepsilon>0$, there exists an $N_0 \ge 1$ so that for all $N \ge N_0$, $I_i \in \D_i$, and $J_i \not\in \D_i(2N)$, there holds $F_i(I_i, J_i) < \varepsilon$ or $\rd(I_i, J_i) \ge 2N^{1/8}$.

\item\label{list:FE-2} For any $\varepsilon>0$, there exists an $N_0 \ge 1$ such that for all $N \ge N_0$, $I_i \in \D_i$, and $J_i \not\in \D_i(2N)$ satisfying $\rd(I_i, J_i)<1$, there holds $\widetilde{F}_i(I_i, J_i) < \varepsilon$ or $\rs(I_i, J_i) \le 2^{-N}$. 

\item\label{list:FE-3} If $I_i \in \D_i(N)$ and $J_i \not\in \D_i(2N)$ with $\ell(J_i)=2^{k_i} \ell(I_i)$, then $|k_i| \ge N$ or $\rd(I_i, J_i) \ge 2N$.

\item\label{list:FE-4} One has 
\begin{align*}
\lim_{N \to \infty} \sup_{I_i \not\in \D_i(N)} \widetilde{F}_i(I_i)
= \lim_{N \to \infty} \sup_{I_i \not\in \D_i(N)} \widehat{F}_i(I_i) 
= 0. 
\end{align*}
\end{list}
\end{lemma}

\begin{proof}
We begin with showing part \eqref{list:FE-1}. Let $I_i \in \D_i$ and $J_i \not\in \D_i(2N)$. If $\ell(J_i)<2^{-2N}$, then
\begin{align}\label{FFJ-1}
F_i(I_i, J_i)
\lesssim F_{i1}(\ell(I_i))
\le F_{i, 1}(\ell(J_i))
\le F_{i, 1}(2^{-2N}).
\end{align}
If $\ell(J_i)>2^{2N}$, then
\begin{align}\label{FFJ-2}
F_i(I_i, J_i)
\lesssim F_{i2}(\ell(I_i \Cup J_i))
\le F_{i, 2}(\ell(J_i))
\le F_{i, 2}(2^{2N}).
\end{align}
If $2^{-2N} \le \ell(J_i) \le 2^{2N}$ and $\rd(J_i, 2^{2N} \I_i) > 2N$, then $|c_{J_i}| \ge 2N 2^{2N}$. If $\rd(I_i \Cup J_i, \I_i) > N^{\frac14}$, we have
\begin{align}\label{FFJ-3}
F_i(I_i, J_i)
\lesssim F_{i, 3}(\rd(I_i \Cup J_i, \I_i))
\le F_{i, 3}(N^{\frac14}).
\end{align}
If $\rd(I_i \Cup J_i, \I_i) \le N^{\frac14}$, then $|c_{I_i \Cup J_i}| \le 2N^{\frac14}(1+\ell(I_i \Cup J_i))$. If $\ell(I_i \Cup J_i) > N^{\frac14} 2^{2N+1}$, then
\begin{align}\label{FFJ-4}
\rd(I_i, J_i)
=\ell(J_i)^{-1} \d(I_i, J_i)
\ge \ell(J_i)^{-1} \ell(I_i \Cup J_i)/3
\ge 2N^{\frac14}/3
\ge 2N^{\frac18},
\end{align}
whenever $N \ge 3^8$. If $\ell(I_i \Cup J_i) \le N^{\frac14} 2^{2N+1}$, then
\begin{align*}
\d(I_i, J_i)
&\simeq |c_{I_i} - c_{J_i}|
\ge |c_{J_i}| - |c_{I_i \Cup J_i}| - |c_{I_i} - c_{I_i \Cup J_i}|
\\
&\ge 2N 2^{2N} - 2N^{\frac14} (1+\ell(I_i \Cup J_i)) - \ell(I_i \Cup J_i)
\\
&\ge 2N 2^{2N} - 3N^{\frac14} (1+N^{\frac14} 2^{2N+1})
\\
&\ge 2N 2^{2N} - 9N^{\frac12} 2^{2N}
\ge N^{\frac12} 2^{2N},
\end{align*}
which implies that
\begin{align}\label{FFJ-5}
\rd(I_i, J_i)
=\ell(J_i)^{-1} \d(I_i, J_i)
\gtrsim N^{\frac12}
\ge 2N^{\frac18},
\end{align}
provided $N>1$ large enough. In view of the properties of $F_{ij}$, $j=1, 2, 3$, the desired conclusion immediately follows from \eqref{FFJ-1}--\eqref{FFJ-5}.

To prove part \eqref{list:FE-2}, let $I_i \in \D_i$ and $J_i \not\in \D_i(2N)$ be such that $\rd(I_i, J_i)<1$. First, considering the case $\ell(J_i)<2^{-2N}$, we have
\begin{align*}
\widetilde{F}_i(I_i, J_i)
\lesssim F_{i, 1}(\ell(I_i))
\le F_{i, 1}(2^{-2N}) \to 0
\quad\text{ as } \, N \to \infty.
\end{align*}
We next treat the case $\ell(J_i)>2^{2N}$. If $\ell(I_i)>2^N$, then
\begin{align*}
\widetilde{F}_i(I_i, J_i)
\lesssim
\widetilde{F}_{i, 2}(\ell(I_i))
\le \widetilde{F}_{i, 2}(2^N) \to 0
\quad\text{ as } \, N \to \infty;
\end{align*}
otherwise, $\ell(I_i) \le 2^N \le 2^{-N} \ell(J_i)$.

If $2^{-2N} \le \ell(J_i) \le 2^{2N}$ and $\rd(J_i, 2^{2N} \I_i) > 2N$, then
\begin{align*}
\d(J_i, \I_i) \ge \d(J_i, 2^{2N} \I_i) > 2N 2^{2N}.
\end{align*}
It follows $\rd(I_i, J_i)<1$ that $I_i \subset 5J_i$. Hence, for any $k \ge 0$,
\begin{align*}
\d(2^k \cdot 5 J_i, \I_i)
\ge \d(J_i, \I_i) - 2^k \cdot 5 \ell(J_i)/2
> 2N 2^{2N} - 2^k \cdot 5 \cdot 2^{2N}/2,
\end{align*}
which in turn implies that
\begin{align*}
\rd(2^k(I_i \Cup J_i), \I_i)
\ge \rd(2^k \cdot 5J_i, \I_i)
\ge \frac{\d(2^k \cdot 5 J_i, \I_i)}{2^k \cdot 5 \cdot 2^{2N}}
>\frac25 2^{-k} N - \frac12.
\end{align*}
Let $N \ge 9$ and $k_N := [\log_2 \sqrt{N}]$. Thus, for all $0 \le k \le k_N$,
\[
\rd(2^k(I_i \Cup J_i), \I_i) \ge \frac{2\sqrt{N}}{5} - \frac12 \ge \frac{\sqrt{N}}{5}.
\]
Using this estimate and the monotonicity of $F_{i, 3}$, we conclude that
\begin{align*}
\widetilde{F}_i(I_i, J_i)
\lesssim \sum_{k=0}^{k_N} 2^{-k \delta_i} F_{i, 3}(\sqrt{N}/5)
+ \sum_{k=k_N+1}^{\infty} 2^{-k \delta_i}
\lesssim F_{i, 3}(\sqrt{N}) + 2^{-k_N \delta_i} \to 0,
\end{align*}
as $N \to \infty$. This completes the proof of part \eqref{list:FE-2}.

Let us next justify part \eqref{list:FE-3}. If $\ell(J_i)<2^{-2N}$, then $2^{-N} \le \ell(I_i)=2^{-k_i} \ell(J_i) \le 2^{-k_i-2N}$, hence $k_i \le -N$. If $\ell(J_i) > 2^{2N}$, then $2^{2N} < \ell(J_i)= 2^{k_i} \ell(I_i) \le 2^{k_i+N}$, hence $k_i \ge N$. If $2^{-2N} \le \ell(J_i) \le 2^{2N}$ and $\rd(J_i, 2^{2N} \I_i) > 2N$, then $\d(J_i, 2^{2N} \I_i) > 2N 2^{2N}$. The condition $I_i \in \D_i(N)$ gives $\d(I_i, 2^N \I_i) \le N 2^N$, which in turn implies $I_i \subset (N+2)2^N \I_i \subset 2^{2N} \I_i$ and
\begin{align*}
\rd(I_i, J_i)
=\frac{\d(I_i, J_i)}{\max\{\ell(I_i), \ell(J_i)\}}
\ge \frac{\d(J_i, 2^{2N} \I_i)}{2^{2N}}
\ge 2N.
\end{align*}
This shows part \eqref{list:FE-3}, while part \eqref{list:FE-4} just follows from definition. 
\end{proof}

Let us next present some estimates for integrals, which will be used below frequently. 

\begin{lemma}\label{lem:PQR}
For each $i=1, 2$, let $(F_{i, 1}, F_{i, 2}, F_{i, 3}) \in \mathscr{F}$ satisfy that $F_{i, 1}$ is increasing, $F_{i, 2}$ and $F_{i, 3}$ are decreasing. Let $I_i, J_i \in \D_i$ with $\ell(I_i) \le \ell(J_i)$, $i=1, 2$. Then the following hold: 
\begin{enumerate}
\item[(a)] If $\rd(I_i, J_i) \ge 1$, then 
\begin{align}\label{def:P}
\mathscr{P}_i(I_i, J_i)
:= \int_{I_i} \int_{J_i} \frac{F_i(x_i, y_i) \ell(I_i)^{\delta_i}}{|x_i-y_i|^{n_i+\delta_i}}  \, dx_i \, dy_i 
\lesssim F_i(I_i, J_i) \frac{\rs(I_i, J_i)^{\frac{n_i}{2}+\delta_i}}{\rd(I_i, J_i)^{n_i + \delta_i}} 
|I_i|^{\frac12} |J_i|^{\frac12},  
\end{align}
where
\begin{align}\label{def:FF1}
F_i(x_i, y_i)
:= F_{i, 1}(|y_i - c_{I_i}|) F_{i, 2}(|x_i - y_i|) F_{i, 3}\bigg(1 + \frac{|x_i + y_i|}{1 + |x_i - y_i|}\bigg).
\end{align}

\item[(b)] There holds  
\begin{equation}\label{def:Q} 
\mathscr{Q}_i(I_i)
:= \int_{I_i} \int_{3I_i \setminus I_i} \frac{F_i(x_i, y_i)}{|x_i-y_i|^{n_i}}  dx_i \, dy_i 
\lesssim \widetilde{F}_i(I_i) |I_i|
\le \widetilde{F}_i(I_i, J_i) |I_i|, 
\end{equation}
where
\begin{align}\label{def:FF2}
F_i(x_i, y_i)
:= F_{i, 1}(|x_i - y_i|) F_{i, 2}(|x_i - y_i|) F_{i, 3}\bigg(1 + \frac{|x_i + y_i|}{1 + |x_i - y_i|}\bigg).
\end{align}

\item[(c)] If $\rd(I_i, J_i) < 1$ and $I_i \cap J_i = \emptyset$, then 
\begin{equation}\label{def:QIJ}
\mathscr{Q}_i(I_i, J_i)
:= \int_{I_i} \int_{J_i \setminus 3I_i}  F_i(x_i, y_i)
\frac{\ell(I_i)^{\delta_i}}{|x_i - c_{I_i}|^{n_i + \delta_i}} \, dx_i \, dy_i
\lesssim \frac{\widetilde{F}_i(I_i, J_i) |I_i|}{\ird(I_i, J_i)^{\delta_i}},
\end{equation}
where $F_i(x_i, y_i)$ is given in \eqref{def:FF1}. 

\item[(d)] There holds 
\begin{align}\label{def:R}
\mathscr{R}_i(I_i)
:= \int_{I_i} \int_{(3I_i)^c} F_i(x_i, y_i) 
\frac{\ell(I_i)^{\delta_i}}{|x_i - y_i|^{n_i+\delta_i}} \, dx_i \, dy_i
\lesssim \widetilde{F}_i(I_i) |I_i|,
\end{align}
where $F_i(x_i, y_i)$ is given in \eqref{def:FF1}.

\item[(e)] Given $I_i \subsetneq J_i$, let $J_{I_i} \in \ch(J_i)$ be such that $I_i \subset J_{I_i}$. Then, 
\begin{align}\label{def:RIJ}
\mathscr{R}_i(I_i, J_i)
:= \int_{I_i} \int_{J_{I_i}^c}  
\frac{F_i(x_i, y_i) \ell(I_i)^{\delta_i}}{|x_i - y_i|^{n_i+\delta_i}} \, dx_i \, dy_i
\lesssim \widetilde{F}_i(I_i, J_i) |I_i| [\ell(I_i)/\ell(J_i)]^{\frac{\delta_i}{2}},
\end{align}
whenever $d_i := \d(I_i, J_{I_i}^c) > \ell(I_i)^{\frac12} \ell(J_i)^{\frac12}$, where $F_i(x_i, y_i)$ is given in \eqref{def:FF1}. 
\end{enumerate} 
\end{lemma}

\begin{proof}
Let us first show \eqref{def:P}. For any $x_i \in J_i$ and $y_i \in I_i$,
\begin{align}\label{PP-1}
|x_i - y_i|
\ge \d(I_i, J_i)
\ge \ell(J_i)
\ge \ell(I_i)
\ge 2|y_i - c_{I_i}|,
\end{align}
and
\begin{align}\label{PP-2}
|x_i - y_i| 
\le \ell(I_i \Cup J_i) 
\le \d(I_i, J_i) + \ell(I_i) + \ell(J_i) 
\le 3\d(I_i, J_i) 
\le 3|x_i - y_i|.
\end{align}
Moreover, $x_i, y_i \in I_i \Cup J_i$ implies $\frac12(x_i + y_i) \in I_i \Cup J_i$, and
\begin{align}\label{PP-3}
1+\frac{|x_i + y_i|}{1 + |x_i - y_i|}
\ge \frac{2\d(I_i \Cup J_i, \I_i)}{1 + \ell(I_i \Cup J_i)}
\ge \frac{\d(I_i \Cup J_i, \I_i)}{\max\{\ell(I_i \Cup J_i), 1\}}
=\rd(I_i \Cup J_i, \I_i).
\end{align}
Then, one has 
\begin{align*}
F_i(x_i, y_i)
\le F_{i, 1}(\ell(I_i)) F_{i, 2}(\ell(I_i \Cup J_i)) F_{i, 3}(\rd(I_i \Cup J_i, \I_i))
= F_i(I_i, J_i).
\end{align*}
which along with \eqref{PP-1}--\eqref{PP-3} gives 
\begin{align*}
\mathscr{P}_i(I_i, J_i)
&\lesssim \frac{\ell(I_i)^{\delta_i}}{\ell(I_i \Cup J_i)^{n_i+\delta_i}} F_i(I_i, J_i) |I_i| |J_i|
\\
&= \bigg(\frac{\ell(I_i)}{\ell(J_i)}\bigg)^{\frac{n_i}{2}+\delta_i}
\bigg(\frac{\ell(J_i)}{\ell(I_i \Cup J_i)}\bigg)^{n_i + \delta_i}  F_i(I_i, J_i) |I_i|^{\frac12} |J_i|^{\frac12}
\\
&\lesssim F_i(I_i, J_i) \frac{\rs(I_i, J_i)^{\frac{n_i}{2}+\delta_i}}{\rd(I_i, J_i)^{n_i + \delta_i}} 
|I_i|^{\frac12} |J_i|^{\frac12}.
\end{align*}

To prove \eqref{def:Q}, we first note that for all $x_i \in 3I_i \setminus I_i$ and $y_i \in I_i$, $|x_i - y_i| \le 2 \ell(I_i)$ and
\begin{align*}
1 + \frac{|x_i + y_i|}{1 + |x_i - y_i|}
&\ge \frac{|x_i + y_i| + |x_i - y_i|}{1 + |x_i - y_i|}
\\
&\ge \frac{2|y_i|}{1 + 2\ell(I_i)}
\ge \frac{\d(I_i, \I_i)}{2\max\{\ell(I_i), 1\}}
= \frac12 \rd(I_i, \I_i),
\end{align*}
which imply that
\begin{align}\label{eq:fxy}
F_i(x_i, y_i)
\le F_{i, 1}(\ell(I_i)) F_{i, 2}(|x_i - y_i|)  F_{i, 3}(\rd(I_i, \I_i)).
\end{align}
Observe also that there exists a sequence of cubes $\{J_{k_i}\}_{k_i=1}^{3^{n_i}-1}$ in $\R^{n_i}$ such that
\begin{align}\label{cover}
3I_i \setminus I_i = \bigcup_{k_i=1}^{3^{n_i}-1} J_{k_i},
\quad\text{ where } \quad
\d(I_i, J_{k_i}) =0 \text{ and }
\ell(J_{k_i}) = \ell(I_i).
\end{align}
Then it follows from \eqref{eq:fxy}, H\"{o}lder's inequality, \eqref{cover}, and Lemma \ref{lem:diag} that
\begin{align*}
\mathscr{Q}_i(I_i)
&\lesssim F_{i, 1}(\ell(I_i)) F_{i, 3}(\rd(I_i, \I_i))
\int_{I_i} \int_{3I_i \setminus I_i} \frac{F_{i, 2}(|x_i - y_i|)}{|x_i - y_i|^{n_i}}  dx_i \, dy_i
\\
&\le F_{i, 1}(\ell(I_i)) F_{i, 3}(\rd(I_i, \I_i)) |I_i|^2
\bigg(\frac{1}{|I_i|^2} \int_{I_i} \int_{3I_i \setminus I_i} \frac{dx_i \, dy_i}{|x_i - y_i|^{s n_i}}  \bigg)^{\frac1s}
\\
&\quad\times \bigg(\frac{1}{|I_i|^2} \int_{I_i} \int_{3I_i \setminus I_i} 
F_{i, 2}(|x_i - y_i|)^{s'}  dx_i \, dy_i \bigg)^{\frac{1}{s'}}
\\
&\lesssim F_{i, 1}(\ell(I_i)) \widetilde{F}_{i, 2}(\ell(I_i)) F_{i, 3}(\rd(I_i, \I_i)) |I_i|
\\
&\le \widetilde{F}_i(I_i) |I_i|
\le \widetilde{F}_i(I_i, J_i) |I_i|,
\end{align*}
provided $\rd(I_i, \I_i) \ge \rd(I_i \Cup J_i, \I_i)$, where $1<s<1+\frac{1}{n_i}$.

Let us next turn to the proof of \eqref{def:QIJ}. Let $x_i \in J_i \setminus 3I_i$ and $y_i \in I_i$. Then $|x_i - y_i| \ge \ell(I_i)$ and $|x_i - c_{I_i}| > \ell(I_i) \ge 2|y_i - c_{I_i}|$. It follows from $\rd(I_i, J_i)<1$ that $\d(I_i, J_i) < \ell(J_i)$ and
\begin{align*}
|x_i - y_i|
\le \ell(J_i) + \d(I_i, J_i) + \ell(I_i)
\le 3\ell(J_i).
\end{align*}
Note that $|y_i| \ge \d(I_i, \I_i) \ge \d(I_i \Cup J_i, \I_i)$. Hence, we obtain
\begin{align*}
1+ \frac{|x_i + y_i|}{1 + |x_i - y_i|}
&\ge \frac{1 + 2|y_i|}{1 + |x_i - y_i|}
\ge \frac{2\d(I_i \Cup J_i, \I_i)}{1 + 3\ell(J_i)}
\\
&=\frac{\max\{\ell(I_i \Cup J_i), 1\}}{1 + 3\ell(J_i)} \rd(I_i \Cup J_i, \I_i)
\ge \frac14 \rd(I_i \Cup J_i, \I_i),
\end{align*}
which implies that
\begin{align}\label{FXY-1}
F_i(x_i, y_i)
&\lesssim F_{i, 1}(\ell(I_i)) F_{i, 2}(\ell(I_i)) F_{i, 3}(\rd(I_i \Cup J_i, \I_i))
\\ \nonumber
&\le F_{i, 1}(\ell(I_i)) \widetilde{F}_{i, 2}(\ell(I_i)) \widetilde{F}_{i, 3}(I_i \Cup J_i)
= \widetilde{F}_i(I_i, J_i).
\end{align}
In addition,
\begin{align}\label{FXY-2}
|x_i - c_{I_i}|
\ge \ell(I_i)/2 + \d(I_i, J_i)
\ge \ell(I_i) \ird(I_i, J_i)/2.
\end{align}
Letting
\begin{align*}
R_{k_i} := \big\{x_i \in \R^{n_i}: 
2^{k_i} \ell(I_i) < |x_i - c_{I_i}| \le 2^{k_i + 1} \ell(I_i)\big\}, \quad k_i \ge k_i^0,
\end{align*}
where $k_i^0 := \max\{0, \, \log_2 (\ird(I_i, J_i)/2)\}$, we invoke \eqref{FXY-1} and \eqref{FXY-2} to conclude
\begin{align*}
\mathscr{Q}_i(I_i, J_i)
&\le \sum_{k_i=k_i^0}^{\infty} \int_{I_i} \int_{R_{k_i}}  F_i(x_i, y_i)
\frac{\ell(I_i)^{\delta_i}}{|x_i - c_{I_i}|^{n_i + \delta_i}}  dx_i \, dy_i
\\
&\lesssim \sum_{k_i=k_i^0}^{\infty} 2^{-k_i \delta_i}  \widetilde{F}_i(I_i, J_i) |I_i|
\lesssim \frac{\widetilde{F}_i(I_i, J_i) |I_i|}{\ird(I_i, J_i)^{\delta_i}}.
\end{align*}

To proceed, note that for all $x_i \in (3I_i)^c$ and $y_i \in I_i$,
\begin{align}\label{TH-2}
|y_i - c_{I_i}| \le \ell(I_i) \le |x_i - y_i|.
\end{align}
Besides, for any $k_i \ge 0$ and $x_i, y_i \in \R^{n_i}$ satisfying $2^{k_i} \ell(I_i) \le |x_i - y_i| < 2^{k_i+1} \ell(I_i)$,
\begin{multline}\label{TH-3}
2\bigg(1 + \frac{|x_i + y_i|}{1 + |x_i - y_i|}\bigg) 
\ge 1 + \frac{1 + |x_i + y_i| + |x_i - y_i|}{1 + |x_i - y_i|}
\ge 1 + \frac{2|y_i|}{1 + |x_i - y_i|}
\\ 
\ge 1 + \frac{2\d(2^{k_i} I_i, \I_i)}{1 + 2^{k_i + 1} \ell(I_i)}
\ge \frac12 \frac{\d(2^{k_i} I_i, \I_i)}{\max\{2^{k_i} \ell(I_i), 1\}}
=\frac12 \rd(2^{k_i} I_i, \I_i).
\end{multline}
Thus, it follows from \eqref{TH-2}--\eqref{TH-3} and the monotonicity of $F_{i, 1}$, $F_{i, 2}$, and $F_{i, 3}$ that
\begin{align}\label{TH-4}
\mathscr{R}_i(I_i)
&= \sum_{k_i=0}^{\infty} \int_{I_i} \int_{2^{k_i} \ell(I_i) \le |x_i - y_i| < 2^{k_i+1}\ell(I_i)}
F_i(x_i, y_i) \frac{\ell(I_i)^{\delta_i}}{|x_i - y_i|^{n_i + \delta_i}} \, dx_i \, dy_i
\\ \nonumber
&\lesssim |I_i| F_{i, 1}(\ell(I_i)) F_{i, 2}(\ell(I_i))
\sum_{k_i = 0}^{\infty} 2^{-k_i \delta_i} F_{i, 3}(\rd(2^{k_i} I_i, \I_i))
\le \widetilde{F}_i(I_i) |I_i|.
\end{align}
This shows \eqref{def:R}. 

Finally, note that $B(c_{I_i}, d_i) \subset J_{I_i}$ and $d_i \le \ell(J_i)$. For all $x_i \in B(c_{I_i}, d_i)^c$ and $y_i \in I_i$, 
\begin{align}\label{TH-21}
2 |y_i - c_{I_i}| 
\le \ell(I_i) 
\le \ell(I_i)^{\frac12} \ell(J_i)^{\frac12} 
\le d_i 
\le |x_i - c_{I_i}|,  
\end{align}
which gives 
\begin{align}
\ell(I_i)/2
&\le |x_i - c_{I_i}|/2
\le |x_i - c_{I_i}| - |y_i - c_{I_i}| 
\\ \nonumber
&\le |x_i - y_i|
\le |x_i - c_{I_i}| + |y_i - c_{I_i}| 
\le 2 |x_i - c_{I_i}|.  
\end{align}
Then, for any $k_i \ge 0$ and $x_i, y_i \in \R^{n_i}$ with $2^{k_i} d_i \le |x_i - c_{I_i}| < 2^{k_i+1} d_i$, we have 
\begin{align}\label{TH-31}
1 + \frac{|x_i + y_i|}{1 + |x_i - y_i|} 
&\ge \frac{2|y_i|}{1 + |x_i - y_i|}
\ge \frac{2|y_i|}{1 + 2|x_i - c_{I_i}|}
\\ \nonumber
&\ge \frac{2 \d(2^{k_i} J_i, \I_i)}{1 + 2^{k_i+2} \ell(J_i)}
\ge \frac13 \frac{\d(2^{k_i} J_i, \I_i)}{\max\{2^{k_i} \ell(J_i), 1\}}
=\frac13 \rd(2^{k_i} J_i, \I_i).
\end{align}
Thus, it follows from \eqref{TH-21}--\eqref{TH-31} and the monotonicity of $F_{i, 1}$, $F_{i, 2}$, and $F_{i, 3}$ that
\begin{align*}
\mathscr{R}_i(I_i, J_i)
&\le \sum_{k_i=0}^{\infty} \int_{I_i} \int_{2^{k_i} d_i \le |x_i - c_{I_i}| < 2^{k_i+1}d_i}
F_i(x_i, y_i) \frac{\ell(I_i)^{\delta_i}}{|x_i - y_i|^{n_i + \delta_i}} \, dx_i \, dy_i
\\
&\lesssim |I_i| (\ell(I_i)/d_i)^{\delta_i} F_{i, 1}(\ell(I_i)) F_{i, 2}(\ell(I_i)) 
\sum_{k_i=0}^{\infty} 2^{-k_i \delta_i} F_{i, 3}(\rd(2^{k_i} J_i, \I_i))
\\
&\le \widetilde{F}_i(I_i, J_i) |I_i| \, \big[\ell(I_i)/\ell(J_i)\big]^{\frac{\delta_i}{2}}.
\end{align*}
The proof of \eqref{def:RIJ} is complete. 
\end{proof}

\subsection{$\BMO$ estimates}
For each $i=1, 2$, let $\D_i$ be dyadic grid  on $\R^{n_i}$, and denote $\D = \D_1 \times \D_2$. Given $i=1, 2$ and $b_i \in \BMO_{\D_i}$, we define the paraproducts on $\R^{n_i}$ by 
\begin{align*}
\Pi_{b_i} f_i  
:= \sum_{I_i \in \D_i} \langle b_i, h_{I_i} \rangle \langle f_i \rangle_{I_i} \, h_{I_i} 
\quad\text{ and }\quad  
\Pi_{b_i}^* f 
:= \sum_{I_i \in \D_i} \langle b_i, h_{I_i} \rangle \langle f_i, h_{I_i} \rangle \, \frac{\mathbf{1}_{I_i}}{|I_i|}. 
\end{align*}
For $b \in \BMO_{\D}$, we define the bi-parameter paraproducts by 
\begin{align*}
\Pi_b f 
&:= \sum_{\substack{I_1 \in \D_1 \\ I_2 \in \D_2}} 
\langle b, h_{I_1} \otimes h_{I_2} \rangle 
\langle f \rangle_{I_1 \times I_2}  \, h_{I_1} \otimes h_{I_2}, 
\\
\Pi_b^* f 
&:= \sum_{\substack{I_1 \in \D_1 \\ I_2 \in \D_2}} 
\langle b, h_{I_1} \otimes h_{I_2} \rangle 
\langle f , h_{I_1} \otimes h_{I_2} \rangle 
\frac{\mathbf{1}_{I_1} \otimes \mathbf{1}_{I_2}}{|I_1 \times I_2|}.
\end{align*}

\begin{lemma}\label{lem:Pi}
The following statements hold: 
\begin{enumerate}

\item For every $i=1, 2$ and $b_i \in \CMO_{\D_i}$, both $\Pi_{b_i}$ and $\Pi^*_{b_i}$ are compact on $L^2(\R^{n_i})$. 

\item For any $b \in \CMO_{\D}$, both $\Pi_b$ and $\Pi^*_b$ are compact on $L^2(\R^{n_1} \times \R^{n_2})$. 
\end{enumerate}
\end{lemma}

\begin{proof}
We just present the proof in the bi-parameter case because the proof in the one-parameter is similar. Let $f, g \in L^2(\R^{n_1} \times \R^{n_2})$ and $N \in \N$. Then we have 
\begin{align*}
\langle P_N^{\perp} (\Pi_b f), g \rangle
&=\sum_{\substack{I_1 \in \D_1 \\ I_2 \in \D_2}} \langle b, h_{I_1} \otimes h_{I_2} \rangle 
\langle f \rangle_{I_1 \times I_2} 
\langle P_N^{\perp}g, h_{I_1} \otimes h_{I_2} \rangle
\\
&=\sum_{\substack{I_1 \not\in \D_1(N) \\ \text{or } I_2 \not\in \D_2(N)}} 
\langle b, h_{I_1} \otimes h_{I_2} \rangle \langle f \rangle_{I_1 \times I_2}  
\langle g, h_{I_1} \otimes h_{I_2} \rangle
\\
&=\sum_{\substack{I_1 \in \D_1 \\ I_2 \in \D_2}} 
\langle P_N^{\perp} b, h_{I_1} \otimes h_{I_2} \rangle 
\langle f \rangle_{I_1 \times I_2}  
\langle g, h_{I_1} \otimes h_{I_2} \rangle
=\langle \Pi_{P_N^{\perp} b} f, g \rangle,  
\end{align*}
which along with \cite[Proposition 6.1]{HPW} implies 
\begin{align}\label{PN-1}
|\langle P_N^{\perp} (\Pi_b f), g \rangle|
=|\langle \Pi_{P_N^{\perp} b} f, g \rangle|
\lesssim \|P_N^{\perp} b\|_{\BMO_{\D}} \|f\|_{L^2} \|g\|_{L^2}, 
\end{align}
where the implicit constant is independent of $N$. Recall that 
\begin{align}\label{PN-12}
b \in \CMO_{\D} \iff 
\lim_{N \to \infty} \|P_N^{\perp} b\|_{\BMO_{\D}} =0.
\end{align}
Hence, \eqref{PN-1} and \eqref{PN-12} yield 
\begin{align*}
\lim_{N \to \infty} \|P_N \Pi_b - \Pi_b \|_{L^2 \to L^2}
=\lim_{N \to \infty} \|P_N^{\perp} \Pi_b \|_{L^2 \to L^2} =0. 
\end{align*}
Since $P_N$ is a finite dimensional operator, we conclude that $P_N \Pi_b$ is a compact operator on $L^2(\Rn)$. Therefore, $\Pi_b$ is compact on $L^2(\R^{n_1} \times \R^{n_2})$. By duality, $\Pi_b^*$ is also compact on $L^2(\R^{n_1} \times \R^{n_2})$. 
\end{proof}

For each $i=1, 2$ and $N \ge 1$, define the projection operator on $\R^{n_i}$ and its orthogonal operator as 
\begin{align*}
P_N f_i := \sum_{I_i \in \D_i(N)} \langle f_i, h_{I_i} \rangle \, h_{I_i} 
\quad\text{ and }\quad 
P_N^{\perp} f_i := \sum_{I_i \not\in \D_i(N)} \langle f_i, h_{I_i} \rangle \, h_{I_i}. 
\end{align*}
Henceforward, given $I_i \subsetneq J_i$, let $J_{I_i} \in \ch(J_i)$ be such that $I_i \subset J_{I_i}$, and denote $\phi_{I_i J_i} := (h_{J_i} - \langle h_{J_i} \rangle_{I_i}) \mathbf{1}_{J_{I_i}^c}$. Note that $\supp(\phi_{I_i J_i}) \subset J_{I_i}^c \subset I_i^c$ and $\|\phi_{I_i J_i}\|_{L^{\infty}} \lesssim |J_i|^{-\frac12}$.

\begin{lemma}\label{lem:HP}
Let $I_1, J_1 \in \D_1$. Assume that all functions appearing in the hypotheses \eqref{list-1}--\eqref{list-5} satisfy properties \eqref{list:P1}--\eqref{list:P5} in Section \ref{sec:proof}. The following statements hold: 
\begin{enumerate} 
\item[(a)] If $\rd(I_1, J_1) \ge 1$, then $b_{I_1 J_1} := \langle T^*(h_{J_1} \otimes 1), h_{I_1} \rangle \in \CMO_{\D_2}$ with 
\begin{align}\label{B1}
\|b_{I_1 J_1}\|_{\BMO_{\D_2}}
\lesssim F_1(I_1, J_1) \frac{\rs(I_1, J_1)^{\frac{n_1}{2} + \delta_1}}{\rd(I_1, J_1)^{n_1 + \delta_1}}, 
\end{align} 
and
\begin{align}\label{PB1}
\|P_N^{\perp} b_{I_1 J_1}\|_{\BMO_{\D_2}}
\lesssim F_1(I_1, J_1) \frac{\rs(I_1, J_1)^{\frac{n_1}{2} + \delta_1}}{\rd(I_1, J_1)^{n_1 + \delta_1}}
\sup_{I_2 \not\in \D_2(N)} \widehat{F}_2(I_2). 
\end{align}

\item[(b)] If $\rd(I_1, J_1) < 1$ and $I_1 \cap J_1 = \emptyset$, then $b_{I_1 J_1} := \langle T^*(h_{J_1} \otimes 1), h_{I_1} \rangle \in \CMO_{\D_2}$ with 
\begin{align}\label{B2}
\|b_{I_1 J_1}\|_{\BMO_{\D_2}}
\lesssim \widetilde{F}_1(I_1, J_1) 
\frac{\rs(I_1, J_1)^{\frac{n_1}{2}}}{\ird(I_1, J_1)^{\delta_1}}, 
\end{align} 
and
\begin{align}\label{PB2}
\|P_N^{\perp} b_{I_1 J_1}\|_{\BMO_{\D_2}}
\lesssim \widetilde{F}_1(I_1, J_1) 
\frac{\rs(I_1, J_1)^{\frac{n_1}{2}}}{\ird(I_1, J_1)^{\delta_1}}
\sup_{I_2 \not\in \D_2(N)} \widehat{F}_2(I_2). 
\end{align}

\item[(c)] There holds $b_{I_1} := \langle T^*(h_{I_1} \otimes 1), h_{I_1}\rangle \in \CMO_{\D_2}$ with 
\begin{align}\label{B3}
\|b_{I_1}\|_{\BMO_{\D_2}} \lesssim \widehat{F}_1(I_1)
\quad\text{and}\quad 
\|P_N^{\perp} b_{I_1}\|_{\BMO_{\D_2}} 
\lesssim \widehat{F}_1(I_1) \sup_{I_2 \not\in \D_2(N)} \widehat{F}_2(I_2). 
\end{align}

\item[(d)] If $I_1 \subsetneq J_1$, then $b_{I_1 J_1} := \langle T^*(\phi_{I_1 J_1} \otimes 1), h_{I_1} \rangle \in \CMO_{\D_2}$ with   
\begin{align}\label{B4}
\|b_{I_1 J_1}\|_{\BMO_{\D_2}}
\lesssim \widetilde{F}_1(I_1) \rs(I_1, J_1)^{\frac{n_1}{2}}, 
\end{align} 
and
\begin{align}\label{PB4}
\|P_N^{\perp} b_{I_1 J_1}\|_{\BMO_{\D_2}}
\lesssim \widetilde{F}_1(I_1) \rs(I_1, J_1)^{\frac{n_1}{2}}
\sup_{I_2 \not\in \D_2(N)} \widehat{F}_2(I_2). 
\end{align}

\item[(e)] If $I_1 \subsetneq J_1$ with $\d(I_1, J_{I_1}^c) > \ell(I_1)^{\frac12} \ell(J_1)^{\frac12}$, then $b_{I_1 J_1} := \langle T^*(\phi_{I_1 J_1} \otimes 1), h_{I_1} \rangle \in \CMO_{\D_2}$ with   
\begin{align}\label{B5}
\|b_{I_1 J_1}\|_{\BMO_{\D_2}}
\lesssim \widetilde{F}_1(I_1, J_1) \rs(I_1, J_1)^{\frac{n_1}{2} + \frac{\delta_1}{2}}, 
\end{align} 
and
\begin{align}\label{PB5}
\|P_N^{\perp} b_{I_1 J_1}\|_{\BMO_{\D_2}}
\lesssim \widetilde{F}_1(I_1, J_1) \rs(I_1, J_1)^{\frac{n_1}{2} + \frac{\delta_1}{2}}
\sup_{I_2 \not\in \D_2(N)} \widehat{F}_2(I_2). 
\end{align}
\end{enumerate} 
\end{lemma}

\begin{proof}

Given $I_2 \in \D_2$, let $a'_{I_2}$ be an arbitrary $\infty$-atom of $\mathrm{H}_{\D_2}^1$, which means that $\supp (a'_{I_2}) \subset I_2$, $\|a'_{I_2}\|_{L^{\infty}} \le |I_2|^{-1}$, and $\int a'_{I_2} \, dx_2 = 0$. Denote $a_{I_2} := a'_{I_2} |I_2|$, which satisfies $\supp (a_{I_2}) \subset I_2$, $\|a_{I_2}\|_{L^{\infty}} \le 1$, and $\int a_{I_2} \, dx_2 = 0$. We split 
\begin{align}\label{PBJ}
\langle b_{I_1 J_1}, a'_{I_2} \rangle
&= \langle T(h_{I_1} \otimes a_{I_2}), h_{J_1} \otimes 1 \rangle \, |I_2|^{-1} 
\\ \nonumber
&= \langle T(h_{I_1} \otimes a_{I_2}), h_{J_1} \otimes \mathbf{1}_{I_2} \rangle \, |I_2|^{-1} 
\\ \nonumber
&\quad+ \langle T(h_{I_1} \otimes a_{I_2}), h_{J_1} \otimes \mathbf{1}_{3I_2 \setminus I_2} \rangle \, |I_2|^{-1} 
\\ \nonumber
&\quad + \langle T(h_{I_1} \otimes a_{I_2}), h_{J_1} \otimes \mathbf{1}_{(3I_2)^c} \rangle \, |I_2|^{-1} 
\\ \nonumber
&=: \mathscr{J}_1 + \mathscr{J}_2 + \mathscr{J}_3. 
\end{align}

First, assume that $\rd(I_1, J_1) \ge 1$. By the compact partial kernel representation, the cancellation of $h_{I_1}$, and \eqref{def:P}, we have 
\begin{align}\label{PBJ-1}
|\mathscr{J}_1|
&\lesssim \mathscr{P}_1(I_1, J_1) |I_1|^{\frac12} |J_1|^{\frac12}  \, C(a_{I_2}, \mathbf{1}_{I_2}) \, |I_2|^{-1} 
\\ \nonumber
&\lesssim F_1(I_1, J_1) \frac{\rs(I_1, J_1)^{\frac{n_1}{2} + \delta_1}}{\rd(I_1, J_1)^{n_1 + \delta_1}} 
F_2(I_2).   
\end{align}
By the compact full kernel representation, the cancellation of $h_{I_1}$, the mixed size-H\"{o}lder and H\"{o}lder conditions of $K$, we obtain that 
\begin{align}\label{PBJ-3}
|\mathscr{J}_2| + |\mathscr{J}_3|
&\lesssim \mathscr{P}_1(I_1, J_1) |I_1|^{-\frac12} |J_1|^{-\frac12} \, 
\mathscr{Q}_2(I_2) |I_2|^{-1}
\\ \nonumber
&\quad+ \mathscr{P}_1(I_1, J_1) |I_1|^{-\frac12} |J_1|^{-\frac12} \, \mathscr{R}_2(I_2) |I_2|^{-1}
\\ \nonumber
&\lesssim F_1(I_1, J_1) \frac{\rs(I_1, J_1)^{\frac{n_1}{2} + \delta_1}}{\rd(I_1, J_1)^{n_1 + \delta_1}}
\widetilde{F}_2(I_2),  
\end{align}
provided \eqref{def:P}, \eqref{def:Q} and \eqref{def:R}. Then, \eqref{PBJ}--\eqref{PBJ-3} imply 
\begin{align}\label{PBJ-34}
|\langle b_{I_1 J_1}, a'_{I_2} \rangle| 
\lesssim F_1(I_1, J_1) \frac{\rs(I_1, J_1)^{\frac{n_1}{2} + \delta_1}}{\rd(I_1, J_1)^{n_1 + \delta_1}}
\widehat{F}_2(I_2), 
\end{align}
which in turn yields \eqref{B1}.

To show \eqref{PB1}, note that by \eqref{B1} and \eqref{PBJ-34}, we have $f_N := P_N^{\perp} b_{I_1 J_1} \in \BMO_{\D_2}$ and 
\begin{align}\label{PN-2}
&|\langle f_N, h_{I_2} \rangle|
= |\langle b_{I_1 J_1}, h_{I_2} \rangle| \mathbf{1}_{\{I \not\in \D_2(N)\}}
\\ \nonumber
&\lesssim A_0 \,A_N |I_2|^{\frac12}
:= F_1(I_1, J_1) \frac{\rs(I_1, J_1)^{\frac{n_1}{2} + \delta_1}}{\rd(I_1, J_1)^{n_1 + \delta_1}} 
\Big(\sup_{I_2 \not\in \D(N)} \widehat{F}_2(I_2)\Big) |I_2|^{\frac12},  
\end{align}
provided that $|I_2|^{-\frac12} h_{I_2}$ is an $\infty$-atom of $\mathrm{H}_{\D_2}^1$. 
Let $J_2 \in \D_2$. It was proved in \cite[p. 35]{J83} that there exist a constant $f_N^{J_2}$ and an $\infty$-atom $a_{J_2} \in \mathrm{H}_{\D_2}^1$ supported on $J_2$ so that 
\begin{align}\label{CFN}
\inf_{c \in \R} \fint_{J_2} |f_N - c| \, dx_2 
=\fint_{J_2} |f_N - f_N^{J_2}| \, dx_2 
=|\langle f_N, a_{J_2}\rangle|. 
\end{align}
Thus, \eqref{PB1} is reduced to showing 
\begin{align}\label{PBA}
|\langle f_N, a_{J_2}\rangle| 
\lesssim A_0\, A_N. 
\end{align}
We may assume $|\langle f_N, a_{J_2}\rangle| \neq 0$, otherwise there is nothing to prove. Since $f_N \in \BMO_{\D_2}$ and $\{h_I\}_{I \in \D_2}$ is an unconditional basis of $\mathrm{H}_{\D_2}^1$, there exists an integer $k \in \N$ such that 
\begin{align}\label{PN-4}
\bigg|\bigg\langle f_N, \, a_{J_2} - \sum_{I \in \D_2(k)} 
\langle a_{J_2}, h_I\rangle \, h_I \bigg\rangle \bigg|
\le \frac12 |\langle f_N, a_{J_2}\rangle|.  
\end{align}
Denote 
\begin{align*}
\phi_k(\alpha) 
:= |J_2|^{\frac12} \bigg(\sum_{I \in \D_2(k)} 
|\langle a_{J_2}, h_I\rangle|^{\frac{2+\alpha}{1+\alpha}} \bigg)^{\frac{1+\alpha}{2+\alpha}} , 
\quad \, \alpha>0. 
\end{align*}
Observe that 
\begin{align*}
\lim_{\alpha \to 0^+} \phi_k(\alpha) 
= |J_2|^{\frac12} \bigg(\sum_{I \in \D_2(k)} |\langle a_{J_2}, h_I\rangle|^2 \bigg)^{\frac12} 
\le |J_2|^{\frac12} \|a_{J_2}\|_{L^2}
\le 1. 
\end{align*}
Thus, there exists some $\alpha_k \in (0, 1)$ small enough so that 
\begin{align}\label{PN-24}
\phi_k(\alpha_k) \le 2. 
\end{align}
It follows from H\"{o}lder's inequality and \eqref{PN-2}--\eqref{PN-24} that 
\begin{align*}
|\langle f_N, a_{J_2} \rangle| 
&\le 2 \bigg|\bigg\langle f_N, \sum_{I \in \D_2(k)} 
\langle a_{J_2}, h_I\rangle \, h_I \bigg\rangle \bigg| 
\le 2 \sum_{I \in \D_2(k)} |\langle a_{J_2}, h_I\rangle| \, |\langle f_N, h_I \rangle| 
\\
&\le 2 \bigg(\sum_{I \in \D_2(k)} |\langle a_{J_2}, h_I\rangle|^{\frac{2+\alpha_k}{1+\alpha_k}} \bigg)^{\frac{1+\alpha_k}{2+\alpha_k}} 
\bigg(\sum_{I \subset J_2} |\langle f_N, h_I \rangle|^{2+\alpha_k} \bigg)^{\frac{1}{2+\alpha_k}}  
\\
&\le 2 \phi_k(\alpha_k) |J_2|^{-\frac12} \big(A_0 A_N |J_2|^{\frac12} \big)^{\frac{\alpha_k}{2+\alpha_k}} 
\bigg(\sum_{I \subset J_2} |\langle f_N - f_N^{J_2}, h_I \rangle|^2 \bigg)^{\frac{1}{2+\alpha_k}}  
\\
&\le 4 (A_0 A_N)^{\frac{\alpha_k}{2+\alpha_k}} 
\bigg(\fint_{J_2} |f_N - f_N^{J_2}|^2 \, dx_2 \bigg)^{\frac{1}{2+\alpha_k}}
\\
&\lesssim (A_0 A_N)^{\frac{\alpha_k}{2+\alpha_k}} 
\bigg(\fint_{J_2} |f_N - f_N^{J_2}| \, dx_2 \bigg)^{\frac{2}{2+\alpha_k}}
\\
&= (A_0 A_N)^{\frac{\alpha_k}{2+\alpha_k}} |\langle f_N, a_{J_2} \rangle|^{\frac{2}{2+\alpha_k}}. 
\end{align*}
Together with that $0 < |\langle f_N, a_{J_2} \rangle| < \infty$, this gives \eqref{PBA}, hence shows \eqref{PB1}.

Second, suppose that $\rd(I_1, J_1) < 1$ and $I_1 \cap J_1 = \emptyset$. By the compact partial kernel representation, the estimates \eqref{def:Q} and \eqref{def:QIJ} imply 
\begin{align}\label{PBJ-4}
|\mathscr{J}_1| 
&\le |\langle T(h_{I_1} \otimes a_{I_2}), 
(h_{J_1} \mathbf{1}_{3I_1 \setminus I_1}) \otimes \mathbf{1}_{I_2} \rangle| |I_2|^{-1}
\\ \nonumber
&\quad+ |\langle T(h_{I_1} \otimes a_{I_2}), 
(h_{J_1} \mathbf{1}_{J_1 \setminus 3I_1}) \otimes \mathbf{1}_{I_2} \rangle| |I_2|^{-1} 
\\ \nonumber
&\lesssim \big[\mathscr{Q}_1(I_1) + \mathscr{Q}_1(I_1, J_1) \big]
|I_1|^{-\frac12} |J_1|^{-\frac12} \, C(a_{I_2}, \mathbf{1}_{I_2}) |I_2|^{-1} 
\\ \nonumber
&\lesssim \bigg[\widetilde{F}_1(I_1) + \frac{\widetilde{F}_1(I_1, J_1)}{\ird(I_1, J_1)^{\delta_1}}\bigg] 
|I_1|^{\frac12} |J_1|^{-\frac12} \, F_2(I_2) 
\\ \nonumber
&\lesssim \widetilde{F}_1(I_1, J_1) \frac{\rs(I_1, J_1)^{\frac{n_1}{2}}}{\ird(I_1, J_1)^{\delta_1}} 
F_2(I_2), 
\end{align}
where we have used that $\ird(I_1, J_1) \le 2$, otherwise $J_1 \cap 3I_1 = \emptyset$. For $\mathscr{J}_2$, we utilize the compact full kernel representation, the size and mixed size-H\"{o}lder conditions of $K$, \eqref{def:Q}, and \eqref{def:QIJ} to arrive at 
\begin{align}\label{PBJ-5}
|\mathscr{J}_2| 
&\le |\langle T(h_{I_1} \otimes a_{I_2}), 
(h_{J_1} \mathbf{1}_{3I_1 \setminus I_1}) \otimes \mathbf{1}_{3I_2 \setminus I_2} \rangle |I_2|^{-1}
\\ \nonumber
&\quad+ |\langle T(h_{I_1} \otimes a_{I_2}), 
(h_{J_1} \mathbf{1}_{J_1 \setminus 3I_1}) \otimes \mathbf{1}_{3I_2 \setminus I_2} \rangle |I_2|^{-1} 
\\ \nonumber
&\lesssim \big[\mathscr{Q}_1(I_1) + \mathscr{Q}_1(I_1, J_1) \big] 
|I_1|^{-\frac12} |J_1|^{-\frac12} \, \mathscr{Q}_2(I_2) |I_2|^{-1} 
\\ \nonumber
&\lesssim \widetilde{F}_1(I_1, J_1) \frac{\rs(I_1, J_1)^{\frac{n_1}{2}}}{\ird(I_1, J_1)^{\delta_1}} 
\widetilde{F}_2(I_2). 
\end{align}
Additionally, by \eqref{def:Q}, \eqref{def:QIJ}, and \eqref{def:R}, 
\begin{align}\label{PBJ-6}
|\mathscr{J}_3| 
&\le |\langle T(h_{I_1} \otimes a_{I_2}), 
(h_{J_1} \mathbf{1}_{3I_1 \setminus I_1}) \otimes \mathbf{1}_{(3I_2)^c} \rangle |I_2|^{-1}
\\ \nonumber
&\quad+ |\langle T(h_{I_1} \otimes a_{I_2}), 
(h_{J_1} \mathbf{1}_{J_1 \setminus 3I_1}) \otimes \mathbf{1}_{(3I_2)^c} \rangle |I_2|^{-1} 
\\ \nonumber
&\lesssim \big[\mathscr{Q}_1(I_1) + \mathscr{Q}_1(I_1, J_1) \big] 
|I_1|^{-\frac12} |J_1|^{-\frac12} \,\mathscr{R}_2(I_2) |I_2|^{-1}
\\ \nonumber
&\lesssim \widetilde{F}_1(I_1, J_1) \frac{\rs(I_1, J_1)^{\frac{n_1}{2}}}{\ird(I_1, J_1)^{\delta_1}} 
\widetilde{F}_2(I_2). 
\end{align}
Consequently, it follows from \eqref{PBJ} and \eqref{PBJ-4}--\eqref{PBJ-6} that  
\begin{align}\label{PBJ-67}
|\langle b_{I_1 J_1}, a'_{I_2} \rangle| 
\lesssim F_1(I_1, J_1) \frac{\rs(I_1, J_1)^{\frac{n_1}{2}}}{\ird(I_1, J_1)^{\delta_1}}
\widehat{F}_2(I_2), 
\end{align}
which implies \eqref{B2} as desired. As argued in \eqref{PBA}, the inequality \eqref{PB2} follows from\eqref{PBJ-67}.

Third, we are going to prove \eqref{B3}. We rewrite 
\begin{align}\label{PBJ-7}
\langle b_{I_1}, a'_{I_2} \rangle
= \sum_{I'_1, \, I''_1 \in \ch(I_1)} \langle h_{I_1} \rangle_{I'_1} \langle h_{I_1} \rangle_{I''_1}  |I_2|^{-1} 
\langle T(\mathbf{1}_{I'_1} \otimes a_{I_2}), \mathbf{1}_{I''_1} \otimes 1 \rangle. 
\end{align}
If $I'_1 \neq I''_1$, then $\d(I'_1, I''_1) < 1$ and $I'_1 \cap I''_1 = \emptyset$, which is similar to the case in \eqref{B2}. Hence, 
\begin{align}\label{PBJ-8}
|\langle T(\mathbf{1}_{I'_1} \otimes a_{I_2}), \mathbf{1}_{I''_1} \otimes 1 \rangle| 
\lesssim \widetilde{F}_1(I'_1, I''_1) |I'_1| \widehat{F}_2(I_2) |I_2|
\lesssim \widetilde{F}_1(I_1) |I_1| \widehat{F}_2(I_2) |I_2|. 
\end{align}
In the case $I'_1 = I''_1$, we use the diagonal $\CMO$ condition, the compact partial kernel representation, \eqref{def:Q}, and \eqref{def:R} to obtain 
\begin{align}\label{PBJ-9}
|\langle T(\mathbf{1}_{I'_1} \otimes a_{I_2}), \mathbf{1}_{I'_1} \otimes 1 \rangle|
&\le |\langle T(\mathbf{1}_{I'_1} \otimes a_{I_2}), \mathbf{1}_{I'_1} \otimes \mathbf{1}_{I_2} \rangle|
\\ \nonumber
&\quad+ |\langle T(\mathbf{1}_{I'_1} \otimes a_{I_2}), 
\mathbf{1}_{I'_1} \otimes \mathbf{1}_{3I_2 \setminus I_2} \rangle|
\\ \nonumber
&\quad + |\langle T(\mathbf{1}_{I'_1} \otimes a_{I_2}), 
\mathbf{1}_{I'_1} \otimes \mathbf{1}_{(3I_2)^c} \rangle|  
\\ \nonumber
&\lesssim F_1(I'_1) |I'_1| \, F_2(I_2) |I_2| 
\\ \nonumber
&\quad+ C(\mathbf{1}_{I'_1}, \mathbf{1}_{I'_1}) \mathscr{Q}_2(I_2) 
\\ \nonumber
&\quad + C(\mathbf{1}_{I'_1}, \mathbf{1}_{I'_1}) \mathscr{R}_2(I_2)
\\ \nonumber
&\lesssim \widehat{F}_1(I_1) |I_1| \widehat{F}_2(I_2) |I_2|. 
\nonumber
\end{align}
Now gathering \eqref{PBJ-7}--\eqref{PBJ-9}, we obtain 
\begin{align}\label{PBJ-91}
|\langle b_{I_1}, a'_{I_2} \rangle|
\lesssim \widehat{F}_1(I_1) \, \widehat{F}_2(I_2). 
\end{align}
Along with the argument in \eqref{PBA}, this gives \eqref{B3}.

To proceed to \eqref{B4} and \eqref{PB4}, we invoke \eqref{def:Q} and \eqref{def:R} to arrive at 
\begin{align*}
|\langle b_{I_1 J_1}, a'_{I_2} \rangle|
&=|\langle T(h_{I_1} \otimes a_{I_2}), \phi_{I_1 J_1} \otimes 1 \rangle| |I_2|^{-1} 
\\
&\le |\langle T(h_{I_1} \otimes a_{I_2}), 
(\phi_{I_1 J_1} \mathbf{1}_{3I_1}) \otimes \mathbf{1}_{I_2} \rangle| |I_2|^{-1}
\\
&\quad + |\langle T(h_{I_1} \otimes a_{I_2}), 
(\phi_{I_1 J_1} \mathbf{1}_{3I_1}) \otimes \mathbf{1}_{3I_2 \setminus I_2} \rangle| |I_2|^{-1}
\\
&\quad + |\langle T(h_{I_1} \otimes a_{I_2}), 
(\phi_{I_1 J_1} \mathbf{1}_{3I_1}) \otimes \mathbf{1}_{(3I_2)^c} \rangle| |I_2|^{-1}
\\
&\quad + |\langle T(h_{I_1} \otimes a_{I_2}), 
(\phi_{I_1 J_1} \mathbf{1}_{(3I_1)^c}) \otimes \mathbf{1}_{I_2} \rangle| |I_2|^{-1}
\\
&\quad + |\langle T(h_{I_1} \otimes a_{I_2}), 
(\phi_{I_1 J_1} \mathbf{1}_{(3I_1)^c}) \otimes \mathbf{1}_{3I_2 \setminus I_2} \rangle| |I_2|^{-1}
\\ 
&\quad + |\langle T(h_{I_1} \otimes h_{I_2}), 
(\phi_{I_1 J_1} \mathbf{1}_{(3I_1)^c}) \otimes \mathbf{1}_{(3I_2)^c} \rangle| |I_2|^{-1}
\\
&\lesssim \mathscr{Q}_1(I_1) |I_1|^{-\frac12} |J_1|^{-\frac12} \, C(a_{I_2}, \mathbf{1}_{I_2}) |I_2|^{-1} 
\\ 
&\quad+ \mathscr{Q}_1(I_1) |I_1|^{-\frac12} |J_1|^{-\frac12} \, \mathscr{Q}_2(I_2) |I_2|^{-1} 
\\
&\quad+ \mathscr{Q}_1(I_1) |I_1|^{-\frac12} |J_1|^{-\frac12} \, \mathscr{R}_2(I_2) |I_2|^{-1}
\\
&\quad+ \mathscr{R}_1(I_1) |I_1|^{-\frac12} |J_1|^{-\frac12} \, C(a_{I_2}, \mathbf{1}_{I_2}) |I_2|^{-1} 
\\
&\quad+ \mathscr{R}_1(I_1) |I_1|^{-\frac12} |J_1|^{-\frac12} \, \mathscr{Q}_2(I_2) |I_2|^{-1} 
\\
&\quad+ \mathscr{R}_1(I_1) |I_1|^{-\frac12} |J_1|^{-\frac12} \, \mathscr{R}_2(I_2) |I_2|^{-1} 
\\
&\lesssim \widetilde{F}_1(I_1) \rs(I_1, J_1)^{\frac{n_1}{2}} 
\widehat{F}_2(I_2),  
\end{align*}
which shows \eqref{B4}. Using the strategy in \eqref{PBA}, one can prove \eqref{PB4}.

Finally, let us demonstrate \eqref{B5} and \eqref{PB5}. In this scenario, $\d(I_1, J_{I_1}^c) \ge \ell(I_1)$. By \eqref{def:Q}, \eqref{def:R}, and \eqref{def:RIJ}, one has 
\begin{align*}
|\langle b_{I_1 J_1}, a'_{I_2} \rangle|
&=|\langle T(h_{I_1} \otimes a_{I_2}), \phi_{I_1 J_1} \otimes 1 \rangle| |I_2|^{-1} 
\\
&\le |\langle T(h_{I_1} \otimes a_{I_2}), 
\phi_{I_1 J_1} \otimes \mathbf{1}_{I_2} \rangle| |I_2|^{-1} 
\\
&\quad + |\langle T(h_{I_1} \otimes a_{I_2}), 
\phi_{I_1 J_1} \otimes \mathbf{1}_{3I_2 \setminus I_2} \rangle| |I_2|^{-1} 
\\
&\quad + |\langle T(h_{I_1} \otimes a_{I_2}), 
\phi_{I_1 J_1} \otimes \mathbf{1}_{(3I_2)^c} \rangle| |I_2|^{-1} 
\\
&\lesssim \mathscr{R}_1(I_1, J_1) |I_1|^{-\frac12} |J_1|^{-\frac12} \, C(a_{I_2}, \mathbf{1}_{I_2})  |I_2|^{-1} 
\\
&\quad+ \mathscr{R}_1(I_1, J_1) |I_1|^{-\frac12} |J_1|^{-\frac12} \, \mathscr{Q}_2(I_2)  |I_2|^{-1} 
\\
&\quad+ \mathscr{R}_1(I_1, J_1) |I_1|^{-\frac12} |J_1|^{-\frac12} \, \mathscr{R}_2(I_2)  |I_2|^{-1} 
\\
&\lesssim \widetilde{F}_1(I_1, J_1) \rs(I_1, J_1)^{\frac{n_1}{2} + \frac{\delta_1}{2}} 
\widehat{F}_2(I_2).  
\end{align*}
This coincides with \eqref{B5}. As above, \eqref{PB5} can be proved.  
\end{proof}

\section{The main estimates}\label{sec:proof}
The goal of this section is to prove \eqref{reduction-3}. Our analysis will revolve around relative size and distance between cubes. Parametrizing these sums is the core of the proof. Throughout this section, we always keep in mind that $\ell(I_1) \le \ell(J_1)$ and $\ell(I_2) \le \ell(J_2)$.

We begin with assuming some extra properties of functions appearing in the hypotheses.
\begin{list}{\rm (\theenumi)}{\usecounter{enumi}\leftmargin=1.2cm \labelwidth=1cm \itemsep=0.2cm \topsep=.2cm \renewcommand{\theenumi}{P\arabic{enumi}}}

\item\label{list:P1} Assume that for each $i=1, 2$ and $j=1, 2, 3$, the function $F_{ij}$ in Definition \ref{def:full} is the same as the one in Definition \ref{def:partial}. Indeed, given $(F_{i, 1}^1, F_{i, 2}^1, F_{i, 3}^1) \in \F$ in Definition \ref{def:full} and $(F_{i, 1}^2, F_{i, 2}^2, F_{i, 3}^2) \in \F$ in Definition \ref{def:partial}, if we define
\begin{align*}
F_{i, j} := \max\{F_{i, j}^1, F_{i, j}^2\}, \quad j=1, 2, 3,
\end{align*}
then the compact full and partial kernel representations hold for $(F_{i, 1}, F_{i, 2}, F_{i, 3}) \in \F$.

\item\label{list:P2} Assume that for each $i=1, 2$, the function $F_i$ in Definition \ref{def:WCP} is the same as that in Definition \ref{def:diag-CMO}. This can be seen as above.

\item\label{list:P3} For each $i=1,2$ and $(F_{i, 1}, F_{i, 2}, F_{i, 3}) \in \F$ in Definitions \ref{def:full} and \ref{def:partial}, we may assume that $F_{i, 1}$ is monotone increasing while $F_{i, 2}$ and $F_{i, 3}$ are monotone decreasing. Indeed, it suffices to define
\begin{align*}
F_{i, 1}^*(t) := \sup_{0 \le s \le t} F_{i, 1}(s), \quad
F_{i, 2}^*(t) := \sup_{s \ge t} F_{i, 2}(s), \quad\text{and}\quad
F_{i, 3}^*(t) := \sup_{s \ge t} F_{i, 3}(s).
\end{align*}
Then it is easy to show that the compact full and partial kernel representations hold for $(F_{i, 1}^*, F_{i, 2}^*, F_{i, 3}^*) \in \F$.

\item\label{list:P4} Since any dilation of functions in $\F$, $\F_1$, and $\F_2$ still belongs to the original space, we will often omit all universal constants appearing in the argument involving these functions.

\item\label{list:P5} In light of Lemma \ref{lem:improve}, we will mostly use alternative estimates for kernels in Definitions \ref{def:full} and \ref{def:partial}. For example, in Definition \ref{def:full}, the size condition can be replaced by
\begin{align*}
|K(x, y)| \leq \frac{F_1(x_1, y_1)}{|x_1-y_1|^{n_1}} \frac{F_2(x_2, y_2)}{|x_2-y_2|^{n_2}},
\end{align*}
where
\begin{align*}
F_i(x_i, y_i) := F_{i, 1}(|x_i - y_i|) F_{i, 2}(|x_i - y_i|) F_{i, 3} \bigg(1+ \frac{|x_i + y_i|}{1+|x_i - y_i|} \bigg),
\end{align*}
while the first H\"{o}lder condition can be replaced by
\begin{align*}
|&K(x, y) - K(x, (y_1, y'_2)) - K(x, (y'_1, y_2)) + K(x, y')|
\\
&\leq \bigg(\frac{|y_1-y'_1|}{|x_1-y_1|}\bigg)^{\delta_1} \frac{F_1(x_1, y_1)}{|x_1 - y_1|^{n_1}}
\bigg(\frac{|y_2-y'_2|}{|x_2-y_2|}\bigg)^{\delta_2} \frac{F_2(x_2, y_2)}{|x_2 - y_2|^{n_2}}
\end{align*}
whenever $|y_1-y'_1| \leq |x_1-y_1|/2$ and $|y_2-y'_2| \leq |x_2-y_2|/2$, where
\begin{align*}
F_i(x_i, y_i) := F_{i, 1}(|y_i - y'_i|) F_{i, 2}(|x_i - y_i|) F_{i, 3} \bigg(1+ \frac{|x_i + y_i|}{1+|x_i - y_i|} \bigg).
\end{align*}
Other conditions can be formulated in a similar way.
\end{list}

Now we consider six cases given in \eqref{reduction-3}. 
For simplicity, denote $L^2 := L^2(\R^{n_1} \times \R^{n_2})$. 

\subsection{$\rd(I_1, J_1) \ge 1$ and $\rd(I_2, J_2) \ge 1$}
Write
\begin{align*}
\mathbf{K}(x, y)
:= K(x, y) - K(x, (y_1, c_{I_2})) - K(x, (c_{I_1}, y_2)) + K(x, (c_{I_1}, c_{I_2})).
\end{align*}
By the compact full kernel representation, the cancellation of $h_{I_1}$ and $h_{I_2}$, the H\"{o}lder condition of $K$, and \eqref{def:P}, we have
\begin{align}\label{Gsep}
|\mathscr{G}_{J_1 J_2}^{I_1 I_2}|
&=\bigg|\int_{I_1 \times I_2} \int_{J_1 \times J_2}
\mathbf{K}(x, y) h_{I_1} \otimes h_{I_2}(y) h_{J_1} \otimes h_{J_2}(x) \, dx \, dy \bigg|
\\ \nonumber
&\lesssim \prod_{i=1}^2 |I_i|^{-\frac12} |J_i|^{-\frac12} \mathscr{P}_i(I_i, J_i)
\lesssim \prod_{i=1}^2 F_i(I_i, J_i) \frac{\rs(I_i, J_i)^{\frac{n_i}{2}+\delta_i}}{\rd(I_i, J_i)^{n_i + \delta_i}}.
\end{align}

Let us turn to the estimate for $\mathscr{I}_1^N$. Given $i=1, 2$, $k_i \ge 0$, $j_i \ge 1$, and $I_i, J_i \in \D_i$, we let
\begin{align*}
J_i(k_i, j_i) 
&:= \big\{I_i \in \D_i: \, \ell(I_i) = 2^{-k_i} \ell(J_i), \, j_i \le \rd(I_i, J_i) < j_i +1\big\}, 
\\
I_i(k_i, j_i) 
&:= \big\{J_i \in \D_i: \, \ell(J_i) = 2^{k_i} \ell(I_i), \, j_i \le \rd(I_i, J_i) < j_i +1\big\}.
\end{align*}
Note that for each $i=1, 2$, $I_i \in J_i(k_i, j_i)$ if and only if $J_i \in I_i(k_i, j_i)$,
\begin{align}\label{car-IJ}
&\# J_i(k_i, j_i) \lesssim 2^{k_i n_i} j_i^{n_i -1},
\quad\text{ and }\quad
\# I_i(k_i, j_i) \lesssim j_i^{n_i-1}.
\end{align}
By \eqref{Gsep} and Lemma \ref{lem:FE} part \eqref{list:FE-1}, there exists an $N_0>1$ so that for all $N \ge N_0$,
\begin{align*}
|\mathscr{I}_1^N|
&\le \sum_{\substack{k_1 \ge 0 \\ k_2 \ge 0}} \sum_{\substack{j_1 \ge 1 \\ j_2 \ge 1}}
\sum_{\substack{J_1 \not\in \D_1(2N) \\ \text{or } J_2 \not\in \D_2(2N)}}
\sum_{\substack{I_1 \in J_1(k_1, j_1) \\ I_2 \in J_2(k_2, j_2)}}
|\G_{J_1 J_2}^{I_1 I_2}| \, |f_{I_1 I_2}| |g_{J_1 J_2}|
\lesssim \sum_{i=0}^2 \mathscr{I}_{1, i}^N,
\end{align*}
where
\begin{align*}
\mathscr{I}_{1, 0}^N
& := \varepsilon \sum_{\substack{k_1 \ge 0 \\ k_2 \ge 0}} 
\sum_{\substack{j_1 \ge 1 \\ j_2 \ge 1}}
\sum_{\substack{J_1 \in \D_1 \\ J_2 \in \D_2}}
\sum_{\substack{I_1 \in J_1(k_1, j_1) \\ I_2 \in J_2(k_2, j_2)}}
\prod_{i=1}^2 2^{-k_i(\frac{n_i}{2} + \delta_i)} j_i^{-n_i - \delta_i} 
|f_{I_1 I_2}| |g_{J_1 J_2}|,
\\ 
\mathscr{I}_{1, i}^N
&:= \sum_{\substack{k_1 \ge 0 \\ k_2 \ge 0}} 
\sum_{\substack{j_1, j_2 \ge 1 \\ j_i \ge N^{1/8}}}
\sum_{\substack{J_1 \in \D_1 \\ J_2 \in \D_2}}
\sum_{\substack{I_1 \in J_1(k_1, j_1) \\ I_2 \in J_2(k_2, j_2)}}
\prod_{i=1}^2 2^{-k_i(\frac{n_i}{2} + \delta_i)} j_i^{-n_i - \delta_i} 
|f_{I_1 I_2}| |g_{J_1 J_2}|, 
\end{align*}
for each $i=1, 2$. It follows from the Cauchy--Schwarz inequality and \eqref{car-IJ} that
\begin{align*}
\mathscr{I}_{1, 0}^N
&\le \varepsilon \sum_{k_1, k_2 \ge 0} \sum_{j_1, j_2 \ge 1}
2^{-k_1(\frac{n_1}{2} + \delta_1)} 2^{-k_2(\frac{n_2}{2} + \delta_2)}
j_1^{-n_1 - \delta_1} j_2^{-n_2 - \delta_2}
\nonumber \\
&\quad\times \bigg(\sum_{\substack{I_1 \in \D_1 \\ I_2 \in \D_2}}
\sum_{\substack{J_1 \in I_1(k_1, j_1) \\ J_2 \in I_2(k_2, j_2)}} |f_{I_1 I_2}|^2 \bigg)^{\frac12}
\bigg( \sum_{\substack{J_1 \in \D_1 \\ J_2 \in \D_2}}
\sum_{\substack{I_1 \in J_1(k_1, j_1) \\ I_2 \in J_2(k_2, j_2)}} |g_{J_1 J_2}|^2 \bigg)^{\frac12}
\nonumber \\
&\lesssim \varepsilon \sum_{k_1, k_2 \ge 0} \sum_{j_1, j_2 \ge 1} 
2^{-k_1 \delta_1} 2^{-k_2 \delta_2} j_1^{-\delta_1 -1} j_2^{-\delta_2 -1}
\nonumber \\
&\quad\times \bigg(\sum_{\substack{I_1 \in \D_1 \\ I_2 \in \D_2}} |f_{I_1 I_2}|^2 \bigg)^{\frac12}
\bigg( \sum_{\substack{J_1 \in \D_1 \\ J_2 \in \D_2}} |g_{J_1 J_2}|^2 \bigg)^{\frac12}
\nonumber \\
&\lesssim \varepsilon \|f\|_{L^2} \|g\|_{L^2}.
\end{align*}
and for each $i=1, 2$, 
\begin{align*}
\mathscr{I}_{1, i}^N
&\lesssim \sum_{k_1, k_2 \ge 0} \sum_{\substack{j_1, j_2 \ge 1 \\ j_i \ge N^{1/8}}}
2^{-k_1 \delta_1} 2^{-k_2 \delta_2} j_1^{-\delta_1 -1} j_2^{-\delta_2 -1}
\|f\|_{L^2} \|g\|_{L^2} 
\nonumber \\
&\lesssim N^{-\delta_i/8} \|f\|_{L^2} \|g\|_{L^2}
\le \varepsilon \|f\|_{L^2} \|g\|_{L^2},  
\end{align*}
provided $N \ge N_0$ sufficiently large. Gathering the estimates above, we conclude that 
\begin{align*}
|\mathscr{I}_1^N|
\lesssim \sum_{i=0}^2 \mathscr{I}_{1, i}^N
\lesssim \varepsilon \|f\|_{L^2} \|g\|_{L^2}.
\end{align*}

\subsection{$\rd(I_1, J_1) \ge 1$, $\rd(I_2, J_2) < 1$, and $I_2 \cap J_2 = \emptyset$}
In this case, $\d(I_2, J_2) < \ell(J_2)$, and hence $I_2 \subset 5J_2 \setminus J_2$.

\begin{lemma}\label{lem:THH-2}
Let $I_1, J_1 \in \D_1$ with $\ell(I_1) \le \ell(J_1)$ so that $\rd(I_1, J_1) \ge 1$, and let $I_2, J_2 \in \D_2$ with $\ell(I_2) \le \ell(J_2)$ so that $\rd(I_2, J_2) < 1$ and $I_2 \cap J_2 = \emptyset$. Then
\begin{align}\label{G2}
|\G_{J_1 J_2}^{I_1 I_2}|
\lesssim  F_1(I_1, J_1) \frac{\rs(I_1, J_1)^{\frac{n_1}{2} + \delta_1}}{\rd(I_1, J_1)^{n_1 + \delta_1}} \,
\widetilde{F}_2(I_2, J_2) \frac{\rs(I_2, J_2)^{\frac{n_2}{2}}}{\ird(I_2, J_2)^{\delta_2}}. 
\end{align}
\end{lemma}

\begin{proof}
We split
\begin{align}\label{GTH}
\langle T (h_{I_1} \otimes h_{I_2}), h_{J_1} \otimes h_{J_2} \rangle
&=\langle T(h_{I_1} \otimes h_{I_2}), h_{J_1} \otimes (h_{J_2} \mathbf{1}_{3I_2 \setminus I_2}) \rangle
\\ \nonumber
&\qquad+ \langle T(h_{I_1} \otimes h_{I_2}), h_{J_1} \otimes (h_{J_2} \mathbf{1}_{J_2 \setminus 3I_2}) \rangle.
\end{align}
Note that if $\ird(I_2, J_2)=1+\d(I_2, \partial J_2)/\ell(I_2)>2$, then $J_2 \cap 3I_2=\emptyset$, and hence
\[
\langle T(h_{I_1} \otimes h_{I_2}), h_{J_1} \otimes (h_{J_2} \mathbf{1}_{3I_2 \setminus I_2}) \rangle
=0.
\]
To treat the case $\ird(I_2, J_2) \le 2$, we let $\mathbf{K}_1(x, y) := K(x, y) - K(x, (c_{I_1}, y_2))$. 
We then utilize the compact full kernel representation, the cancellation of $h_{I_1}$, and the H\"{o}lder-size condition to arrive at
\begin{align}\label{GTH-1}
&|\langle T(h_{I_1} \otimes h_{I_2}), h_{J_1} \otimes (h_{J_2} \mathbf{1}_{3I_2 \setminus I_2}) \rangle|
\\ \nonumber
&\quad=\bigg|\int_{I_1 \times I_2} \int_{J_1 \times (3I_2 \setminus I_2)} \mathbf{K}_1(x, y)
h_{I_1} \otimes h_{I_2}(y) h_{J_1} \otimes h_{J_2}(x) \, dx \, dy \bigg|
\\ \nonumber
&\quad\lesssim |I_1|^{-\frac12} |J_1|^{-\frac12} \mathscr{P}_1(I_1, J_1) \,
|I_2|^{-\frac12} |J_2|^{-\frac12} \mathscr{Q}_2(I_2)
\\ \nonumber
&\quad\lesssim F_1(I_1, J_1) \frac{\rs(I_1, J_1)^{\frac{n_1}{2}+\delta_1}}{\rd(I_1, J_1)^{n_1 + \delta_1}}
\widetilde{F}_2(I_2, J_2) \frac{\rs(I_2, J_2)^{\frac{n_2}{2}}}{\ird(I_2, J_2)^{\delta_2}},
\end{align}
where we have used \eqref{def:P}, \eqref{def:Q}, and $\ird(I_2, J_2) \le 2$.

Write
\begin{align*}
\mathbf{K}_2(x, y) :=
K(x, (c_{I_1}, c_{I_2})) - K(x, (c_{I_1}, y_2)) - K(x, (y_1, c_{I_2})) + K(x, y).
\end{align*}
By the compact full kernel representation, the cancellation of $h_{I_1}$ and $h_{I_2}$, and the H\"{o}lder condition, we invoke \eqref{def:P} and \eqref{def:QIJ} to deduce that
\begin{align}\label{GTH-2}
&|\langle T(h_{I_1} \otimes h_{I_2}), h_{J_1} \otimes (h_{J_2} \mathbf{1}_{J_2 \setminus 3I_2}) \rangle|
\\ \nonumber
&\quad=\bigg|\int_{\R^{n_1} \times \R^{n_2}} \int_{\R^{n_1} \times \R^{n_2}} \mathbf{K}_2(x, y) h_{I_1} \otimes h_{I_2}(y)
h_{J_1} \otimes (h_{J_2} \mathbf{1}_{J_2 \setminus 3I_2})(x) \, dx \, dy \bigg|
\\ \nonumber
&\quad\lesssim |I_1|^{-\frac12} |J_1|^{-\frac12} \mathscr{P}_1(I_1, J_1) \, |I_2|^{-\frac12} |J_2|^{-\frac12} \mathscr{Q}_2(I_2, J_2)
\\ \nonumber
&\quad\lesssim F_1(I_1, J_1) \frac{\rs(I_1, J_1)^{\frac{n_1}{2}+\delta_1}}{\rd(I_1, J_1)^{n_1 + \delta_1}}
\widetilde{F}_2(I_2, J_2) \frac{\rs(I_2, J_2)^{\frac{n_2}{2}}}{\ird(I_2, J_2)^{\delta_2}}.
\end{align}
Therefore, \eqref{G2} is a consequence of \eqref{GTH}--\eqref{GTH-2}. 
\end{proof}

To estimate $\mathscr{I}_2^N$, we introduce some notation. Given $i=1, 2$, $k_i \ge 0$, $m_i \ge 1$, and $I_i, J_i  \in \D_i$, we write
\begin{align*}
J_i(k_i, 0, m_i) 
&:= \big\{I_i \in \D_i: \, \ell(I_i) = 2^{-k_i} \ell(J_i), \, \rd(I_i, J_i) < 1, \, m_i \le \ird(I_i, J_i) < m_i + 1\big\}
\\
I_i(k_i, 0, m_i) 
&:= \big\{J_i \in \D_i: \, \ell(J_i) = 2^{k_i} \ell(I_i), \, \rd(I_i, J_i) < 1, \, m_i \le \ird(I_i, J_i) < m_i + 1\big\}.
\end{align*}
Then it is not hard to check that $I_i \in J_i(k_i, 0, m_i)$ if and only if $J_i \in I_i(k_i, 0, m_i)$,
\begin{align}\label{car-IJK}
\#J_i(k_i, 0, m_i) \lesssim  (2^{k_i} + m_i)^{n_i - 1}, \quad\text{ and }\quad
 \# \bigg(\bigcup_{m_i=1}^{2^{k_i}} I_i(k_i, 0, m_i) \bigg) \lesssim 1.
\end{align}
Now invoking Lemma \ref{lem:THH-2} and Lemma \ref{lem:FE} parts \eqref{list:FE-1} and \eqref{list:FE-2} , we obtain for all $N \ge N_0$,
\begin{align*}
|\mathscr{I}_2^N|
&\le \sum_{\substack{k_1 \ge 0 \\ j_1 \ge 1}}
\sum_{\substack{k_2 \ge 0 \\ 1 \le m_2 \le 2^{k_2}}}
\sum_{\substack{J_1 \not\in \D_1(2N) \\ \text{or } J_2 \not\in \D_2(2N)}}
\sum_{\substack{I_1 \in J_1(k_1, j_1) \\ I_2 \in J_2(k_2, 0, m_2)}}
|\G_{J_1 J_2}^{I_1 I_2}| \, |f_{I_1 I_2}| |g_{J_1 J_2}|
\lesssim \sum_{i=1}^3 \mathscr{I}_{2, i}^N,
\end{align*}
where
\begin{align*}
\mathscr{I}_{2, 1}^N
&:= \varepsilon
\sum_{\substack{k_1 \ge 0 \\ j_1 \ge 1}}
\sum_{\substack{k_2 \ge 0 \\ 1 \le m_2 \le 2^{k_2}}}
\sum_{\substack{J_1 \in \D_1 \\ J_2 \in \D_2}}
\sum_{\substack{I_1 \in J_1(k_1, j_1) \\ I_2 \in J_2(k_2, 0, m_2)}}
2^{-k_1(\frac{n_1}{2} + \delta_1)} 2^{-k_2\frac{n_2}{2}} j_1^{-n_1 - \delta_1} m_2^{-\delta_2}
|f_{I_1 I_2}| |g_{J_1 J_2}|,
\\ 
\mathscr{I}_{2, 2}^N
&:= \sum_{\substack{k_1 \ge 0 \\ j_1 \ge N^{1/8}}}
\sum_{\substack{k_2 \ge 0 \\ 1 \le m_2 \le 2^{k_2}}}
\sum_{\substack{J_1 \in \D_1 \\ J_2 \in \D_2}}
\sum_{\substack{I_1 \in J_1(k_1, j_1) \\ I_2 \in J_2(k_2, 0, m_2)}}
2^{-k_1(\frac{n_1}{2} + \delta_1)} 2^{-k_2\frac{n_2}{2}} j_1^{-n_1 - \delta_1} m_2^{-\delta_2}
|f_{I_1 I_2}| |g_{J_1 J_2}|,
\\ 
\mathscr{I}_{2, 3}^N
&:= \sum_{\substack{k_1 \ge 0 \\ j_1 \ge 1}}
\sum_{\substack{k_2 \ge N \\ 1 \le m_2 \le 2^{k_2}}}
\sum_{J\substack{_1 \in \D_1 \\ J_2 \in \D_2}}
\sum_{\substack{I_1 \in J_1(k_1, j_1) \\ I_2 \in J_2(k_2, 0, m_2)}}
2^{-k_1(\frac{n_1}{2} + \delta_1)}2^{-k_2\frac{n_2}{2}}  j_1^{-n_1 - \delta_1} m_2^{-\delta_2}
|f_{I_1 I_2}| |g_{J_1 J_2}|.
\end{align*} 
From the Cauchy--Schwarz inequality, \eqref{car-IJ}, and \eqref{car-IJK}, it follows that
\begin{align*}
\mathscr{I}_{2, 1}^N
&\le \varepsilon
\sum_{k_1, k_2 \ge 0} \sum_{j_1 \ge 1}
2^{-k_1(\frac{n_1}{2} + \delta_1)} 2^{-k_2\frac{n_2}{2}} j_1^{-n_1 - \delta_1}
\nonumber \\
&\quad\times \bigg(\sum_{\substack{I_1 \in \D_1 \\ I_2 \in \D_2}} \sum_{1 \le m_2 \le 2^{k_2}}
\sum_{\substack{J_1 \in I_1(k_1, j_1) \\ J_2 \in I_2(k_2, 0, m_2)}} |f_{I_1 I_2}|^2 \bigg)^{\frac12}
\nonumber \\
&\quad\times \bigg(\sum_{J\substack{_1 \in \D_1 \\ J_2 \in \D_2}} 
\sum_{1 \le m_2 \le 2^{k_2}} m_2^{-2\delta_2}
\sum_{\substack{I_1 \in J_1(k_1, j_1) \\ I_2 \in J_2(k_2, 0, m_2)}}  
|g_{J_1 J_2}|^2 \bigg)^{\frac12}
\nonumber \\
&\lesssim \varepsilon \sum_{k_1, k_2 \ge 0} \sum_{j_1 \ge 1} 2^{-k_1 \delta_1} j_1^{- \delta_1-1}
\bigg(\sum_{\substack{I_1 \in \D_1 \\ I_2 \in \D_2}} |f_{I_1 I_2}|^2 \bigg)^{\frac12}
\nonumber \\
&\quad\times \bigg(2^{-k_2} \sum_{1 \le m_2 \le 2^{k_2}} m_2^{-2\delta_2} \bigg)^{\frac12}
\bigg(\sum_{\substack{J_1 \in \D_1 \\ J_2 \in \D_2}} |g_{J_1 J_2}|^2 \bigg)^{\frac12}
\nonumber \\
&\lesssim \varepsilon
\sum_{k_1, k_2 \ge 0} \sum_{j_1 \ge 1} 
2^{-k_1 \delta_1} 2^{-k_2 \theta \delta_2} j_1^{- \delta_1-1} 
\|f\|_{L^2} \|g\|_{L^2}
\lesssim \varepsilon \|f\|_{L^2} \|g\|_{L^2},
\end{align*}
where we have used the following estimate
\begin{align}\label{MDT}
2^{-k} \sum_{1 \le m \le 2^k} m^{-2\delta} 
=2^{-k} \bigg[\sum_{1 \le m \le 2^{k \theta}} 
+ \sum_{2^{k\theta} < m \le 2^k}\bigg] m^{-2\delta} 
\le 2^{-k+k\theta} + 2^{-2k \theta \delta}
\lesssim 2^{-2k \theta \delta},  
\end{align}
for all $0<\theta<\frac{1}{1+2\delta}$.
Analogously, when $N \ge N_0$ is large enough,
\begin{align*}
\mathscr{I}_{2, 2}^N
&\lesssim \sum_{k_1, k_2 \ge 0} \sum_{j_1 \ge N^{1/8}} 
2^{-k_1 \delta_1} 2^{-k_2 \theta \delta_2} j_1^{- \delta_1-1} 
\|f\|_{L^2} \|g\|_{L^2}
\\
&\lesssim N^{-\delta_1/8} \|f\|_{L^2} \|g\|_{L^2}
\lesssim \varepsilon \|f\|_{L^2} \|g\|_{L^2},
\end{align*}
and
\begin{align*}
\mathscr{I}_{2, 3}^N
&\lesssim \sum_{k_1 \ge 0} \sum_{k_2 \ge N} \sum_{j_1 \ge 0} 
2^{-k_1 \delta_1} 2^{-k_2 \theta \delta_2} j_1^{- \delta_1-1} 
\|f\|_{L^2} \|g\|_{L^2}
\\
&\lesssim 2^{-N \theta \delta_2} 
\|f\|_{L^2} \|g\|_{L^2}
\lesssim \varepsilon \|f\|_{L^2} \|g\|_{L^2}.
\end{align*}
Hence, the estimates above imply that for all $N \ge N_0$ sufficiently large,
\begin{align*}
|\mathscr{I}_2^N| 
\lesssim \sum_{i=1}^3 \mathscr{I}_{2, i}^N
\lesssim \varepsilon \|f\|_{L^2} \|g\|_{L^2}.
\end{align*}

\subsection{$\rd(I_1, J_1) \ge 1$ and $I_2 \subset J_2$}\label{sec:SI} 
We begin with the case $I_2=J_2$. 

\begin{lemma}\label{lem:THH-33}
Let $I_1, J_1 \in \D_1$ with $\ell(I_1) \le \ell(J_1)$ and $\rd(I_1, J_1) \ge 1$, and let $I_2 \in \D_2$. Then
\begin{align*}
|\langle T(h_{I_1} \otimes h_{I_2}), h_{J_1} \otimes h_{I_2} \rangle| 
\lesssim F_1(I_1, J_1) \frac{\rs(I_1, J_1)^{\frac{n_1}{2} + \delta_1}}{\rd(I_1, J_1)^{n_1 + \delta_1}} 
\widehat{F}_2(I_2). 
\end{align*} 
\end{lemma}

\begin{proof}
We decompose
\begin{align*}
\langle T(h_{I_1} \otimes h_{I_2}), h_{J_1} \otimes h_{I_2} \rangle
= \sum_{I'_2, I''_2 \in \ch(I_2)}
\langle h_{I_2} \rangle_{I'_2} \langle h_{I_2} \rangle_{I''_2}
\big\langle T (h_{I_1} \otimes \mathbf{1}_{I'_2}),
h_{J_1} \otimes \mathbf{1}_{I''_2} \big\rangle.
\end{align*}
Consider first the case $I'_2=I''_2$. By the compact partial kernel representation, the H\"{o}lder condition of $K_{\mathbf{1}_{I'_2}, \mathbf{1}_{I'_2}}$, and \eqref{def:P}, we have
\begin{align*}
|\langle T (h_{I_1} \otimes \mathbf{1}_{I'_2}),
h_{J_1} \otimes \mathbf{1}_{I'_2} \rangle|
&\lesssim |I_1|^{-\frac12} |J_1|^{-\frac12} \mathscr{P}_1(I_1, J_1) \, C(\mathbf{1}_{I'_2}, \mathbf{1}_{I'_2})
\\
&\lesssim F_1(I_1, J_1) \frac{\rs(I_1, J_1)^{\frac{n_1}{2} + \delta_1}}{\rd(I_1, J_1)^{n_1 + \delta_1}} \, 
\widehat{F}_2(I_2) |I_2|.
\end{align*}
If $I'_2 \neq I''_2$, then $I''_2 \subset 3I'_2 \setminus I'_2$. We use the compact full kernel representation and the mixed size-H\"{o}lder condition of $K$ to conclude that
\begin{align*}
|\langle T(h_{I_1} \otimes \mathbf{1}_{I'_2}),
h_{J_1} \otimes \mathbf{1}_{I''_2} \rangle|
&\lesssim |I_1|^{-\frac12} |J_1|^{-\frac12} \mathscr{P}_1(I_1, J_1) \, \mathscr{Q}_2(I'_2)
\\
&\lesssim F_1(I_1, J_1)  \frac{\rs(I_1, J_1)^{\frac{n_1}{2} + \delta_1}}{\rd(I_1, J_1)^{n_1 + \delta_1}}
\, \widetilde{F}_2(I_2) |I_2|,
\end{align*}
provided \eqref{def:P} and \eqref{def:Q}. Thus, these inequalities above give the desired estimate. 
\end{proof}

Next, let us deal with the case $I_2 \subsetneq J_2$. Let $J_{I_2} \in \ch(J_2)$ be the unique dyadic cube such that $I_2 \subset J_{I_2}$. Set $\phi_{I_2 J_2} := (h_{J_2} - \langle h_{J_2} \rangle_{I_2}) \mathbf{1}_{J_{I_2}^c}$. Noting that $h_{J_2} = \phi_{I_2 J_2} + \langle h_{J_2} \rangle_{I_2}$, we split
\begin{align}\label{STH}
\langle T (h_{I_1} \otimes h_{I_2}), h_{J_1} \otimes h_{J_2} \rangle
&=\langle T (h_{I_1} \otimes h_{I_2}), h_{J_1} \otimes \phi_{I_2 J_2}\rangle
\\ \nonumber
&\qquad+ \langle T (h_{I_1} \otimes h_{I_2}), h_{J_1} \otimes 1 \rangle \langle h_{J_2} \rangle_{I_2}. 
\end{align}

\begin{lemma}\label{lem:THH-3}
Let $I_1, J_1 \in \D_1$ with $\ell(I_1) \le \ell(J_1)$ and $\rd(I_1, J_1) \ge 1$, and let $I_2 \subsetneq J_2 \in \D_2$. Then
\begin{align}\label{eq:THH-3}
|\langle T (h_{I_1} \otimes h_{I_2}), h_{J_1} \otimes \phi_{I_2 J_2}\rangle|
\lesssim F_1(I_1, J_1) \frac{\rs(I_1, J_1)^{\frac{n_1}{2} + \delta_1}}{\rd(I_1, J_1)^{n_1 + \delta_1}}
\, \widetilde{F}_2(I_2) \rs(I_2, J_2)^{\frac{n_2}{2}}. 
\end{align}
If we assume in addition that $\d(I_2, J_{I_2}^c) > \ell(I_2)^{\frac12} \ell(J_2)^{\frac12}$, then 
\begin{align}\label{eq:THH-3d}
|\langle T (h_{I_1} \otimes h_{I_2}), h_{J_1} \otimes \phi_{I_2 J_2}\rangle|
\lesssim F_1(I_1, J_1) \frac{\rs(I_1, J_1)^{\frac{n_1}{2} + \delta_1}}{\rd(I_1, J_1)^{n_1 + \delta_1}}
\, \widetilde{F}_2(I_2, J_2) \rs(I_2, J_2)^{\frac{n_2}{2} + \frac{\delta_2}{2}}. 
\end{align}
\end{lemma}

\begin{proof}
We split 
\begin{align}\label{TSS}
\langle T (h_{I_1} \otimes h_{I_2}), h_{J_1} \otimes \phi_{I_2 J_2}\rangle
&= \langle T (h_{I_1} \otimes h_{I_2}), h_{J_1} 
\otimes (\phi_{I_2 J_2} \mathbf{1}_{3I_2 \setminus I_2})\rangle
\\ \nonumber
&\quad+ \langle T (h_{I_1} \otimes h_{I_2}), h_{J_1} 
\otimes (\phi_{I_2 J_2} \mathbf{1}_{(3I_2)^c})\rangle. 
\end{align}
To control the first term, using the compact full kernel representation, the cancellation of $h_{I_1}$, and the mixed size-H\"{o}lder condition of $K$, we arrive at
\begin{align}\label{STH-2}
|\langle &T (h_{I_1} \otimes h_{I_2}), h_{J_1} \otimes 
(\phi_{I_2 J_2} \mathbf{1}_{3 I_2 \setminus I_2}) \rangle|
\\ \nonumber
&\lesssim |I_1|^{-\frac12} |J_1|^{-\frac12} \mathscr{P}_1(I_1, J_1) \,
|I_2|^{-\frac12} |J_2|^{-\frac12} \mathscr{Q}_2(I_2)
\\ \nonumber
&\lesssim F_1(I_1, J_1) \frac{\rs(I_1, J_1)^{\frac{n_1}{2} + \delta_1}}{\rd(I_1, J_1)^{n_1 + \delta_1}}
\widetilde{F}_2(I_2) \rs(I_2, J_2)^{\frac{n_2}{2}},
\end{align}
where \eqref{def:P} and \eqref{def:Q} have been used. For the second term, it follows from the compact full kernel representation, the cancellation of $h_{I_1}$ and $h_{I_2}$, the H\"{o}lder condition of $K$, \eqref{def:P}, and \eqref{def:R} that   
\begin{align}\label{STH-3}
|\langle &T (h_{I_1} \otimes h_{I_2}), h_{J_1} 
\otimes (\phi_{I_2 J_2} \mathbf{1}_{(3 I_2)^c}) \rangle|
\\ \nonumber
&\lesssim |I_1|^{-\frac12} |J_1|^{-\frac12} \mathscr{P}_1(I_1, J_1) \,
|I_2|^{-\frac12} |J_2|^{-\frac12} \mathscr{R}_2(I_2)
\\ \nonumber
&\lesssim F_1(I_1, J_1) \frac{\rs(I_1, J_1)^{\frac{n_1}{2} + \delta_1}}{\rd(I_1, J_1)^{n_1 + \delta_1}}
\widetilde{F}_2(I_2) \rs(I_2, J_2)^{\frac{n_2}{2}}. 
\end{align}
Therefore, \eqref{eq:THH-3} follows from \eqref{TSS}--\eqref{STH-3}. Additionally, \eqref{eq:THH-3d} is just a consequence the compact full kernel representation, the cancellation of $h_{I_1}$ and $h_{I_2}$, the H\"{o}lder condition of $K$, and \eqref{def:RIJ}. 
\end{proof}

Denote 
\begin{align*}
\mathscr{I}_{3, 0}^N
:= \bigg|\sum_{\substack{J_1 \not\in \D_1(2N) \\ \text{or } J_2 \not\in \D_2(2N)}}
\sum_{\substack{I_1 \in \D_1 \\ \ell(I_1) \le \ell(J_1) \\ \rd(I_1, J_1) \ge 1}}
\sum_{\substack{I_2 \in \D_2 \\ I_2 \subsetneq J_2}}
\langle T (h_{I_1} \otimes h_{I_2}), h_{J_1} \otimes 1 \rangle \langle h_{J_2} \rangle_{I_2} \, 
f_{I_1 I_2} \, g_{J_1 J_2} \bigg|. 
\end{align*}

\begin{lemma}\label{lem:hjpi}
For any $\varepsilon>0$, there exists an $N_0>1$ so that for all $N \ge N_0$, 
\begin{align*}
\mathscr{I}_{3, 0}^N
\lesssim \varepsilon \|f\|_{L^2} \|g\|_{L^2}. 
\end{align*}
\end{lemma}

\begin{proof}
We split  
\begin{align*}
\mathscr{I}_{3, 0}^N
\le \bigg|\sum_{\substack{J_1 \not\in \D_1(2N) \\ J_2 \in \D_2}}
\sum_{\substack{I_1 \in \D_1 \\ \ell(I_1) \le \ell(J_1) \\ \rd(I_1, J_1) \ge 1}}
\sum_{\substack{I_2 \in \D_2 \\ I_2 \subsetneq J_2}} \cdots \bigg|
+ \bigg|\sum_{\substack{J_1 \in \D_1(2N) \\ J_2 \not\in \D_2(2N)}}
\sum_{\substack{I_1 \in \D_1 \\ \ell(I_1) \le \ell(J_1) \\ \rd(I_1, J_1) \ge 1}}
\sum_{\substack{I_2 \in \D_2 \\ I_2 \subsetneq J_2}} \cdots \bigg|
=: \mathscr{I}_{3, 0}^{N, 1} + \mathscr{I}_{3, 0}^{N, 2}. 
\end{align*}
Write $b_{I_1 J_1} := \langle T^*(h_{J_1} \otimes 1), h_{I_1} \rangle$. Observe that 
\begin{align}\label{JDN-1}
&\sum_{J_2 \not\in \D_2(2N)}
\sum_{\substack{I_2 \in \D_2 \\ I_2 \subsetneq J_2}}
\langle T (h_{I_1} \otimes h_{I_2}), h_{J_1} \otimes 1 \rangle \langle h_{J_2} \rangle_{I_2} \, 
f_{I_1 I_2} \, g_{J_1 J_2}
\\ \nonumber
&= \sum_{I_2 \in \D_2} \langle b_{I_1 J_1}, h_{I_2} \rangle  
\bigg(\sum_{J_2 \in \D_2: J_2 \supsetneq I_2} 
\big\langle P_{2N}^{\perp} (\langle g, h_{J_1} \rangle), h_{J_2} \big\rangle 
\langle h_{J_2} \rangle_{I_2} \bigg)  
\big\langle \langle f, h_{I_1}\rangle, h_{I_2} \big\rangle 
\\ \nonumber
&= \sum_{I_2 \in \D_2} \langle b_{I_1 J_1}, h_{I_2} \rangle 
\big\langle P_{2N}^{\perp} (\langle g, h_{J_1} \rangle) \big\rangle_{I_2} 
\big\langle \langle f, h_{I_1}\rangle, h_{I_2} \big\rangle
\\ \nonumber
&= \big\langle \langle f, h_{I_1}\rangle, \Pi_{b_{I_1 J_1}} 
P_{2N}^{\perp} (\langle g, h_{J_1} \rangle) \big\rangle
= \big\langle h_{J_1} \otimes P_{2N}^{\perp} \Pi_{b_{I_1 J_1}}^* (\langle f, h_{I_1}\rangle),  
g \big\rangle, 
\end{align}
and similarly, 
\begin{align}\label{JDN-2}
\sum_{\substack{I_2, J_2 \in \D_2 \\ I_2 \subsetneq J_2}}
\langle T (h_{I_1} \otimes h_{I_2}), h_{J_1} \otimes 1 \rangle \langle h_{J_2} \rangle_{I_2} \, 
f_{I_1 I_2} \, g_{J_1 J_2}
= \big\langle h_{J_1} \otimes \Pi_{b_{I_1 J_1}}^* (\langle f, h_{I_1}\rangle),  
g \big\rangle. 
\end{align}
By \eqref{JDN-2}, the Cauchy--Schwarz inequality, the orthogonality of $\{h_{J_1}\}_{J_1 \in \D_1}$, and Minkowski's inequality, there holds 
\begin{align}\label{Nbb}
\mathscr{I}_{3, 0}^{N, 1}
&= \bigg|\sum_{J_1 \not\in \D_1(2N)}
\sum_{\substack{I_1 \in \D_1 \\ \ell(I_1) \le \ell(J_1) \\ \rd(I_1, J_1) \ge 1}} 
\big\langle h_{J_1} \otimes \Pi_{b_{I_1 J_1}}^* (\langle f, h_{I_1}\rangle),  
g \big\rangle \bigg| 
\\ \nonumber
&\le \bigg\|\sum_{J_1 \not\in \D_1(2N)}
\sum_{\substack{I_1 \in \D_1 \\ \ell(I_1) \le \ell(J_1) \\ \rd(I_1, J_1) \ge 1}} 
h_{J_1} \otimes \Pi_{b_{I_1 J_1}}^* (\langle f, h_{I_1}\rangle) \bigg\|_{L^2} 
\, \|g\|_{L^2}
\\ \nonumber
&= \Bigg[\sum_{J_1 \not\in \D_1(2N)} \bigg\|
\sum_{\substack{I_1 \in \D_1 \\ \ell(I_1) \le \ell(J_1) \\ \rd(I_1, J_1) \ge 1}} 
\Pi_{b_{I_1 J_1}}^* (\langle f, h_{I_1}\rangle) \bigg\|_{L^2(\R^{n_2})}^2  
\Bigg]^{\frac12} \, \|g\|_{L^2}
\\ \nonumber
&\le \Bigg[ \sum_{J_1 \not\in \D_1(2N)} 
\bigg(\sum_{\substack{I_1 \in \D_1 \\ \ell(I_1) \le \ell(J_1) \\ \rd(I_1, J_1) \ge 1}} 
\big\|\Pi_{b_{I_1 J_1}}^* (\langle f, h_{I_1}\rangle) \big\|_{L^2(\R^{n_2})} \bigg)^2 
\Bigg]^{\frac12} \, \|g\|_{L^2}
\\ \nonumber
&\le \Bigg[ \sum_{J_1 \not\in \D_1(2N)} 
\bigg(\sum_{\substack{I_1 \in \D_1 \\ \ell(I_1) \le \ell(J_1) \\ \rd(I_1, J_1) \ge 1}} 
\|\Pi_{b_{I_1 J_1}}^* \|_{L^2 \to L^2} 
\|\langle f, h_{I_1}\rangle\|_{L^2(\R^{n_2})} \bigg)^2 
\Bigg]^{\frac12} \, \|g\|_{L^2}
\\ \nonumber
&\lesssim \Bigg[ \sum_{J_1 \not\in \D_1(2N)} 
\bigg(\sum_{\substack{k_1 \ge 0 \\ j_1 \ge 1}} \sum_{I_1 \in J_1(k_1, j_1)}  
\|b_{I_1 J_1}\|_{\BMO_{\D_2}} \|\langle f, h_{I_1}\rangle\|_{L^2(\R^{n_2})} \bigg)^2 
\Bigg]^{\frac12} \, \|g\|_{L^2}. 
\end{align}
Then invoking \eqref{B1} and Lemma \ref{lem:FE} part \eqref{list:FE-1}, we obtain that for any $N \ge N_0$, 
\begin{align*}
\mathscr{I}_{3, 0}^{N, 1}
\lesssim \mathscr{I}_{3, 0}^{N, 1, 1} + \mathscr{I}_{3, 0}^{N, 1, 2}, 
\end{align*}
where 
\begin{align*}
\mathscr{I}_{3, 0}^{N, 1, 1}
&:= \varepsilon \Bigg[ \sum_{J_1 \in \D_1} 
\bigg(\sum_{\substack{k_1 \ge 0 \\ j_1 \ge 1}}  \sum_{I_1 \in J_1(k_1, j_1)} 
2^{-k_1(\frac{n_1}{2} + \delta_1)} j_1^{-n_1 - \delta_1} 
\|\langle f, h_{I_1}\rangle\|_{L^2(\R^{n_2})} \bigg)^2 
\Bigg]^{\frac12} \, \|g\|_{L^2}, 
\\ 
\mathscr{I}_{3, 0}^{N, 1, 2}
&:= \Bigg[ \sum_{J_1 \in \D_1} 
\bigg(\sum_{\substack{k_1 \ge 0 \\ j_1 \ge N^{1/8}}} 
\sum_{I_1 \in J_1(k_1, j_1)}
2^{-k_1(\frac{n_1}{2} + \delta_1)} j_1^{-n_1 - \delta_1} 
\|\langle f, h_{I_1}\rangle\|_{L^2(\R^{n_2})} \bigg)^2 
\Bigg]^{\frac12} \, \|g\|_{L^2}. 
\end{align*}
By the Cauchy--Schwarz inequality, we have 
\begin{align}\label{PJK}
&\bigg(\sum_{\substack{k_1 \ge 0 \\ j_1 \ge 1}}  \sum_{I_1 \in J_1(k_1, j_1)} 
2^{-k_1(\frac{n_1}{2} + \delta_1)} j_1^{-n_1 - \delta_1} 
\|\langle f, h_{I_1}\rangle\|_{L^2(\R^{n_2})} \bigg)^2 
\\ \nonumber
&\le \bigg(\sum_{\substack{k_1 \ge 0 \\ j_1 \ge 1}}  \sum_{I_1 \in J_1(k_1, j_1)} 
2^{-k_1(n_1 + \delta_1)} j_1^{-n_1 - \delta_1} \bigg) 
\\ \nonumber
&\quad\times \bigg(\sum_{\substack{k_1 \ge 0 \\ j_1 \ge 1}}  \sum_{I_1 \in J_1(k_1, j_1)} 
2^{-k_1 \delta_1} j_1^{-n_1 - \delta_1} 
\|\langle f, h_{I_1}\rangle\|_{L^2(\R^{n_2})}^2 \bigg) 
\\ \nonumber
&\lesssim \sum_{\substack{k_1 \ge 0 \\ j_1 \ge 1}}  \sum_{I_1 \in J_1(k_1, j_1)} 
2^{-k_1 \delta_1} j_1^{-n_1 - \delta_1} 
\|\langle f, h_{I_1}\rangle\|_{L^2(\R^{n_2})}^2
\end{align}
provided that 
\begin{align*}
\sum_{\substack{k_1 \ge 0 \\ j_1 \ge 1}}  \sum_{I_1 \in J_1(k_1, j_1)} 
2^{-k_1(n_1 + \delta_1)} j_1^{-n_1 - \delta_1}
\lesssim \sum_{\substack{k_1 \ge 0 \\ j_1 \ge 1}} 
2^{-k_1 \delta_1} j_1^{- \delta_1 -1}
\lesssim 1. 
\end{align*}
Furthermore, \eqref{PJK} and \eqref{car-IJ} give 
\begin{align*}
\mathscr{I}_{3, 0}^{N, 1, 1}
&\lesssim \varepsilon \Bigg[ 
\sum_{\substack{k_1 \ge 0 \\ j_1 \ge 1}} 2^{-k_1 \delta_1} j_1^{-n_1 - \delta_1}  
\sum_{I_1 \in \D_1}  \sum_{J_1 \in I_1(k_1, j_1)} 
\|\langle f, h_{I_1} \rangle\|_{L^2(\R^{n_2})}^2   
\Bigg]^{\frac12} \, \|g\|_{L^2} 
\\
&\lesssim \varepsilon \Bigg[ 
\sum_{\substack{k_1 \ge 0 \\ j_1 \ge 1}} 2^{-k_1 \delta_1} j_1^{- \delta_1 -1}  
\sum_{I_1 \in \D_1} \|\langle f, h_{I_1} \rangle\|_{L^2(\R^{n_2})}^2   
\Bigg]^{\frac12} \, \|g\|_{L^2} 
\lesssim \varepsilon \|f\|_{L^2} \|g\|_{L^2}. 
\end{align*}
Much as in the same way, there holds for any $N \ge N_0$ sufficiently large, 
\begin{align*}
\mathscr{I}_{3, 0}^{N, 1, 2}
&\lesssim \Bigg[ 
\sum_{\substack{k_1 \ge 0 \\ j_1 \ge N^{1/8}}} 2^{-k_1 \delta_1} j_1^{- \delta_1 -1}  
\sum_{I_1 \in \D_1} \|\langle f, h_{I_1} \rangle\|_{L^2(\R^{n_2})}^2   
\Bigg]^{\frac12} \, \|g\|_{L^2} 
\\
&\lesssim N^{-\delta_1/16} \|f\|_{L^2} \|g\|_{L^2} 
\le \varepsilon \|f\|_{L^2} \|g\|_{L^2}. 
\end{align*}
Hence, 
\begin{align}\label{30N1}
\mathscr{I}_{3, 0}^{N, 1}
\lesssim \mathscr{I}_{3, 0}^{N, 1, 1} + \mathscr{I}_{3, 0}^{N, 1, 2} 
\lesssim \varepsilon \|f\|_{L^2} \|g\|_{L^2}. 
\end{align}

Let us turn to estimate $\mathscr{I}_{3, 0}^{N, 2}$. In light of \eqref{JDN-1}, there holds 
\begin{align*}
\mathscr{I}_{3, 0}^{N, 2}
= \bigg|\sum_{J_1 \in \D_1(2N)}
\sum_{\substack{I_1 \in \D_1 \\ \ell(I_1) \le \ell(J_1) \\ \rd(I_1, J_1) \ge 1}} 
\big\langle h_{J_1} \otimes P_{2N}^{\perp} \Pi_{b_{I_1 J_1}}^* (\langle f, h_{I_1}\rangle),  
g \big\rangle \bigg| 
\le \mathscr{I}_{3, 0}^{N, 2, 1} + \mathscr{I}_{3, 0}^{N, 2, 2}, 
\end{align*}
where 
\begin{align*}
\mathscr{I}_{3, 0}^{N, 2, 1}
:= \bigg|\sum_{J_1 \in \D_1(2N)}
\sum_{\substack{I_1 \in \D_1 \\ \ell(I_1) \le \ell(J_1) \\ \rd(I_1, J_1) \ge 1}} 
\big\langle h_{J_1} \otimes P_{2N}^{\perp} \Pi_{P_N b_{I_1 J_1}}^* (\langle f, h_{I_1}\rangle),  
g \big\rangle \bigg|, 
\\
\mathscr{I}_{3, 0}^{N, 2, 2}
:= \bigg|\sum_{J_1 \in \D_1(2N)}
\sum_{\substack{I_1 \in \D_1 \\ \ell(I_1) \le \ell(J_1) \\ \rd(I_1, J_1) \ge 1}} 
\big\langle h_{J_1} \otimes P_{2N}^{\perp} \Pi_{P_N^{\perp} b_{I_1 J_1}}^* (\langle f, h_{I_1}\rangle),  
g \big\rangle \bigg|.  
\end{align*}
Note that 
\begin{align}\label{PPbb}
\big\|P_{2N}^{\perp} \Pi_{P_N^{\perp} b_{I_1 J_1}}^*\|_{L^2 \to L^2} 
\le \big\| \Pi_{P_N^{\perp} b_{I_1 J_1}}^*\|_{L^2 \to L^2} 
\lesssim \big\|P_N^{\perp} b_{I_1 J_1}\|_{\BMO_{\D_2}},  
\end{align}
where the implicit constant is independent of $N$, $I_1$, and $J_1$. Then by \eqref{PB1}, \eqref{PPbb}, and Lemma \ref{lem:FE} part \eqref{list:FE-4}, an argument as in \eqref{Nbb} yields that for any $N \ge N_0$ large enough, 
\begin{align}\label{30N22}
\mathscr{I}_{3, 0}^{N, 2, 2}
&\le \Bigg[ \sum_{J_1 \in \D_1} 
\bigg(\sum_{\substack{I_1 \in \D_1 \\ \ell(I_1) \le \ell(J_1) \\ \rd(I_1, J_1) \ge 1}} 
\big\|P_{2N}^{\perp} \Pi_{P_N^{\perp} b_{I_1 J_1}}^*\|_{L^2 \to L^2} 
\|\langle f, h_{I_1}\rangle\|_{L^2(\R^{n_2})} \bigg)^2 
\Bigg]^{\frac12} \, \|g\|_{L^2}
\\ \nonumber
& \lesssim \mathscr{I}_{3, 0}^{N, 1, 1}  
\lesssim \varepsilon \|f\|_{L^2} \|g\|_{L^2}. 
\end{align}
On the other hand, we see that 
\begin{align*}
\mathscr{I}_{3, 0}^{N, 2, 1}
= \bigg|\sum_{\substack{J_1 \in \D_1(2N) \\ J_2 \not\in \D_2(2N)}}
\sum_{\substack{I_1 \in \D_1 \\ \ell(I_1) \le \ell(J_1) \\ \rd(I_1, J_1) \ge 1}} 
\sum_{\substack{I_2 \in \D_2(N) \\ I_2 \subsetneq J_2}}
\langle b_{I_1 J_1}, h_{I_2} \rangle \langle h_{J_2} \rangle_{I_2} 
f_{I_1 I_2} g_{J_1 J_2}\bigg|,
\end{align*}
and 
\begin{align*}
|\langle b_{I_1 J_1}, h_{I_2} \rangle| |\langle h_{J_2} \rangle_{I_2}|
\lesssim \|b_{I_1 J_1}\|_{\BMO_{\D_2}} \|h_{I_2}\|_{\mathrm{H}^1_{\D_2}} |J_2|^{-\frac12}. 
\end{align*}
Observe that for each $i=1, 2$, 
\begin{align}\label{IDN-1}
I_i \in \D_i(N) \text{ and }   
J_i \not\in \D_i(2N) \text{ with } I_i \subsetneq J_i
\quad \Longrightarrow \quad 
\ell(J_i) > 2^{2N}, 
\end{align}
and 
\begin{align}\label{IDN-2}
\# \D_i^{\ell_i}(N) 
:= \#\big\{I_i \in \D_i(N): \ell(I_i) = 2^{\ell_i}\big\} 
\lesssim N^{n_i} 2^{(N- \ell_i) n_i}, \quad\forall \,  \ell_i \le N. 
\end{align}
Hence, \eqref{B1}, \eqref{car-IJ}, and \eqref{IDN-1}--\eqref{IDN-2} imply 
\begin{align}\label{30N21}
\mathscr{I}_{3, 0}^{N, 2, 1}
&\lesssim \sum_{\substack{k_1 \ge 0 \\ j_1 \ge 1}}
\sum_{-N \le \ell_2 \le N} \sum_{k_2 > 2N - \ell_2} 
2^{-k_1(\frac{n_1}{2} + \delta_1)} 2^{-k_2 \frac{n_2}{2}} j_1^{-n_1 - \delta_1} 
\\ \nonumber 
&\quad\times \sum_{\substack{J_1 \in \D_1 \\ I_1 \in J_1(k_1, j_1)}}
\sum_{I_2 \in \D_2^{\ell_2}(N)} |f_{I_1 I_2}| |g_{J_1 I_2^{(k_2)}}|
\\ \nonumber
&\le \sum_{\substack{k_1 \ge 0 \\ j_1 \ge 1}}
\sum_{-N \le \ell_2 \le N} \sum_{k_2 > 2N - \ell_2} 
2^{-k_1(\frac{n_1}{2} + \delta_1)} 2^{-k_2 \frac{n_2}{2}} j_1^{-n_1 - \delta_1} 
\\ \nonumber
&\quad\times \bigg(\sum_{\substack{I_1 \in \D_1 \\ I_2 \in \D_2}}
\sum_{J_1 \in I_1(k_1, j_1)} |f_{I_1 I_2}|^2 \bigg)^{\frac12}  
\bigg(\sum_{\substack{J_1 \in \D_1 \\ J_2 \in \D_2}}
\sum_{\substack{I_1 \in J_1(k_1, j_1) \\ I_2 \in \D_2^{\ell_2}(N) \\ I_2^{(k_2)} = J_2}} 
|g_{J_1 J_2}|^2 \bigg)^{\frac12} 
\\ \nonumber
&\lesssim \sum_{\substack{k_1 \ge 0 \\ j_1 \ge 1}}
\sum_{-N \le \ell_2 \le N} \sum_{k_2 > 2N - \ell_2} 
2^{-k_1 \delta_1} 2^{-k_2 \frac{n_2}{2}} j_1^{- \delta_1 - 1} 
N^{\frac{n_2}{2}} 2^{(N - \ell_2) \frac{n_2}{2}}
\|f\|_{L^2} \|g\|_{L^2}
\\ \nonumber
&\lesssim N^{1 + \frac{n_2}{2}} 2^{-N \frac{n_2}{2}}
\|f\|_{L^2} \|g\|_{L^2}
\le \varepsilon \|f\|_{L^2} \|g\|_{L^2}, 
\end{align}
whenever $N \ge 2$ large enough. Collecting \eqref{30N22} and \eqref{30N21}, we obtain 
\begin{align}\label{30N2}
\mathscr{I}_{3, 0}^{N, 2} 
\le \mathscr{I}_{3, 0}^{N, 2, 1} + \mathscr{I}_{3, 0}^{N, 2, 2} 
\lesssim \varepsilon \|f\|_{L^2} \|g\|_{L^2}.
\end{align}
Consequently, it follows from \eqref{30N1} and \eqref{30N2} that 
\begin{align*}
\mathscr{I}_{3, 0}^N 
\le \mathscr{I}_{3, 0}^{N, 1} + \mathscr{I}_{3, 0}^{N, 2} 
\lesssim \varepsilon \|f\|_{L^2} \|g\|_{L^2}. 
\end{align*}
This completes the proof. 
\end{proof}

Recall that for every $i=1, 2$, $k_i \ge 1$, and $I_i \in \D_i$, $I_i^{(k_i)}$ denotes the unique dyadic cube in $\D_i$ such that $I_i \subsetneq I_i^{(k_i)}$ and $\ell(I_i^{(k_i)}) = 2^{k_i} \ell(I_i)$. Write $I_i^{(0)} := I_i$. 
Given $i=1, 2$, $k_i \ge 0$, and $J_i \in \D_i$, let 
\begin{align*}
\mathcal{G}_i^{k_i}(J_i) 
& := \big\{I_i \in \D_i^{k_i}(J_i): \, \d(I_i, \sk(J_i)) > \ell(I_i)^{\frac12} \ell(J_i)^{\frac12} \big\}, 
\\
\mathcal{B}_i^{k_i}(J_i) 
& := \big\{I_i \in \D_i^{k_i}(J_i): \, \d(I_i, \sk(J_i)) \le \ell(I_i)^{\frac12} \ell(J_i)^{\frac12} \big\}, 
\end{align*}
where 
\begin{align*}
\D_i^{k_i}(J_i) :=\{I_i \in \D_i: I_i^{(k_i)} = J_i\}
\quad\text{ and }\quad 
\sk(J_i) := \bigcup_{J'_i \in \ch(J_i)} \partial J'_i. 
\end{align*}
Observe that for every $i=1, 2$, $k_i \ge 0$, and $J_i \in \D_i$, 
\begin{align}\label{GB}
\# \mathcal{G}_i^{k_i}(J_i)
\le \# \D_i^{k_i}(J_i) = 2^{k_i n_i}
\quad\text{ and }\quad 
\# \mathcal{B}_i^{k_i}(J_i) \lesssim 2^{k_i(n_i - \frac12)}. 
\end{align}
Note that $J_i \in \mathcal{B}_i^0(J_i)$ for each $i=1, 2$. By Lemma \ref{lem:THH-33}, \eqref{STH}, Lemma \ref{lem:THH-3}, and Lemma \ref{lem:FE} parts \eqref{list:FE-1} and  \eqref{list:FE-3}--\eqref{list:FE-4}, we obtain, for all $N \ge N_0$,
\begin{align*}
|\mathscr{I}_3^N|
= \bigg|\sum_{\substack{k_1 \ge 0 \\ k_2 \ge 0}} \sum_{j_1 \ge 1}
\sum_{\substack{J_1 \not\in \D_1(2N) \\ \text{or } J_2 \not\in \D_2(2N)}}
\sum_{\substack{I_1 \in J_1(k_1, j_1) \\ I_2 \in \D_2^{k_2}(J_2)}}
\G_{J_1 J_2}^{I_1 I_2} \, f_{I_1 I_2} \, g_{J_1 J_2} \bigg|
\lesssim \sum_{i=0}^6 \mathscr{I}_{3, i}^N,
\end{align*}
where
\begin{align*}
\mathscr{I}_{3, 1}^N
&:= \varepsilon \sum_{\substack{k_1 \ge 0 \\ k_2 \ge 0}} \sum_{j_1 \ge 1}
\sum_{\substack{J_1 \in \D_1 \\ J_2 \in \D_2}} 
\sum_{\substack{I_1 \in J_1(k_1, j_1) \\ I_2 \in \mathcal{B}_2^{k_2}(J_2)}}
2^{-k_1(\frac{n_1}{2}+\delta_1)} 2^{-k_2 \frac{n_2}{2}}
j_1^{-n_1-\delta_1} |f_{I_1 I_2}| |g_{J_1 J_2}|,
\\
\mathscr{I}_{3, 2}^N
&:= \sum_{\substack{k_1 \ge 0 \\ k_2 \ge 0}} \sum_{j_1 \ge N^{1/8}} 
\sum_{\substack{J_1 \in \D_1 \\ J_2 \in \D_2}}
\sum_{\substack{I_1 \in J_1(k_1, j_1) \\ I_2 \in \mathcal{B}_2^{k_2}(J_2)}}
2^{-k_1(\frac{n_1}{2}+\delta_1)} 2^{-k_2 \frac{n_2}{2}}
j_1^{-n_1-\delta_1} |f_{I_1 I_2}| |g_{J_1 J_2}|,
\\
\mathscr{I}_{3, 3}^N
&:= \sum_{\substack{k_1 \ge 0 \\ k_2 \ge N}} \sum_{j_1 \ge 1} 
\sum_{\substack{J_1 \in \D_1 \\ J_2 \in \D_2}}
\sum_{\substack{I_1 \in J_1(k_1, j_1) \\ I_2 \in \mathcal{B}_2^{k_2}(J_2)}} 
2^{-k_1(\frac{n_1}{2}+\delta_1)} 2^{-k_2 \frac{n_2}{2}}
j_1^{-n_1-\delta_1} |f_{I_1 I_2}| |g_{J_1 J_2}|,
\\
\mathscr{I}_{3, 4}^N
&:= \varepsilon \sum_{\substack{k_1 \ge 0 \\ k_2 \ge 1}} \sum_{j_1 \ge 1}
\sum_{\substack{J_1 \in \D_1 \\ J_2 \in \D_2}} 
\sum_{\substack{I_1 \in J_1(k_1, j_1) \\ I_2 \in \mathcal{G}_2^{k_2}(J_2)}}
2^{-k_1(\frac{n_1}{2}+\delta_1)} 2^{-k_2 (\frac{n_2}{2} + \frac{\delta_2}{2})}
j_1^{-n_1-\delta_1} |f_{I_1 I_2}| |g_{J_1 J_2}|,
\\
\mathscr{I}_{3, 5}^N
&:= \sum_{\substack{k_1 \ge 0 \\ k_2 \ge 1}} \sum_{j_1 \ge N^{1/8}} 
\sum_{\substack{J_1 \in \D_1 \\ J_2 \in \D_2}}
\sum_{\substack{I_1 \in J_1(k_1, j_1) \\ I_2 \in \mathcal{G}_2^{k_2}(J_2)}}
2^{-k_1(\frac{n_1}{2}+\delta_1)} 2^{-k_2 (\frac{n_2}{2} + \frac{\delta_2}{2})}
j_1^{-n_1-\delta_1} |f_{I_1 I_2}| |g_{J_1 J_2}|,
\\ 
\mathscr{I}_{3, 6}^N
&:= \sum_{\substack{k_1 \ge 0 \\ k_2 \ge N}} \sum_{j_1 \ge 1} 
\sum_{\substack{J_1 \in \D_1 \\ J_2 \in \D_2}}
\sum_{\substack{I_1 \in J_1(k_1, j_1) \\ I_2 \in \mathcal{G}_2^{k_2}(J_2)}}
2^{-k_1(\frac{n_1}{2}+\delta_1)} 2^{-k_2 (\frac{n_2}{2} + \frac{\delta_2}{2})}
j_1^{-n_1-\delta_1} |f_{I_1 I_2}| |g_{J_1 J_2}|. 
\end{align*}
It follows from the Cauchy--Schwarz inequality, \eqref{car-IJ}, and \eqref{GB} that
\begin{align*}
\mathscr{I}_{3, 1}^N
&\le \varepsilon \sum_{k_1, k_2 \ge 0} \sum_{j_1 \ge 1} 
2^{-k_1(\frac{n_1}{2}+\delta_1)} 2^{-k_2 \frac{n_2}{2}} j_1^{-n_1-\delta_1}
\\
&\quad\times \bigg(\sum_{\substack{I_1 \in \D_1 \\ I_2 \in \D_2}}
\sum_{\substack{J_1 \in I_1(k_1, j_1) \\ J_2 = I_2^{(k_2)}}}
|f_{I_1 I_2}|^2 \bigg)^{\frac12}
\bigg(\sum_{\substack{J_1 \in \D_1 \\ J_2 \in \D_2}}
\sum_{\substack{I_1 \in J_1(k_1, j_1) \\ I_2 \in \mathcal{B}_2^{k_2}(J_2)}} 
|g_{J_1 J_2}|^2 \bigg)^{\frac12}
\\
&\lesssim \varepsilon \sum_{k_1, k_2 \ge 0} \sum_{j_1 \ge 1}
2^{-k_1 \delta_1} 2^{-\frac{k_2}{4}} j_1^{-\delta_1-1}
\bigg(\sum_{\substack{I_1 \in \D_1 \\ I_2 \in \D_2}} |f_{I_1 I_2}|^2 \bigg)^{\frac12}
\bigg(\sum_{\substack{J_1 \in \D_1 \\ J_2 \in \D_2}} |g_{J_1 J_2}|^2 \bigg)^{\frac12}
\\
&\lesssim \varepsilon \|f\|_{L^2} \|g\|_{L^2}, 
\end{align*}
and 
\begin{align*}
\mathscr{I}_{3, 4}^N
&\le \varepsilon \sum_{k_1, k_2 \ge 0} \sum_{j_1 \ge 1} 
2^{-k_1(\frac{n_1}{2}+\delta_1)} 
2^{-k_2 (\frac{n_2}{2} + \frac{\delta_2}{2})} j_1^{-n_1-\delta_1}
\\
&\quad\times \bigg(\sum_{\substack{I_1 \in \D_1 \\ I_2 \in \D_2}}
\sum_{\substack{J_1 \in I_1(k_1, j_1) \\ J_2 = I_2^{(k_2)}}} 
|f_{I_1 I_2}|^2 \bigg)^{\frac12}
\bigg(\sum_{\substack{J_1 \in \D_1 \\ J_2 \in \D_2}}
\sum_{\substack{I_1 \in J_1(k_1, j_1) \\ I_2 \in \mathcal{G}_2^{k_2}(J_2)}}
|g_{J_1 J_2}|^2 \bigg)^{\frac12}
\\
&\lesssim \varepsilon \sum_{k_1, k_2 \ge 0} \sum_{j_1 \ge 1}
2^{-k_1 \delta_1} 2^{-k_2 \frac{\delta_2}{2}} j_1^{-\delta_1-1}
\bigg(\sum_{\substack{I_1 \in \D_1 \\ I_2 \in \D_2}} |f_{I_1 I_2}|^2 \bigg)^{\frac12}
\bigg(\sum_{\substack{J_1 \in \D_1 \\ J_2 \in \D_2}} |g_{J_1 J_2}|^2 \bigg)^{\frac12}
\\
&\lesssim \varepsilon \|f\|_{L^2} \|g\|_{L^2}
\le \varepsilon \|f\|_{L^2} \|g\|_{L^2}.
\end{align*}
Similarly, for all $N \ge N_0$ large enough, 
\begin{align*}
\mathscr{I}_{3, 2}^N
&\lesssim \sum_{k_1, k_2 \ge 0} \sum_{j_1 \ge N^{1/8}} 
2^{-k_1 \delta_1} 2^{-\frac{k_2}{4}} j_1^{-\delta_1-1}
\|f\|_{L^2} \|g\|_{L^2}
\\
&\lesssim N^{-\delta_1/8} \|f\|_{L^2} \|g\|_{L^2}
\le \varepsilon \|f\|_{L^2} \|g\|_{L^2},
\end{align*}
\begin{align*}
\mathscr{I}_{3, 3}^N
&\lesssim \sum_{k_1 \ge 0, k_2 \ge N} \sum_{j_1 \ge 1} 
2^{-k_1 \delta_1} 2^{-\frac{k_2}{4}} j_1^{-\delta_1-1}
\|f\|_{L^2} \|g\|_{L^2}
\\
&\lesssim 2^{-N/4} \|f\|_{L^2} \|g\|_{L^2}
\le \varepsilon \|f\|_{L^2} \|g\|_{L^2},
\end{align*}
\begin{align*}
\mathscr{I}_{3, 5}^N
&\lesssim \sum_{k_1, k_2 \ge 0} \sum_{j_1 \ge N^{1/8}} 
2^{-k_1 \delta_1} 2^{- k_2 \frac{\delta_2}{2}} j_1^{-\delta_1-1}
\|f\|_{L^2} \|g\|_{L^2} 
\\
&\lesssim N^{-\delta_1/8} \|f\|_{L^2} \|g\|_{L^2}
\le \varepsilon \|f\|_{L^2} \|g\|_{L^2},
\end{align*}
and 
\begin{align*}
\mathscr{I}_{3, 6}^N
&\lesssim \sum_{k_1 \ge 0, k_2 \ge N} \sum_{j_1 \ge 1} 
 2^{-k_1 \delta_1} 2^{-k_2 \frac{\delta_2}{2}} j_1^{-\delta_1-1}
\|f\|_{L^2} \|g\|_{L^2}
\\
&\lesssim 2^{-N\delta_2/2} \|f\|_{L^2} \|g\|_{L^2}
\le \varepsilon \|f\|_{L^2} \|g\|_{L^2}. 
\end{align*}
Consequently, collecting Lemma \ref{lem:hjpi} and these estimates above, we have proved that for any $N \ge N_0$ sufficiently large,
\begin{align*}
|\mathscr{I}_3^N|
\lesssim \sum_{i=0}^6 \mathscr{I}_{3, i}^N
\lesssim \varepsilon \|f\|_{L^2} \|g\|_{L^2}.
\end{align*}

\black 

\subsection{$\rd(I_i, J_i) < 1$ and $I_i \cap J_i = \emptyset$, $i=1, 2$}

\begin{lemma}\label{lem:THH-5}
Let $I_i, J_i \in \D_i$ with $\ell(I_i) \le \ell(J_i)$ such that $\rd(I_i, J_i) < 1$ and $I_i \cap J_i = \emptyset$, $i=1, 2$. Then
\begin{align*}
|\G_{J_1 J_2}^{I_1 I_2}|
\lesssim \prod_{i=1}^2 \widetilde{F}_i(I_i, J_i) \frac{\rs(I_i, J_i)^{\frac{n_i}{2}}}{\ird(I_i, J_i)^{\delta_i}}.
\end{align*}
\end{lemma}

\begin{proof}
We split
\begin{align*}
\langle T (h_{I_1} \otimes h_{I_2}), h_{J_1} \otimes h_{J_2} \rangle
&=\langle T(h_{I_1} \otimes h_{I_2}),
(h_{J_1} \mathbf{1}_{3I_1 \setminus I_1}) \otimes (h_{J_2} \mathbf{1}_{3I_2 \setminus I_2}) \rangle
\\
&\quad+ \langle T(h_{I_1} \otimes h_{I_2}),
(h_{J_1} \mathbf{1}_{3I_1 \setminus I_1})  \otimes (h_{J_2} \mathbf{1}_{J_2 \setminus 3I_2}) \rangle
\\
&\quad+ \langle T(h_{I_1} \otimes h_{I_2}),
(h_{J_1} \mathbf{1}_{J_1 \setminus 3I_1})  \otimes (h_{J_2} \mathbf{1}_{3I_2 \setminus I_2}) \rangle
\\
&\quad+ \langle T(h_{I_1} \otimes h_{I_2}),
(h_{J_1} \mathbf{1}_{J_1 \setminus 3I_1})  \otimes (h_{J_2} \mathbf{1}_{J_2 \setminus 3I_2}) \rangle.
\end{align*}
Note that if $\ird(I_i, J_i)=1+\d(I_i, J_i)/\ell(I_i)>2$ for some $i=1, 2$, then $J_i \cap 3I_i=\emptyset$, and hence
\begin{align*}
\langle T(h_{I_1} \otimes h_{I_2}),
(h_{J_1} \mathbf{1}_{3I_1 \setminus I_1}) \otimes (h_{J_2} \mathbf{1}_{3I_2 \setminus I_2}) \rangle
=0.
\end{align*}
Thus to treat the first term, it suffices to consider the case $\ird(I_i, J_i) \le 2$ for each $i=1, 2$. It follows from the compact full kernel representation, the size condition of $K$, and \eqref{def:Q} that
\begin{align*}
|\langle &T(h_{I_1} \otimes h_{I_2}),
(h_{J_1} \mathbf{1}_{3I_1 \setminus I_1})  \otimes (h_{J_2} \mathbf{1}_{3I_2 \setminus I_2}) \rangle|
\\
&\lesssim  \prod_{i=1}^2 |I_i|^{-\frac12} |J_i|^{-\frac12} \mathscr{Q}_i(I_i)
\lesssim \prod_{i=1}^2 \widetilde{F}_i(I_i, J_i)
\frac{\rs(I_i, J_i)^{\frac{n_i}{2}}}{\ird(I_i, J_i)^{\delta_2}}. 
\end{align*}

To estimate the second term, as above, it suffices to analyze the case $\ird(I_1, J_1) \le 2$. We use the compact full kernel representation and the mixed size-H\"{o}lder condition of $K$ to obtain
\begin{align*}
|\langle &T(h_{I_1} \otimes h_{I_2}),
(h_{J_1} \mathbf{1}_{3I_1 \setminus I_1})  \otimes (h_{J_2} \mathbf{1}_{J_2 \setminus 3I_2}) \rangle|
\\
&\lesssim  |I_1|^{-\frac12} |J_1|^{-\frac12} \mathscr{Q}_1(I_1) \,
|I_2|^{-\frac12} |J_2|^{-\frac12} \mathscr{Q}_2(I_2, J_2)
\\
&\lesssim \widetilde{F}_1(I_1, J_1)  |I_1|^{\frac12} |J_1|^{-\frac12}
\widetilde{F}_2(I_2, J_2) \frac{|I_2|^{\frac12} |J_2|^{-\frac12}}{\ird(I_2, J_2)^{\delta_2}}
\\
&\lesssim \prod_{i=1}^2 \widetilde{F}_i(I_i, J_i)
\frac{\rs(I_i, J_i)^{\frac{n_i}{2}}}{\ird(I_i, J_i)^{\delta_i}},
\end{align*}
where \eqref{def:QIJ} and \eqref{def:Q} were utilized. Analogously,
\begin{align*}
|\langle T(h_{I_1} \otimes h_{I_2}),
(h_{J_1} \mathbf{1}_{J_1 \setminus 3I_1})  \otimes (h_{J_2} \mathbf{1}_{3I_2 \setminus I_2}) \rangle|
\lesssim \prod_{i=1}^2 \widetilde{F}_i(I_i, J_i)
\frac{\rs(I_i, J_i)^{\frac{n_i}{2}}}{\ird(I_i, J_i)^{\delta_i}}.
\end{align*}

To control the last term, in light of \eqref{def:QIJ}, we apply the compact full kernel representation, the cancellation of $h_{I_1}$ and $h_{I_2}$, and the H\"{o}lder condition to deduce
\begin{align*}
|\langle &T(h_{I_1} \otimes h_{I_2}),
(h_{J_1} \mathbf{1}_{J_1 \setminus 3I_1})  \otimes (h_{J_2} \mathbf{1}_{J_2 \setminus 3I_2}) \rangle|
\\
&\lesssim \prod_{i=1}^2 |I_i|^{-\frac12} |J_i|^{-\frac12} \mathscr{Q}_i(I_i, J_i)
\lesssim \prod_{i=1}^2 \widetilde{F}_i(I_i, J_i)
\frac{\rs(I_i, J_i)^{\frac{n_i}{2}}}{\ird(I_i, J_i)^{\delta_i}}.
\end{align*}
Therefore, the desired inequality is a consequence of these estimates.
\end{proof}

By Lemma \ref{lem:THH-5} and Lemma \ref{lem:FE} part \eqref{list:FE-2}, we have for all $N \ge N_0$ sufficiently large,
\begin{align*}
|\mathscr{I}_4^N| 
\le \sum_{\substack{k_1 \ge 0 \\ k_2 \ge 0}}
\sum_{\substack{1 \le m_1 \le 2^{k_1} \\ 1 \le m_2 \le 2^{k_2}}}
\sum_{\substack{J_1 \not\in \D_1(2N) \\ \text{or } J_2 \not\in \D_2(2N)}}
\sum_{\substack{I_1 \in J_1(k_1, 0, m_1) \\ I_2 \in J_2(k_2, 0, m_2)}}
|\G_{J_1 J_2}^{I_1 I_2}| \, |f_{I_1 I_2}| |g_{J_1 J_2}|
\lesssim \sum_{i=0}^2 \mathscr{I}_{4, i}^N,
\end{align*}
where
\begin{align*}
\mathscr{I}_{4, 0}^N
&:= \varepsilon
\sum_{\substack{k_1 \ge 0 \\ k_2 \ge 0}}
\sum_{\substack{1 \le m_1 \le 2^{k_1} \\ 1 \le m_2 \le 2^{k_2}}}
\sum_{\substack{J_1 \in \D_1 \\ J_2 \in \D_2}}
\sum_{\substack{I_1 \in J_1(k_1, 0, m_1) \\ I_2 \in J_2(k_2, 0, m_2)}}
2^{-k_1 \frac{n_1}{2}} 2^{-k_2\frac{n_2}{2}} m_1^{-\delta_1} m_2^{-\delta_2}
|f_{I_1 I_2}| |g_{J_1 J_2}|,
\\ 
\mathscr{I}_{4, i}^N
&:= \sum_{\substack{k_1, k_2 \ge 0 \\ k_i \ge N}}
\sum_{\substack{1 \le m_1 \le 2^{k_1} \\ 1 \le m_2 \le 2^{k_2}}}
\sum_{\substack{J_1 \in \D_1 \\ J_2 \in \D_2}}
\sum_{\substack{I_1 \in J_1(k_1, 0, m_1) \\ I_2 \in J_2(k_2, 0, m_2)}}
2^{-k_1 \frac{n_1}{2}} 2^{-k_2\frac{n_2}{2}} m_1^{-\delta_1} m_2^{-\delta_2}
|f_{I_1 I_2}| |g_{J_1 J_2}|,
\end{align*}
for each $i=1, 2$. It follows from the Cauchy--Schwarz inequality, \eqref{car-IJK}, and \eqref{MDT} that
\begin{align*}
\mathscr{I}_{4, 0}^N
&\le \varepsilon
\sum_{k_1, k_2 \ge 0} 2^{-k_1 \frac{n_1}{2}} 2^{-k_2\frac{n_2}{2}}
\bigg(\sum_{\substack{I_1 \in \D_1 \\ I_2 \in \D_2}}
\sum_{\substack{1 \le m_1 \le 2^{k_1} \\ J_1 \in I_1(k_1, 0, m_1)}}
\sum_{\substack{1 \le m_2 \le 2^{k_2} \\ J_2 \in I_2(k_2, 0, m_2)}} |f_{I_1 I_2}|^2 \bigg)^{\frac12}
\\
&\quad\times
\bigg(\sum_{\substack{1 \le m_1 \le 2^{k_1} \\ 1 \le m_2 \le 2^{k_2}}} 
m_1^{-2\delta_1} m_2^{-2\delta_2}
\sum_{\substack{J_1 \in \D_1 \\ J_2 \in \D_2}}
\sum_{\substack{I_1 \in J_1(k_1, 0, m_1) \\ I_2 \in J_2(k_2, 0, m_2)}} 
|g_{J_1 J_2}|^2 \bigg)^{\frac12}
\\
&\le \varepsilon
\sum_{k_1, k_2 \ge 0} \bigg(2^{-k_1} \sum_{1 \le m_1 \le 2^{k_1}} m_1^{-2\delta_1} \bigg)^{\frac12}
\bigg(\sum_{\substack{I_1 \in \D_1 \\ I_2 \in \D_2}} |f_{I_1 I_2}|^2 \bigg)^{\frac12}
\\
&\quad\times
\bigg(2^{-k_2} \sum_{1 \le m_2 \le 2^{k_2}} m_2^{-2\delta_2} \bigg)^{\frac12}
\bigg(\sum_{\substack{J_1 \in \D_1 \\ J_2 \in \D_2}} |g_{J_1 J_2}|^2 \bigg)^{\frac12}
\\
&\lesssim \varepsilon \sum_{k_1, k_2 \ge 0} 
2^{-k_1 \theta \delta_1} 2^{-k_2 \theta \delta_2}  \|f\|_{L^2} \|g\|_{L^2}
\lesssim \varepsilon \|f\|_{L^2} \|g\|_{L^2},
\end{align*}
where \eqref{MDT} was used in the last inequality. Likewise, for each $i=1, 2$, 
\begin{align*}
\mathscr{I}_{4, i}^N
&\lesssim
\sum_{\substack{k_1, k_2 \ge 0 \\ k_i \ge N}}  
2^{-k_1 \theta \delta_1} 2^{-k_2 \theta \delta_2}  \|f\|_{L^2} \|g\|_{L^2} 
\lesssim 2^{-N \theta \delta_i} \|f\|_{L^2} \|g\|_{L^2}
\le \varepsilon \|f\|_{L^2} \|g\|_{L^2},
\end{align*}
provided $N \ge 1$ large enough. Consequently, we achieve 
\begin{align*}
|\mathscr{I}_4^N|
\lesssim \sum_{i=0}^2 \mathscr{I}_{4, i}^N
\lesssim \varepsilon \|f\|_{L^2} \|g\|_{L^2}.
\end{align*}

\subsection{$\rd(I_1, J_1) < 1$, $I_1 \cap J_1 = \emptyset$, and $I_2 \subset J_2$}\label{sec:NI}  
First, consider the case $I_2=J_2$. 

\begin{lemma}\label{lem:THH-66}
Let $I_1, J_1 \in \D_1$ with $\ell(I_1) \le \ell(J_1)$ be such that $\rd(I_1, J_1) < 1$ and $I_1 \cap J_1 = \emptyset$, and let $I_2 \in \D_2$. Then, 
\begin{align*}
|\langle T (h_{I_1} \otimes h_{I_2}), h_{J_1} \otimes h_{I_2} \rangle|
\lesssim \widetilde{F}_1(I_1, J_1) \frac{\rs(I_1, J_1)^{\frac{n_1}{2}}}{\ird(I_1, J_1)^{\delta_1}} \, 
\widehat{F}_2(I_2). 
\end{align*}
\end{lemma}

\begin{proof}
We rewrite
\begin{align*}
\langle &T (h_{I_1} \otimes h_{I_2}), h_{J_1} \otimes h_{I_2} \rangle
\\
&= \langle T (h_{I_1} \otimes h_{I_2}), (h_{J_1} \mathbf{1}_{3I_1 \setminus I_1}) \otimes h_{I_2} \rangle
+ \langle T (h_{I_1} \otimes h_{I_2}), (h_{J_1} \mathbf{1}_{J_1 \setminus 3I_1}) \otimes h_{I_2} \rangle
\\
&= \sum_{I'_2, I''_2 \in \ch(I_2)} \langle h_{I_2} \rangle_{I'_2} \langle h_{I_2} \rangle_{I''_2}
\big[\langle T (h_{I_1} \otimes \mathbf{1}_{I'_2}),
(h_{J_1} \mathbf{1}_{3I_1 \setminus I_1}) \otimes \mathbf{1}_{I''_2} \rangle
\\
&\quad+ \langle T (h_{I_1} \otimes \mathbf{1}_{I'_2}),
(h_{J_1} \mathbf{1}_{J_1 \setminus 3I_1}) \otimes \mathbf{1}_{I''_2} \rangle \big].
\end{align*}
In the case $I'_2=I''_2$, by the compact partial kernel representation, the size condition of $K_{\mathbf{1}_{I'_2}, \mathbf{1}_{I'_2}}$, and \eqref{def:Q}, there holds 
\begin{align*}
|\langle T(h_{I_1} \otimes \mathbf{1}_{I'_2}),
(h_{J_1} \mathbf{1}_{3I_1 \setminus I_1}) \otimes \mathbf{1}_{I'_2} \rangle|
&\lesssim |I_1|^{-\frac12} |J_1|^{-\frac12} \mathscr{Q}_1(I_1)
C(\mathbf{1}_{I'_2}, \mathbf{1}_{I'_2})
\\
&\lesssim \widetilde{F}_1(I_1, J_1) \frac{\rs(I_1, J_1)^{\frac{n_1}{2}}}{\ird(I_1, J_1)^{\delta_1}} \, 
\widehat{F}_2(I_2) |I_2|,
\end{align*}
where we have used that $\ird(I_1, J_1) \le 2$, otherwise $J_1 \cap 3I_1 = \emptyset$. Applying the compact partial kernel representation and the H\"{o}lder condition of $K_{\mathbf{1}_{I'_2}, \mathbf{1}_{I'_2}}$, we conclude that
\begin{align*}
|\langle T (h_{I_1} \otimes \mathbf{1}_{I'_2}),
(h_{J_1} \mathbf{1}_{(3I_1)^c}) \otimes \mathbf{1}_{I'_2} \rangle|
&\lesssim |I_1|^{-\frac12} |J_1|^{-\frac12} \mathscr{Q}_1(I_1, J_1) \, C(\mathbf{1}_{I'_2}, \mathbf{1}_{I'_2})
\\
&\lesssim \widetilde{F}_1(I_1, J_1) \frac{\rs(I_1, J_1)^{\frac{n_1}{2}}}{\ird(I_1, J_1)^{\delta_1}} \, 
\widehat{F}_2(I_2) |I_2|,
\end{align*}
where the last inequality follows from \eqref{def:QIJ}.

Let us turn to the case $I'_2 \neq I''_2$. Note that $I''_2 \subset 3I'_2 \setminus I'_2$. By the compact full kernel representation, the size condition, and \eqref{def:Q}, we obtain 
\begin{align*}
|\langle T (h_{I_1} \otimes \mathbf{1}_{I'_2}),
(h_{J_1} \mathbf{1}_{3I_1 \setminus I_1}) \otimes \mathbf{1}_{I''_2} \rangle|
&\lesssim |I_1|^{-\frac12} |J_1|^{-\frac12} \mathscr{Q}_1(I_1) \mathscr{Q}_2(I'_2)
\\
&\lesssim \widetilde{F}_1(I_1, J_1)  |I_1|^{\frac12} |J_1|^{-\frac12} \, \widetilde{F}_2(I'_2) |I'_2|
\\
&\lesssim \widetilde{F}_1(I_1, J_1)  \frac{\rs(I_1, J_1)^{\frac{n_1}{2}}}{\ird(I_1, J_1)^{\delta_1}}
\widetilde{F}_2(I_2) |I_2|,
\end{align*}
because $\ird(I_1, J_1) \le 2$, otherwise $J_1 \cap 3I_1 = \emptyset$. As above, the compact full kernel representation and the mixed size-H\"{o}lder condition yield
\begin{align*}
|\langle T (h_{I_1} \otimes \mathbf{1}_{I'_2}),
(h_{J_1} \mathbf{1}_{J_1 \setminus 3I_1}) \otimes \mathbf{1}_{I''_2} \rangle|
&\lesssim |I_1|^{-\frac12} |J_1|^{-\frac12} \mathscr{Q}_1(I_1, J_1)  \, \mathscr{Q}_2(I'_2)
\\
&\lesssim \widetilde{F}_1(I_1, J_1)  \frac{\rs(I_1, J_1)^{\frac{n_1}{2}}}{\ird(I_1, J_1)^{\delta_1}}
\widetilde{F}_2(I_2) |I_2|,
\end{align*}
where \eqref{def:Q} and \eqref{def:QIJ} were used. The desired inequality immediately follows from these estimates above.
\end{proof}

Next, we treat the case $I_2 \subsetneq J_2$. Recall the decomposition \eqref{STH}.

\begin{lemma}\label{lem:THH-6}
Let $I_1, J_1 \in \D_1$ with $\ell(I_1) \le \ell(J_1)$, $\rd(I_1, J_1) < 1$, and $I_1 \cap J_1 = \emptyset$, and let $I_2 \subsetneq J_2 \in \D_2$. Then  
\begin{align}\label{eq:THH-6}
|\langle T (h_{I_1} \otimes h_{I_2}),
h_{J_1} \otimes \phi_{I_2 J_2} \rangle| 
\lesssim \widetilde{F}_1(I_1, J_1) \frac{\rs(I_1, J_1)^{\frac{n_1}{2}}}{\ird(I_1, J_1)^{\delta_1}}
\, \widetilde{F}_2(I_2) \rs(I_2, J_2)^{\frac{n_2}{2}}. 
\end{align}
If we assume in addition that $\d(I_2, J_{I_2}^c) > \ell(I_2)^{\frac12} \ell(J_2)^{\frac12}$, then 
\begin{align}\label{eq:THH-6d}
|\langle T (h_{I_1} \otimes h_{I_2}),
h_{J_1} \otimes \phi_{I_2 J_2} \rangle| 
\lesssim \widetilde{F}_1(I_1, J_1) \frac{\rs(I_1, J_1)^{\frac{n_1}{2}}}{\ird(I_1, J_1)^{\delta_1}}
\, \widetilde{F}_2(I_2, J_2) \rs(I_2, J_2)^{\frac{n_2}{2} + \frac{\delta_2}{2}}. 
\end{align}
\end{lemma}

\begin{proof}
Noting that $\supp(\phi_{I_2 J_2}) \subset J_{I_2}^c \subset I_2^c$, we rewrite  
\begin{align}\label{TSP}
\mathscr{E} 
:= \langle T (h_{I_1} \otimes h_{I_2}),
h_{J_1} \otimes \phi_{I_2 J_2} \rangle 
= (\mathscr{E}_{1, 1} + \mathscr{E}_{1, 2}) 
+ (\mathscr{E}_{2, 1} + \mathscr{E}_{2, 2}) 
=: \mathscr{E}_1 + \mathscr{E}_2,   
\end{align}
where 
\begin{align*}
\mathscr{E}_{1, 1} 
&:= \langle T (h_{I_1} \otimes h_{I_2}), (h_{J_1} \mathbf{1}_{3I_1 \setminus I_1}) 
\otimes (\phi_{I_2 J_2} \mathbf{1}_{3I_2 \setminus I_2})\rangle, 
\\
\mathscr{E}_{1, 2}
&:= \langle T (h_{I_1} \otimes h_{I_2}), (h_{J_1} \mathbf{1}_{3I_1 \setminus I_1})
\otimes (\phi_{I_2 J_2} \mathbf{1}_{(3I_2)^c})\rangle, 
\\ 
\mathscr{E}_{2, 1}
&:= \langle T (h_{I_1} \otimes h_{I_2}), (h_{J_1} \mathbf{1}_{J_1 \setminus 3I_1}) 
\otimes (\phi_{I_2 J_2} \mathbf{1}_{3I_2 \setminus I_2})\rangle, 
\\
\mathscr{E}_{2, 2}
&:= \langle T (h_{I_1} \otimes h_{I_2}), (h_{J_1} \mathbf{1}_{J_1 \setminus 3I_1})
\otimes (\phi_{I_2 J_2} \mathbf{1}_{(3I_2)^c})\rangle. 
\end{align*}
We use the compact full kernel representation, the size condition, and \eqref{def:Q} to obtain
\begin{align}\label{NEN-2}
|\mathscr{E}_{1, 1}|  
\lesssim \prod_{i=1}^2 |I_i|^{-\frac12} |J_i|^{-\frac12} \mathscr{Q}_i(I_i)
\lesssim \widetilde{F}_1(I_1, J_1) \frac{\rs(I_1, J_1)^{\frac{n_1}{2}}}{\ird(I_1, J_1)^{\delta_1}}
\, \widetilde{F}_2(I_2) \rs(I_2, J_2)^{\frac{n_2}{2}}. 
\end{align}
Applying the compact full kernel representation, the cancellation of $h_{I_2}$, the mixed size-H\"{o}lder condition, \eqref{def:Q}, and \eqref{def:R}, we arrive at
\begin{align}\label{NEN-3}
|\mathscr{E}_{1, 2}| 
&\lesssim |I_1|^{-\frac12} |J_1|^{-\frac12} \mathscr{Q}_1(I_1) \,
|I_2|^{-\frac12} |J_2|^{-\frac12} \mathscr{R}_2(I_2)
\\ \nonumber
&\lesssim  \widetilde{F}_1(I_1, J_1) \frac{\rs(I_1, J_1)^{\frac{n_1}{2}}}{\ird(I_1, J_1)^{\delta_1}}
\, \widetilde{F}_2(I_2) \rs(I_2, J_2)^{\frac{n_2}{2}}.
\end{align}
Along with \eqref{def:Q} and \eqref{def:QIJ}, the compact full kernel representation and the mixed size-H\"{o}lder condition imply that
\begin{align}\label{NEN-5}
|\mathscr{E}_{2, 1}| 
&\lesssim  |I_1|^{-\frac12} |J_1|^{-\frac12} \mathscr{Q}_1(I_1, J_1)
|I_2|^{-\frac12} |J_2|^{-\frac12} \mathscr{Q}_2(I_2)
\\ \nonumber
&\lesssim \widetilde{F}_1(I_1, J_1) \frac{\rs(I_1, J_1)^{\frac{n_1}{2}}}{\ird(I_1, J_1)^{\delta_1}} \,
\widetilde{F}_2(I_2) \rs(I_2, J_2)^{\frac{n_2}{2}}.
\end{align}
Additionally, it follows from the compact full kernel representation, the H\"{o}lder condition, \eqref{def:QIJ}, and \eqref{def:R} that
\begin{align}\label{NEN-6}
|\mathscr{E}_{2, 2}| 
&\lesssim |I_1|^{-\frac12} |J_1|^{-\frac12} \mathscr{Q}_1(I_1, J_1)
|I_2|^{-\frac12} |J_2|^{-\frac12} \mathscr{R}_2(I_2)
\\ \nonumber
&\lesssim \widetilde{F}_1(I_1, J_1) \frac{\rs(I_1, J_1)^{\frac{n_1}{2}}}{\ird(I_1, J_1)^{\delta_1}} \,
\widetilde{F}_2(I_2) \rs(I_2, J_2)^{\frac{n_2}{2}}.
\end{align}
Thus, \eqref{eq:THH-6} is a consequence of \eqref{TSP}--\eqref{NEN-6}.

Now assume that $\d(I_2, J_{I_2}^c) > \ell(I_2)^{\frac12} \ell(J_2)^{\frac12}$. Then $\d(I_2, J_{I_2}^c) \ge \ell(I_2)$. By the compact full kernel representation, the mixed size-H\"{o}lder condition of $K$, \eqref{def:Q}, and \eqref{def:RIJ}, we have 
\begin{align}\label{NEN-7}
|\mathscr{E}_1| 
&\lesssim |I_1|^{-\frac12} |J_1|^{-\frac12} \mathscr{Q}_1(I_1) \,
|I_2|^{-\frac12} |J_2|^{-\frac12} \mathscr{R}_2(I_2, J_2)
\\ \nonumber
&\lesssim  \widetilde{F}_1(I_1, J_1) \frac{\rs(I_1, J_1)^{\frac{n_1}{2}}}{\ird(I_1, J_1)^{\delta_1}}
\, \widetilde{F}_2(I_2, J_2) \rs(I_2, J_2)^{\frac{n_2}{2} + \frac{\delta_2}{2}}.
\end{align}
On the other hand, by the compact full kernel representation, the H\"{o}lder condition of $K$, \eqref{def:QIJ}, and \eqref{def:RIJ}, we deduce  
\begin{align}\label{NEN-8}
|\mathscr{E}_2| 
&\lesssim |I_1|^{-\frac12} |J_1|^{-\frac12} \mathscr{Q}_1(I_1, J_1)
|I_2|^{-\frac12} |J_2|^{-\frac12} \mathscr{R}_2(I_2, J_2)
\\ \nonumber
&\lesssim \widetilde{F}_1(I_1, J_1) \frac{\rs(I_1, J_1)^{\frac{n_1}{2}}}{\ird(I_1, J_1)^{\delta_1}} \,
\widetilde{F}_2(I_2, J_2) \rs(I_2, J_2)^{\frac{n_2}{2} + \frac{\delta_2}{2}}.
\end{align}
Consequently, \eqref{eq:THH-6d} follows from \eqref{TSP} and \eqref{NEN-7}--\eqref{NEN-8}.
\end{proof}

Denote 
\begin{align*}
\mathscr{I}_{5, 0}^N
&:= \bigg|\sum_{\substack{J_1 \not\in \D_1(2N) \\ \text{or } J_2 \not\in \D_2(2N)}}
\sum_{\substack{I_1 \in \D_1 \\ \ell(I_1) \le \ell(J_1) \\ \rd(I_1, J_1) < 1 \\ I_1 \cap J_1 = \emptyset}}
\sum_{\substack{I_2 \in \D_2 \\ I_2 \subsetneq J_2}}
\langle T (h_{I_1} \otimes h_{I_2}), h_{J_1} \otimes 1 \rangle 
\langle h_{J_2} \rangle_{I_2} \, 
f_{I_1 I_2} \, g_{J_1 J_2} \bigg|. 
\end{align*}

\begin{lemma}\label{lem:5hjpi}
For any $\varepsilon>0$, there exists an $N_0>1$ so that for all $N \ge N_0$, 
\begin{align*}
\mathscr{I}_{5, 0}^N
\lesssim \varepsilon \|f\|_{L^2} \|g\|_{L^2}. 
\end{align*}
\end{lemma}

\begin{proof}
Write 
\begin{align*}
\mathscr{I}_{5, 0}^N
\le \bigg|\sum_{\substack{J_1 \not\in \D_1(2N) \\ J_2 \in \D_2}}
\sum_{\substack{I_1 \in \D_1 \\ \ell(I_1) \le \ell(J_1) \\ \rd(I_1, J_1) < 1 \\ I_1 \cap J_1 = \emptyset}}
\sum_{\substack{I_2 \in \D_2 \\ I_2 \subsetneq J_2}} \cdots \bigg|
+ \bigg|\sum_{\substack{J_1 \in \D_1(2N) \\ J_2 \not\in \D_2(2N)}}
\sum_{\substack{I_1 \in \D_1 \\ \ell(I_1) \le \ell(J_1) \\ \rd(I_1, J_1) < 1 \\ I_1 \cap J_1 = \emptyset}}
\sum_{\substack{I_2 \in \D_2 \\ I_2 \subsetneq J_2}} \cdots \bigg|
=: \mathscr{I}_{5, 0}^{N, 1} + \mathscr{I}_{5, 0}^{N, 2}. 
\end{align*}
Denote $b_{I_1 J_1} := \langle T^*(h_{J_1} \otimes 1), h_{I_1} \rangle$. As argued in \eqref{Nbb}, we utilize \eqref{B2} and Lemma \ref{lem:FE} part \eqref{list:FE-2} to obtain that for any $N \ge N_0$, 
\begin{align*}
\mathscr{I}_{5, 0}^{N, 1}
\lesssim \mathscr{I}_{5, 0}^{N, 1, 1} + \mathscr{I}_{5, 0}^{N, 1, 2}, 
\end{align*}
where 
\begin{align*}
\mathscr{I}_{5, 0}^{N, 1, 1}
:= \varepsilon \Bigg[ \sum_{J_1 \in \D_1} 
\bigg(\sum_{k_1 \ge 0}  \sum_{\substack{1 \le m_1 \le 2^{k_1} \\ I_1 \in J_1(k_1, 0, m_1)}} 
2^{-k_1 \frac{n_1}{2}} m_1^{- \delta_1} 
\|\langle f, h_{I_1}\rangle\|_{L^2(\R^{n_2})} \bigg)^2 
\Bigg]^{\frac12} \, \|g\|_{L^2}, 
\end{align*}
and 
\begin{align*}
\mathscr{I}_{5, 0}^{N, 1, 2}
:= \Bigg[ \sum_{J_1 \in \D_1} 
\bigg(\sum_{k_1 \ge N} \sum_{\substack{1 \le m_1 \le 2^{k_1} \\ I_1 \in J_1(k_1, 0, m_1)}}
2^{-k_1 \frac{n_1}{2}} m_1^{- \delta_1} 
\|\langle f, h_{I_1}\rangle\|_{L^2(\R^{n_2})} \bigg)^2 
\Bigg]^{\frac12} \, \|g\|_{L^2}. 
\end{align*}
The Cauchy--Schwarz inequality, \eqref{car-IJK}, and \eqref{MDT} imply  
\begin{align}\label{PJKK}
&\Bigg(\sum_{k_1 \ge 0}  \sum_{\substack{m_1 \le 2^{k_1} \\ I_1 \in J_1(k_1, 0, m_1)}}
2^{-k_1 \frac{n_1}{2}} m_1^{-\delta_1} 
\|\langle f, h_{I_1}\rangle\|_{L^2(\R^{n_2})} \Bigg)^2
\\ \nonumber
&\le \Bigg(\sum_{k_1 \ge 0} 
\bigg[\sum_{\substack{1 \le m_1 \le 2^{k_1} \\ I_1 \in J_1(k_1, 0, m_1)}}
2^{- k_1 n_1} m_1^{-2 \delta_1} \bigg]^{\frac12}
\bigg[\sum_{\substack{1 \le m_1 \le 2^{k_1} \\ I_1 \in J_1(k_1, 0, m_1)}}
\|\langle f, h_{I_1}\rangle\|_{L^2(\R^{n_2})}^2 \bigg]^{\frac12} \Bigg)^2 
\\ \nonumber
&\lesssim \Bigg(\sum_{k_1 \ge 0} 2^{-k_1 \theta \delta_1} 
\bigg[\sum_{\substack{1 \le m_1 \le 2^{k_1} \\ I_1 \in J_1(k_1, 0, m_1)}} 
\|\langle f, h_{I_1}\rangle\|_{L^2(\R^{n_2})}^2 \bigg]^{\frac12} \Bigg)^2
\\ \nonumber
&\lesssim \sum_{k_1 \ge 0} 2^{-k_1 \theta \delta_1} 
\sum_{\substack{1 \le m_1 \le 2^{k_1} \\ I_1 \in J_1(k_1, 0, m_1)}} 
\|\langle f, h_{I_1}\rangle\|_{L^2(\R^{n_2})}^2 
\end{align}
provided that 
\begin{align*}
\sum_{\substack{1 \le m_1 \le 2^{k_1} \\ I_1 \in J_1(k_1, 0, m_1)}}
2^{- k_1 n_1} m_1^{-2 \delta_1} 
\lesssim 2^{-k_1} \sum_{1 \le m_1 \le 2^{k_1}} m_1^{-2 \delta_1}
\lesssim 2^{-2k_1 \theta \delta_1}. 
\end{align*}
Thus, \eqref{PJKK} and \eqref{car-IJK} give 
\begin{align*}
\mathscr{I}_{5, 0}^{N, 1, 1}
&\lesssim \varepsilon \bigg[ 
\sum_{k_1 \ge 0} 2^{-k_1 \theta \delta_1} \sum_{I_1 \in \D_1} 
\sum_{\substack{1 \le m_1 \le 2^{k_1} \\ J_1 \in I_1(k_1, 0, m_1)}}
\|\langle f, h_{I_1} \rangle\|_{L^2(\R^{n_2})}^2  \bigg]^{\frac12} \, \|g\|_{L^2} 
\\
&\lesssim \varepsilon \bigg[ 
\sum_{k_1 \ge 0} 2^{-k_1 \theta \delta_1}   
\sum_{I_1 \in \D_1} \|\langle f, h_{I_1} \rangle\|_{L^2(\R^{n_2})}^2   
\bigg]^{\frac12} \, \|g\|_{L^2} 
\lesssim \varepsilon \|f\|_{L^2} \|g\|_{L^2}, 
\end{align*}
and for any $N \ge N_0$ sufficiently large, 
\begin{align*}
\mathscr{I}_{5, 0}^{N, 1, 2}
&\lesssim \bigg[ 
\sum_{k_1 \ge N} 2^{-k_1 \theta \delta_1}   
\sum_{I_1 \in \D_1} \|\langle f, h_{I_1} \rangle\|_{L^2(\R^{n_2})}^2   
\bigg]^{\frac12} \, \|g\|_{L^2} 
\\
&\lesssim 2^{-N \theta \delta_1/2} \|f\|_{L^2} \|g\|_{L^2} 
\le \varepsilon \|f\|_{L^2} \|g\|_{L^2}. 
\end{align*}
Therefore,  
\begin{align}\label{50N1}
\mathscr{I}_{5, 0}^{N, 1}
\lesssim \mathscr{I}_{5, 0}^{N, 1, 1} + \mathscr{I}_{5, 0}^{N, 1, 2} 
\lesssim \varepsilon \|f\|_{L^2} \|g\|_{L^2}. 
\end{align}

In order to estimate $\mathscr{I}_{5, 0}^{N, 2}$, we split 
\begin{align*}
\mathscr{I}_{5, 0}^{N, 2}
= \bigg|\sum_{J_1 \in \D_1(2N)}
\sum_{\substack{I_1 \in \D_1 \\ \ell(I_1) \le \ell(J_1) \\ \rd(I_1, J_1) < 1 \\ I_1 \cap J_1 = \emptyset}} 
\big\langle h_{J_1} \otimes P_{2N}^{\perp} \Pi_{b_{I_1 J_1}}^* (\langle f, h_{I_1}\rangle),  
g \big\rangle \bigg| 
\le \mathscr{I}_{5, 0}^{N, 2, 1} + \mathscr{I}_{5, 0}^{N, 2, 2}, 
\end{align*}
where 
\begin{align*}
\mathscr{I}_{5, 0}^{N, 2, 1}
&:= \bigg|\sum_{J_1 \in \D_1(2N)}
\sum_{\substack{I_1 \in \D_1 \\ \ell(I_1) \le \ell(J_1) \\ \rd(I_1, J_1) < 1 \\ I_1 \cap J_1 = \emptyset}} 
\big\langle h_{J_1} \otimes P_{2N}^{\perp} \Pi_{P_N b_{I_1 J_1}}^* (\langle f, h_{I_1}\rangle),  
g \big\rangle \bigg|, 
\\ 
\mathscr{I}_{5, 0}^{N, 2, 2}
&:= \bigg|\sum_{J_1 \in \D_1(2N)}
\sum_{\substack{I_1 \in \D_1 \\ \ell(I_1) \le \ell(J_1) \\ \rd(I_1, J_1) < 1 \\ I_1 \cap J_1 = \emptyset}} 
\big\langle h_{J_1} \otimes P_{2N}^{\perp} \Pi_{P_N^{\perp} b_{I_1 J_1}}^* (\langle f, h_{I_1}\rangle),  
g \big\rangle \bigg|.  
\end{align*}
Much as in \eqref{30N22}, we apply \eqref{PB2} and Lemma \ref{lem:FE} part \eqref{list:FE-4} to obtain that for any $N \ge N_0$ large enough, 
\begin{align*}
\mathscr{I}_{5, 0}^{N, 2, 2}
&\le \Bigg[ \sum_{J_1 \in \D_1} 
\bigg(\sum_{\substack{I_1 \in \D_1 \\ \ell(I_1) \le \ell(J_1) \\ \rd(I_1, J_1) < 1 \\ I_1 \cap J_1 = \emptyset}} 
\|P_{2N}^{\perp} \Pi_{P_N^{\perp} b_{I_1 J_1}}^*\|_{L^2 \to L^2} 
\|\langle f, h_{I_1}\rangle\|_{L^2(\R^{n_2})} \bigg)^2 
\Bigg]^{\frac12} \, \|g\|_{L^2}
\\
& \lesssim \mathscr{I}_{5, 0}^{N, 1, 1}  
\lesssim \varepsilon \|f\|_{L^2} \|g\|_{L^2}. 
\end{align*}
Analogously to \eqref{30N21}, the estimates \eqref{B2}, \eqref{car-IJK}, and \eqref{MDT} give  
\begin{align*}
\mathscr{I}_{5, 0}^{N, 2, 1}
&= \bigg|\sum_{\substack{J_1 \in \D_1(2N) \\ J_2 \not\in \D_2(2N)}}
\sum_{\substack{I_1 \in \D_1 \\ \ell(I_1) \le \ell(J_1) \\ \rd(I_1, J_1) < 1 \\ I_1 \cap J_1 = \emptyset}}
\sum_{\substack{I_2 \in \D_2(N) \\ I_2 \subsetneq J_2}} 
\langle b_{I_1 J_1}, h_{I_2} \rangle \langle h_{J_2} \rangle_{I_2} 
f_{I_1 I_2} g_{J_1 J_2}\bigg|
\\ 
&\lesssim \sum_{\substack{k_1 \ge 0 \\ 1 \le m_1 \le 2^{k_1}}} 
\sum_{\substack{-N \le \ell_2 \le N \\ k_2 > 2N - \ell_2}}  
2^{-k_1 \frac{n_1}{2}} 2^{-k_2 \frac{n_2}{2}} m_1^{- \delta_1} 
\\ 
&\quad\times \sum_{\substack{J_1 \in \D_1 \\ I_1 \in J_1(k_1, 0, m_1)}} 
\sum_{I_2 \in \D_2^{\ell_2}(N)} 
|f_{I_1 I_2}| |g_{J_1 I_2^{(k_2)}}|
\\ 
&\le \sum_{k_1 \ge 0}  
\sum_{\substack{-N \le \ell_2 \le N \\ k_2 > 2N - \ell_2}}  
2^{-k_1 \frac{n_1}{2}} 2^{-k_2 \frac{n_2}{2}} 
\bigg(\sum_{\substack{I_1 \in \D_1 \\ I_2 \in \D_2}} 
\sum_{\substack{1 \le m_1 \le 2^{k_1} \\ J_1 \in I_1(k_1, 0, m_1)}} 
|f_{I_1 I_2}|^2 \bigg)^{\frac12}
\\
&\quad\times   
\bigg(\sum_{\substack{J_1 \in \D_1 \\ J_2 \in \D_2}} 
\sum_{1 \le m_1 \le 2^{k_1}} m_1^{-2 \delta_1}  
\sum_{\substack{I_1 \in J_1(k_1, 0, m_1) \\ I_2 \in \D_2^{\ell_2}(N) \\ I_2^{(k_2)} = J_2}} 
|g_{J_1 J_2}|^2 \bigg)^{\frac12} 
\\ 
&\lesssim \sum_{k_1 \ge 0} 
\sum_{\substack{-N \le \ell_2 \le 2N \\ k_2 > 2N - \ell_2}} 
2^{- k_1 \theta \delta_1} 2^{-k_2 \frac{n_2}{2}}  
N^{\frac{n_2}{2}} 2^{(N-\ell_2) \frac{n_2}{2}}
\|f\|_{L^2} \|g\|_{L^2}
\\ 
&\lesssim N^{1 + \frac{n_2}{2}} 2^{-N \frac{n_2}{2}}
\|f\|_{L^2} \|g\|_{L^2}
\le \varepsilon \|f\|_{L^2} \|g\|_{L^2}, 
\end{align*}
provided $N \ge 2$ sufficiently large. Hence, 
\begin{align}\label{50N2}
\mathscr{I}_{5, 0}^{N, 2}
\le \mathscr{I}_{5, 0}^{N, 2, 1} + \mathscr{I}_{5, 0}^{N, 2, 2} 
\lesssim \varepsilon \|f\|_{L^2} \|g\|_{L^2}. 
\end{align}
As a consequence of \eqref{50N1} and \eqref{50N2}, there holds 
\begin{align*}
\mathscr{I}_{5, 0}^N 
\le \mathscr{I}_{5, 0}^{N, 1} + \mathscr{I}_{5, 0}^{N, 2} \lesssim \varepsilon \|f\|_{L^2} \|g\|_{L^2}.
\end{align*}
The proof is complete.  
\end{proof}

In view of \eqref{STH}, Lemmas \ref{lem:THH-66} and \ref{lem:THH-6}, and Lemma \ref{lem:FE} parts \eqref{list:FE-2}--\eqref{list:FE-4}, we have that for any $N \ge N_0$ large enough,
\begin{align*}
|\mathscr{I}_5^N|
= \bigg|\sum_{\substack{k_1 \ge 0 \\ k_2 \ge 0}} \sum_{1 \le m_1 \le 2^{k_1}}
\sum_{\substack{J_1 \not\in \D_1(2N) \\ \text{or } J_2 \not\in \D_2(2N)}}
\sum_{\substack{I_1 \in J_1(k_1, 0, m_1) \\ I_2 \in \D_2^{k_2}(J_2)}}
\G_{J_1 J_2}^{I_1 I_2} \, f_{I_1 I_2} \, g_{J_1 J_2} \bigg|
\lesssim \sum_{i=0}^5 \mathscr{I}_{5, i}^N,
\end{align*}
where
\begin{align*}
\mathscr{I}_{5, 1}^N
&:= \varepsilon
\sum_{\substack{k_1 \ge 0 \\ k_2 \ge 0}} \sum_{1 \le m_1 \le 2^{k_1}} 
\sum_{\substack{J_1 \in \D_1 \\ J_2 \in \D_2}} 
\sum_{\substack{I_1 \in J_1(k_1, 0, m_1) \\ I_2 \in \mathcal{B}_2^{k_2}(J_2)}} 
2^{-k_1 \frac{n_1}{2}} 2^{-k_2\frac{n_2}{2}} m_1^{-\delta_1}
|f_{I_1 I_2}| |g_{J_1 J_2}|,
\\
\mathscr{I}_{5, 2}^N
&:= \sum_{\substack{k_1 \ge N \\ k_2 \ge 0}} \sum_{1 \le m_1 \le 2^{k_1}} 
\sum_{\substack{J_1 \in \D_1 \\ J_2 \in \D_2}} 
\sum_{\substack{I_1 \in J_1(k_1, 0, m_1) \\ I_2 \in \mathcal{B}_2^{k_2}(J_2)}} 
2^{-k_1 \frac{n_1}{2}} 2^{-k_2\frac{n_2}{2}} m_1^{-\delta_1}
|f_{I_1 I_2}| |g_{J_1 J_2}|,
\\
\mathscr{I}_{5, 3}^N
&:= \sum_{\substack{k_1 \ge 0 \\ k_2 \ge N}} \sum_{1 \le m_1 \le 2^{k_1}} 
\sum_{\substack{J_1 \in \D_1 \\ J_2 \in \D_2}} 
\sum_{\substack{I_1 \in J_1(k_1, 0, m_1) \\ I_2 \in \mathcal{B}_2^{k_2}(J_2)}} 
2^{-k_1 \frac{n_1}{2}} 2^{-k_2\frac{n_2}{2}} m_1^{-\delta_1}
|f_{I_1 I_2}| |g_{J_1 J_2}|,
\\ 
\mathscr{I}_{5, 4}^N
&:= \varepsilon
\sum_{\substack{k_1 \ge 0 \\ k_2 \ge 1}} \sum_{1 \le m_1 \le 2^{k_1}} 
\sum_{\substack{J_1 \in \D_1 \\ J_2 \in \D_2}} 
\sum_{\substack{I_1 \in J_1(k_1, 0, m_1) \\ I_2 \in \mathcal{G}_2^{k_2}(J_2)}} 
2^{-k_1 \frac{n_1}{2}} 2^{-k_2 (\frac{n_2}{2} + \frac{\delta_2}{2})} m_1^{-\delta_1}
|f_{I_1 I_2}| |g_{J_1 J_2}|,
\\
\mathscr{I}_{5, 5}^N
&:= \sum_{\substack{k_1 \ge N \\ k_2 \ge 1}} \sum_{1 \le m_1 \le 2^{k_1}} 
\sum_{\substack{J_1 \in \D_1 \\ J_2 \in \D_2}} 
\sum_{\substack{I_1 \in J_1(k_1, 0, m_1) \\ I_2 \in \mathcal{G}_2^{k_2}(J_2)}} 
2^{-k_1 \frac{n_1}{2}} 2^{-k_2 (\frac{n_2}{2} + \frac{\delta_2}{2})} m_1^{-\delta_1}
|f_{I_1 I_2}| |g_{J_1 J_2}|,
\\ 
\mathscr{I}_{5, 6}^N
&:= \sum_{\substack{k_1 \ge 0 \\ k_2 \ge N}} \sum_{1 \le m_1 \le 2^{k_1}} 
\sum_{\substack{J_1 \in \D_1 \\ J_2 \in \D_2}} 
\sum_{\substack{I_1 \in J_1(k_1, 0, m_1) \\ I_2 \in \mathcal{G}_2^{k_2}(J_2)}} 
2^{-k_1 \frac{n_1}{2}} 2^{-k_2 (\frac{n_2}{2} + \frac{\delta_2}{2})} m_1^{-\delta_1}
|f_{I_1 I_2}| |g_{J_1 J_2}|. 
\end{align*}
Invoking the Cauchy--Schwarz inequality, \eqref{GB}, \eqref{car-IJK}, and \eqref{MDT}, we obtain
\begin{align*}
\mathscr{I}_{5, 1}^N
&\le \varepsilon
\sum_{k_1, k_2 \ge 0} 2^{-k_1 \frac{n_1}{2}} 2^{-k_2\frac{n_2}{2}}
\bigg(\sum_{\substack{I_1 \in \D_1 \\ I_2 \in \D_2}} 
\sum_{1 \le m_1 \le 2^{k_1}} 
\sum_{\substack{J_1 \in I_1(k_1, 0, m_1) \\ J_2 = I_2^{(k_2)}}}  
|f_{I_1 I_2}|^2 \bigg)^{\frac12}
\\
&\quad\times \bigg(\sum_{1 \le m_1 \le 2^{k_1}} m_1^{-2\delta_1}
\sum_{\substack{J_1 \in \D_1 \\ J_2 \in \D_2}} 
\sum_{\substack{I_1 \in J_1(k_1, 0, m_1) \\ I_2 \in \mathcal{B}_2^{k_2}(J_2)}}  
|g_{J_1 J_2}|^2 \bigg)^{\frac12}
\\
&\lesssim \varepsilon
\sum_{k_1, k_2 \ge 0} 2^{- \frac{k_2}{4}} 
\bigg(2^{-k_1} \sum_{1 \le m_1 \le 2^{k_1}} m_1^{-2\delta_1} \bigg)^{\frac12}
\\
&\quad\times \bigg(\sum_{\substack{I_1 \in \D_1 \\ I_2 \in \D_2}} |f_{I_1 I_2}|^2 \bigg)^{\frac12} 
\bigg(\sum_{\substack{J_1 \in \D_1 \\ J_2 \in \D_2}} |g_{J_1 J_2}|^2 \bigg)^{\frac12}
\\
&\lesssim \varepsilon \sum_{k_1, k_2 \ge 0} 
2^{-k_1 \theta \delta_1} 2^{- \frac{k_2}{4}}  
\|f\|_{L^2} \|g\|_{L^2}
\lesssim \varepsilon \|f\|_{L^2} \|g\|_{L^2}, 
\end{align*}
and 
\begin{align*}
\mathscr{I}_{5, 4}^N
&\le \varepsilon
\sum_{k_1, k_2 \ge 0} 2^{-k_1 \frac{n_1}{2}} 
2^{-k_2(\frac{n_2}{2} + \frac{\delta_2}{2})}
\bigg(\sum_{\substack{I_1 \in \D_1 \\ I_2 \in \D_2}} 
\sum_{1 \le m_1 \le 2^{k_1}} 
\sum_{\substack{J_1 \in I_1(k_1, 0, m_1) \\ J_2 = I_2^{(k_2)}}}  
|f_{I_1 I_2}|^2 \bigg)^{\frac12}
\\
&\quad\times \bigg(\sum_{1 \le m_1 \le 2^{k_1}} m_1^{-2\delta_1}
\sum_{\substack{J_1 \in \D_1 \\ J_2 \in \D_2}} 
\sum_{\substack{I_1 \in J_1(k_1, 0, m_1) \\ I_2 \in \mathcal{G}_2^{k_2}(J_2)}}  
|g_{J_1 J_2}|^2 \bigg)^{\frac12}
\\
&\lesssim \varepsilon
\sum_{k_1, k_2 \ge 0} 2^{- k_2 \frac{\delta_2}{2}} 
\bigg(2^{-k_1} \sum_{1 \le m_1 \le 2^{k_1}} m_1^{-2\delta_1} \bigg)^{\frac12}
\\
&\quad\times \bigg(\sum_{\substack{I_1 \in \D_1 \\ I_2 \in \D_2}} |f_{I_1 I_2}|^2 \bigg)^{\frac12} 
\bigg(\sum_{\substack{J_1 \in \D_1 \\ J_2 \in \D_2}} |g_{J_1 J_2}|^2 \bigg)^{\frac12}
\\
&\lesssim \varepsilon \sum_{k_1, k_2 \ge 0} 
2^{-k_1 \theta \delta_1} 2^{- k_2 \frac{\delta_2}{2}}  
\|f\|_{L^2} \|g\|_{L^2}
\lesssim \varepsilon \|f\|_{L^2} \|g\|_{L^2}.
\end{align*}
Analogously, \eqref{MDT} implies that for all $N \ge N_0$ large enough,
\begin{align*}
\mathscr{I}_{5, 2}^N
\lesssim \sum_{k_1 \ge N, k_2 \ge 0} 
2^{-k_1 \theta \delta_1} 2^{- \frac{k_2}{4}} \|f\|_{L^2} \|g\|_{L^2}
\lesssim 2^{-N \theta \delta_1} \|f\|_{L^2} \|g\|_{L^2}
\le \varepsilon \|f\|_{L^2} \|g\|_{L^2},
\end{align*}
\begin{align*}
\mathscr{I}_{5, 3}^N
\lesssim \sum_{k_1 \ge 0, k_2 \ge N} 
2^{-k_1 \theta \delta_1} 2^{- \frac{k_2}{4}} \|f\|_{L^2} \|g\|_{L^2}
\lesssim 2^{-N/4} \|f\|_{L^2} \|g\|_{L^2}
\le \varepsilon \|f\|_{L^2} \|g\|_{L^2},
\end{align*}
\begin{align*}
\mathscr{I}_{5, 5}^N 
&\lesssim \sum_{k_1 \ge N, k_2 \ge 0} 
2^{-k_1 \theta \delta_1} 2^{- k_2 \frac{\delta_2}{2}}  \|f\|_{L^2} \|g\|_{L^2}
\lesssim 2^{-N \theta \delta_1} \|f\|_{L^2} \|g\|_{L^2}
\le \varepsilon \|f\|_{L^2} \|g\|_{L^2}.
\end{align*}
and
\begin{align*}
\mathscr{I}_{5, 6}^N 
&\lesssim \sum_{k_1 \ge 0, k_2 \ge N} 
2^{-k_1 \theta \delta_1} 2^{- k_2 \frac{\delta_2}{2}}  \|f\|_{L^2} \|g\|_{L^2}
\lesssim 2^{-N \delta_2/2} \|f\|_{L^2} \|g\|_{L^2}
\le \varepsilon \|f\|_{L^2} \|g\|_{L^2}.
\end{align*}
Therefore, these estimates above and Lemma \ref{lem:5hjpi} yield 
\begin{align*}
\mathscr{I}_5^N 
\lesssim \sum_{i=0}^6 \mathscr{I}_{5, i}^N 
\lesssim \varepsilon \|f\|_{L^2} \|g\|_{L^2}.
\end{align*}

\subsection{$I_1 \subset J_1$ and $I_2 \subset J_2$}\label{sec:Inside}  
First, we handle the case $I_1 = J_1$ and $I_2 = J_2$.

\begin{lemma}\label{lem:THH-81}
For all $I_1 \in \D_1$ and $I_2 \in \D_2$, there holds 
\begin{align}\label{eq:thh81}
|\langle T(h_{I_1} \otimes h_{I_2}), h_{I_1} \otimes h_{I_2} \rangle|
\lesssim \prod_{i=1}^2 \widehat{F}_i(I_i) \rs(I_i, J_i)^{\frac{n_i}{2}}.
\end{align}
\end{lemma}

\begin{proof}
Write
\begin{align}\label{EQEQ}
\langle &T(h_{I_1} \otimes h_{I_2}), h_{I_1} \otimes h_{I_2} \rangle
\nonumber \\
&=\sum_{\substack{I'_1, I''_1 \in \ch(I_1) \\ I'_2, I''_2 \in \ch(I_2)}}
\langle h_{I_1} \rangle_{I'_1} \langle h_{I_2} \rangle_{I'_2}
\langle h_{I_1} \rangle_{I''_1} \langle h_{I_2} \rangle_{I''_2}
\langle T (\mathbf{1}_{I'_1} \otimes \mathbf{1}_{I'_2}),
\mathbf{1}_{I''_1} \otimes \mathbf{1}_{I''_2} \rangle.
\end{align}
The weak compactness property gives
\begin{align}\label{EQEQ-1}
|\langle T (\mathbf{1}_{I'_1} \otimes \mathbf{1}_{I'_2}),
\mathbf{1}_{I'_1} \otimes \mathbf{1}_{I'_2} \rangle|
\le F_1(I'_1) |I'_1| \, F_2(I'_2) |I'_2|
\lesssim \widehat{F}_1(I_1) |I_1| \, \widehat{F}_2(I_2) |I_2|.
\end{align}
If $I'_1=I''_1$ and $I'_2 \neq I''_2$, we use the compact partial kernel representation, the size condition of $K_{\mathbf{1}_{I'_1}, \mathbf{1}_{I'_1}}$, and \eqref{def:Q} to arrive at
\begin{align}\label{EQEQ-2}
|\langle &T(\mathbf{1}_{I'_1} \otimes \mathbf{1}_{I'_2}),
\mathbf{1}_{I'_1} \otimes \mathbf{1}_{I''_2} \rangle| 
\lesssim C(\mathbf{1}_{I'_1},  \mathbf{1}_{I'_1}) \mathscr{Q}_2(I'_2)
\\ \nonumber
&\lesssim F_1(I'_1) |I'_1| \, \widetilde{F}_2(I_2) |I_2|
\lesssim \widehat{F}_1(I_1) |I_1| \, \widetilde{F}_2(I_2) |I_2|.
\end{align}
Similarly, if $I'_1 \neq I''_1$ and $I'_2 = I''_2$, we have
\begin{align}\label{EQEQ-3}
|\langle T (\mathbf{1}_{I'_1} \otimes \mathbf{1}_{I'_2}),
\mathbf{1}_{I''_1} \otimes \mathbf{1}_{I'_2} \rangle|
\lesssim \widetilde{F}_1(I_1) |I_1| \, \widehat{F}_2(I_2) |I_2|.
\end{align}
If $I'_1 \neq I'_2$ and $I''_1 \neq I''_2$, then in light of \eqref{def:Q}, the compact full kernel representation and the size condition imply
\begin{align}\label{EQEQ-4}
|\langle T (\mathbf{1}_{I'_1} \otimes \mathbf{1}_{I'_2}),
\mathbf{1}_{I''_1} \otimes \mathbf{1}_{I''_2} \rangle|
\lesssim \prod_{i=1}^2 \mathscr{Q}_i(I'_i) 
\lesssim \prod_{i=1}^2 \widetilde{F}_i(I'_i) |I'_i|
\lesssim \prod_{i=1}^2 \widetilde{F}_i(I_i) |I_i|.
\end{align}
Therefore, \eqref{eq:thh81} is a consequence of \eqref{EQEQ}--\eqref{EQEQ-4}.
\end{proof}

Second, let us treat the cases (i) $I_1 = J_1$ and $I_2 \subsetneq J_2$, and (ii) $I_1 \subsetneq J_1$ and $I_2 = J_2$.  Write 
\begin{align}\label{TH42}
\langle T (h_{I_1} \otimes h_{I_2}), h_{I_1} \otimes h_{J_2} \rangle
&= \langle T (h_{I_1} \otimes h_{I_2}), h_{I_1} \otimes \phi_{I_2 J_2} \rangle
\\ \nonumber
&\quad+ \langle T (h_{I_1} \otimes h_{I_2}), h_{I_1} \otimes 1 \rangle \langle h_{J_2} \rangle_{I_2}, 
\end{align}
and 
\begin{align}\label{TH43}
\langle T (h_{I_1} \otimes h_{I_2}), h_{J_1} \otimes h_{I_2} \rangle
&= \langle T (h_{I_1} \otimes h_{I_2}), \phi_{I_1 J_1} \otimes h_{I_2} \rangle
\\ \nonumber
&\quad+ \langle T (h_{I_1} \otimes h_{I_2}), 1 \otimes h_{I_2} \rangle \langle h_{J_1} \rangle_{I_1}.
\end{align}

\begin{lemma}\label{lem:THH-82}
Let $I_1, J_1 \in \D_1$ and $I_2, J_2 \in \D_2$. The following hold: 
\begin{enumerate}
\item[(a)] If $I_2 \subsetneq J_2$, then 
\begin{align}\label{eq:THH821}
|\langle T (h_{I_1} \otimes h_{I_2}),
h_{I_1} \otimes \phi_{I_2 J_2} \rangle| 
\lesssim \widehat{F}_1(I_1)  
\widetilde{F}_2(I_2) \rs(I_2, J_2)^{\frac{n_2}{2}}. 
\end{align}

\item[(b)] If $I_2 \subsetneq J_2$ with $\d(I_2, J_{I_2}^c) > \ell(I_2)^{\frac12} \ell(J_2)^{\frac12}$, then  
\begin{align}\label{eq:THH822}
|\langle T (h_{I_1} \otimes h_{I_2}),
h_{I_1} \otimes \phi_{I_2 J_2} \rangle| 
\lesssim \widehat{F}_1(I_1)  
\widetilde{F}_2(I_2, J_2) \rs(I_2, J_2)^{\frac{n_2}{2} + \frac{\delta_2}{2}}. 
\end{align}

\item[(c)] If $I_1 \subsetneq J_1$, then 
\begin{align}\label{eq:THH823}
|\langle T (h_{I_1} \otimes h_{I_2}),
\phi_{I_1 J_1} \otimes h_{I_2} \rangle| 
\lesssim  \widetilde{F}_1(I_1) \rs(I_1, J_1)^{\frac{n_1}{2}} \widehat{F}_2(I_2). 
\end{align}

\item[(d)] If $I_1 \subsetneq J_1$ with $\d(I_1, J_{I_1}^c) > \ell(I_1)^{\frac12} \ell(J_1)^{\frac12}$, then 
\begin{align}\label{eq:THH824}
|\langle T (h_{I_1} \otimes h_{I_2}),
\phi_{I_1 J_1} \otimes h_{I_2} \rangle| 
\lesssim  \widetilde{F}_1(I_1) \rs(I_1, J_1)^{\frac{n_1}{2} + \frac{\delta_1}{2}} \widehat{F}_2(I_2). 
\end{align}
\end{enumerate} 
\end{lemma}

\begin{proof}
We rewrite
\begin{align}\label{U1}
\langle T (h_{I_1} \otimes h_{I_2}),
h_{I_1} \otimes \phi_{I_2 J_2} \rangle 
&= \sum_{I'_1, I''_1 \in \ch(I_1)} 
\langle T(\mathbf{1}_{I'_1} \otimes h_{I_2}), 
\mathbf{1}_{I''_1} \otimes \phi_{I_2 J_2} \rangle 
\langle h_{I_1} \rangle_{I'_1} \langle h_{I_1} \rangle_{I''_1}. 
\end{align}
If $I'_1 = I''_1$, then the compact partial kernel representation, \eqref{def:Q}, and \eqref{def:R} give 
\begin{align}\label{U2}
|\langle &T(\mathbf{1}_{I'_1} \otimes h_{I_2}), 
\mathbf{1}_{I''_1} \otimes \phi_{I_2 J_2} \rangle| 
\\ \nonumber
&\le |\langle T(\mathbf{1}_{I'_1} \otimes h_{I_2}), 
\mathbf{1}_{I'_1} \otimes (\phi_{I_2 J_2} \mathbf{1}_{3I_2}) \rangle| 
+ |\langle T(\mathbf{1}_{I'_1} \otimes h_{I_2}), 
\mathbf{1}_{I'_1} \otimes (\phi_{I_2 J_2} \mathbf{1}_{(3I_2)^c}) \rangle| 
\\ \nonumber
&\le C(\mathbf{1}_{I'_1}, \mathbf{1}_{I'_1}) 
[\mathscr{Q}_2(I_2) + \mathscr{R}_2(I_2)] |I_2|^{-\frac12} |J_2|^{-\frac12}
\lesssim F_1(I'_1) |I_1|\, \widetilde{F}_2(I_2) \rs(I_2, J_2)^{\frac{n_2}{2}}.
\end{align}
If $I'_1 \neq I''_1$, then $I''_1 \subset 3I'_1 \setminus I'_1$, and the compact full kernel representation, \eqref{def:Q}, and \eqref{def:R} imply 
\begin{align}\label{U3}
|\langle &T(\mathbf{1}_{I'_1} \otimes h_{I_2}), 
\mathbf{1}_{I''_1} \otimes \phi_{I_2 J_2} \rangle| 
\\ \nonumber
&\le |\langle T(\mathbf{1}_{I'_1} \otimes h_{I_2}), 
\mathbf{1}_{I''_1} \otimes (\phi_{I_2 J_2} \mathbf{1}_{3I_2}) \rangle| 
+ |\langle T(\mathbf{1}_{I'_1} \otimes h_{I_2}), 
\mathbf{1}_{I''_1} \otimes (\phi_{I_2 J_2} \mathbf{1}_{(3I_2)^c}) \rangle| 
\\ \nonumber
&\le \mathscr{Q}_1(I'_1) 
[\mathscr{Q}_2(I_2) + \mathscr{R}_2(I_2)] |I_2|^{-\frac12} |J_2|^{-\frac12}
\lesssim \widetilde{F}_1(I_1) |I_1|\, \widetilde{F}_2(I_2) \rs(I_2, J_2)^{\frac{n_2}{2}}.
\end{align}
Hence, \eqref{eq:THH821} follows from \eqref{U1}--\eqref{U3}. 

If $\d(I_2, J_{I_2}^c) > \ell(I_2)^{\frac12} \ell(J_2)^{\frac12}$, then the compact partial kernel representation and \eqref{def:RIJ} yield 
\begin{align}\label{U4}
|\langle T(\mathbf{1}_{I'_1} \otimes h_{I_2}), 
\mathbf{1}_{I'_1} \otimes \phi_{I_2 J_2} \rangle| 
&\lesssim C(\mathbf{1}_{I'_1}, \mathbf{1}_{I'_1})  
\mathscr{R}_2(I_2, J_2) |I_2|^{-\frac12} |J_2|^{-\frac12}
\\ \nonumber
&\lesssim F_1(I'_1) |I_1| \, \widetilde{F}_2(I_2, J_2) 
\rs(I_2, J_2)^{\frac{n_2}{2} + \frac{\delta_2}{2}}
\end{align}
and the compact full kernel representation, \eqref{def:Q}, and \eqref{def:RIJ} imply 
\begin{align}\label{U5}
|\langle T(\mathbf{1}_{I'_1} \otimes h_{I_2}), 
\mathbf{1}_{I''_1} \otimes \phi_{I_2 J_2} \rangle| 
&\lesssim \mathscr{Q}_1(I'_1)  
\mathscr{R}_2(I_2, J_2) |I_2|^{-\frac12} |J_2|^{-\frac12}
\\ \nonumber
&\lesssim \widetilde{F}_1(I_1) |I_1| \, \widetilde{F}_2(I_2, J_2) 
\rs(I_2, J_2)^{\frac{n_2}{2} + \frac{\delta_2}{2}}
\end{align}
Therefore, \eqref{eq:THH822} is a consequence of \eqref{U1} and \eqref{U4}--\eqref{U5}. The inequalities \eqref{eq:THH823} and \eqref{eq:THH824} can be shown symmetrically.
\end{proof}

Denote 
\begin{align*}
\mathscr{P}_1^N
&:= \bigg|\sum_{\substack{J_1 \not\in \D_1(2N) \\ \text{or } J_2 \not\in \D_2(2N)}}
\sum_{\substack{I_1 \in \D_1 \\ I_1 = J_1}}
\sum_{\substack{I_2 \in \D_2 \\ I_2 \subsetneq J_2}}
\langle T (h_{I_1} \otimes h_{I_2}), h_{I_1} \otimes 1 \rangle 
\langle h_{J_2} \rangle_{I_2} \, 
f_{I_1 I_2} \, g_{J_1 J_2} \bigg|, 
\\
\mathscr{P}_2^N
&:= \bigg|\sum_{\substack{J_1 \not\in \D_1(2N) \\ \text{or } J_2 \not\in \D_2(2N)}}
\sum_{\substack{I_1 \in \D_1 \\ I_1 \subsetneq J_1}}
\sum_{\substack{I_2 \in \D_2 \\ I_2 = J_2}}
\langle T (h_{I_1} \otimes h_{I_2}),  1 \otimes h_{I_2} \rangle 
\langle h_{J_1} \rangle_{I_1} \, 
f_{I_1 I_2} \, g_{J_1 J_2} \bigg|. 
\end{align*}

\begin{lemma}\label{lem:PP12}
For any $\varepsilon>0$, there exists an $N_0>1$ so that for all $N \ge N_0$, 
\begin{align*}
\mathscr{P}_i^N 
\lesssim \varepsilon \|f\|_{L^2} \|g\|_{L^2}, \qquad i=1, 2. 
\end{align*}
\end{lemma}

\begin{proof}
By symmetry, it suffices to bound $\mathscr{P}_1^N$. We split 
\begin{align*}
\mathscr{P}_1^N
\le \bigg|\sum_{\substack{I_1 \not\in \D_1(2N) \\ J_2 \in \D_2}} 
\sum_{\substack{I_2 \in \D_2 \\ I_2 \subsetneq J_2}} \cdots \bigg| 
+ \bigg|\sum_{\substack{I_1 \in \D_1(2N) \\ J_2 \not\in \D_2(2N)}} 
\sum_{\substack{I_2 \in \D_2 \\ I_2 \subsetneq J_2}} \cdots \bigg|  
=: \mathscr{P}_{1, 1}^N + \mathscr{P}_{1, 2}^N. 
\end{align*}
Let $b_{I_1} := \langle T^*(h_{I_1} \otimes 1), h_{I_1}\rangle$. As argued in \eqref{Nbb}, by \eqref{B3} and Lemma \ref{lem:FE} part \eqref{list:FE-4}, there exists an $N_0>1$ so that for any $N \ge N_0$, 
\begin{align}\label{P11N}
\mathscr{P}_{1, 1}^N
&= \bigg|\sum_{I_1 \not\in \D_1(2N)} \big\langle h_{I_1} \otimes 
\Pi_{b_{I_1}}^*(\langle f, h_{I_1} \rangle), g \big\rangle \bigg| 
\\ \nonumber
&\le \bigg(\sum_{I_1 \not\in \D_1(2N)} 
\|\Pi_{b_{I_1}}^*(\langle f, h_{I_1} \rangle)\|_{L^2(\R^{n_2})}^2 \bigg)^{\frac12}  \|g\|_{L^2}
\\ \nonumber
&\lesssim \bigg(\sum_{I_1 \not\in \D_1(2N)} \|b_{I_1}\|_{\BMO_{\D_2}}^2 
\|\langle f, h_{I_1} \rangle\|_{L^2(\R^{n_2})}^2 \bigg)^{\frac12}  \|g\|_{L^2}
\\ \nonumber
&\lesssim \sup_{I_1 \not\in \D_1(2N)} \widehat{F}_1(I_1) 
\bigg(\sum_{I_1 \in \D_1} \|\langle f, h_{I_1} \rangle\|_{L^2(\R^{n_2})}^2 \bigg)^{\frac12}  \|g\|_{L^2}
\\ \nonumber 
&\le \varepsilon \|f\|_{L^2} \|g\|_{L^2}. 
\end{align}
To proceed, we split  
\begin{align*}
\mathscr{P}_{1, 2}^N  
= \bigg|\sum_{I_1 \in \D_1(2N)} \big\langle h_{I_1} \otimes 
P_{2N}^{\perp} \Pi_{b_{I_1}}^* (\langle f, h_{I_1}\rangle),  
g \big\rangle \bigg|
\le \mathscr{P}_{1, 2}^{N, 1}  + \mathscr{P}_{1, 2}^{N, 2}, 
\end{align*}
where 
\begin{align*}
\mathscr{P}_{1, 2}^{N, 1}  
&:= \bigg|\sum_{I_1 \in \D_1(2N)} \big\langle h_{I_1} \otimes 
P_{2N}^{\perp} \Pi_{P_N b_{I_1}}^* (\langle f, h_{I_1}\rangle),  g \big\rangle \bigg|,  
\\
\mathscr{P}_{1, 2}^{N, 2} 
&:= \bigg|\sum_{I_1 \in \D_1(2N)} \big\langle h_{I_1} \otimes 
P_{2N}^{\perp} \Pi_{P_N^{\perp} b_{I_1}}^* (\langle f, h_{I_1}\rangle),  g \big\rangle \bigg|. 
\end{align*}
Similarly to \eqref{P11N}, there holds 
\begin{align*}
\mathscr{P}_{1, 2}^{N, 2}
&\lesssim \bigg(\sum_{I_1 \in \D_1} \|P_N^{\perp} b_{I_1}\|_{\BMO_{\D_2}}^2 
\|\langle f, h_{I_1} \rangle\|_{L^2(\R^{n_2})}^2 \bigg)^{\frac12}  \|g\|_{L^2}
\\ 
&\lesssim \sup_{I_2 \not\in \D_2(2N)} \widehat{F}_2(I_2) 
\bigg(\sum_{I_1 \in \D_1} \|\langle f, h_{I_1} \rangle\|_{L^2(\R^{n_2})}^2 \bigg)^{\frac12}  \|g\|_{L^2}
\\  
&\le \varepsilon \|f\|_{L^2} \|g\|_{L^2}. 
\end{align*}
Moreover, by \eqref{B3}, \eqref{IDN-1}, and \eqref{IDN-2}, we arrive at 
\begin{align*}
\mathscr{P}_{1, 2}^{N, 1}
&= \bigg|\sum_{\substack{I_1 \in \D_1(2N) \\ J_2 \not\in \D_2(2N)}}
\sum_{\substack{I_2 \in \D_2(N) \\ I_2 \subsetneq J_2}}
\langle b_{I_1}, h_{I_2} \rangle 
\langle h_{J_2} \rangle_{I_2} \, 
f_{I_1 I_2} \, g_{I_1 J_2} \bigg|
\\
&\le \sum_{\substack{I_1 \in \D_1 \\ J_2 \not\in \D_2(2N)}}
\sum_{\substack{I_2 \in \D_2(N) \\ I_2 \subsetneq J_2}}
\|b_{I_1}\|_{\BMO_{\D_2}}  \|h_{I_2}\|_{\mathrm{H}^1_{\D_2}}
|J_2|^{-\frac12}  |f_{I_1 I_2}|  |g_{I_1 J_2}|
\\
&\lesssim \sum_{\substack{-N \le \ell_2 \le N \\ k_2 > 2N - \ell_2}}  
2^{-k_2 \frac{n_2}{2}} \sum_{I_1 \in \D_1} 
\sum_{I_2 \in \D_2^{\ell_2}(N)} |f_{I_1 I_2}| |g_{I_1 I_2^{(k_2)}}| 
\\
&\le \sum_{\substack{-N \le \ell_2 \le N \\ k_2 > 2N - \ell_2}}  2^{-k_2 \frac{n_2}{2}}
\bigg(\sum_{\substack{I_1 \in \D_1 \\ I_2 \in \D_2}} |f_{I_1 I_2}|^2 \bigg)^{\frac12} 
\bigg(\sum_{\substack{I_1 \in \D_1 \\ J_2 \in \D_2}} 
\sum_{\substack{I_2 \in \D_2^{\ell_2}(N) \\ I_2^{(k_2)}=J_2}} 
|g_{I_1 J_2}|^2 \bigg)^{\frac12} 
\\
& \lesssim \sum_{\substack{-N \le \ell_2 \le N \\ k_2 > 2N - \ell_2}}  
2^{-k_2 \frac{n_2}{2}} N^{\frac{n_2}{2}} 2^{(N-\ell_2) \frac{n_2}{2}} 
\|f\|_{L^2} \|g\|_{L^2}
\\
&\lesssim N^{1 + \frac{n_2}{2}} 2^{-N \frac{n_2}{2}} \|f\|_{L^2} \|g\|_{L^2}
\le \varepsilon \|f\|_{L^2} \|g\|_{L^2}, 
\end{align*}
for all $N \ge 2$ large enough. Then, 
\begin{align}\label{P12N}
\mathscr{P}_{1, 2}^N
\le \mathscr{P}_{1, 2}^{N, 1}  + \mathscr{P}_{1, 2}^{N, 2} 
\lesssim \varepsilon \|f\|_{L^2} \|g\|_{L^2}. 
\end{align}
Hence, it follows from \eqref{P11N} and \eqref{P12N} that $\mathscr{P}_1^N \le \mathscr{P}_{1, 1}^N + \mathscr{P}_{1, 2}^N \lesssim \varepsilon \|f\|_{L^2} \|g\|_{L^2}$. 
\end{proof}

Third, we work in the scenario $I_1 \subsetneq J_1$ and $I_2 \subsetneq J_2$. As before, we have 
\begin{align}\label{TH4}
\langle T (h_{I_1} \otimes h_{I_2}), h_{J_1} \otimes h_{J_2} \rangle
&= \langle T (h_{I_1} \otimes h_{I_2}), \phi_{I_1 J_1} \otimes \phi_{I_2 J_2} \rangle
\\ \nonumber
&\quad+ \langle T (h_{I_1} \otimes h_{I_2}), \phi_{I_1 J_1} \otimes 1 \rangle \langle h_{J_2} \rangle_{I_2} 
\\ \nonumber
&\quad+ \langle T (h_{I_1} \otimes h_{I_2}), 1 \otimes \phi_{I_2 J_2} \rangle \langle h_{J_1} \rangle_{I_1}
\\ \nonumber
&\quad+ \langle T (h_{I_1} \otimes h_{I_2}), 1 \rangle \langle h_{J_1} \rangle_{I_1} \langle h_{J_2} \rangle_{I_2}.
\end{align}

\begin{lemma}\label{lem:THH-8}
Let $I_1, J_1 \in \D_1$ with $I_1 \subsetneq J_1$, and let $I_2, J_2 \in \D_2$ with $I_2 \subsetneq J_2$. Then the following statements hold: 
\begin{enumerate}
\item[(a)] One has 
\begin{align}\label{eq:THH81}
|\langle T (h_{I_1} \otimes h_{I_2}),
\phi_{I_1 J_1} \otimes \phi_{I_2 J_2} \rangle| 
\lesssim \prod_{i=1}^2 \widetilde{F}_i(I_i) \rs(I_i, J_i)^{\frac{n_i}{2}}. 
\end{align}

\item[(b)] If $\d(I_2, J_{I_2}^c) > \ell(I_2)^{\frac12} \ell(J_2)^{\frac12}$, then  
\begin{align}\label{eq:THH82}
|\langle T (h_{I_1} \otimes h_{I_2}),
\phi_{I_1 J_1} \otimes \phi_{I_2 J_2} \rangle| 
\lesssim \widetilde{F}_1(I_1) \rs(I_1, J_1)^{\frac{n_1}{2}} 
\widetilde{F}_2(I_2, J_2) \rs(I_2, J_2)^{\frac{n_2}{2} + \frac{\delta_2}{2}}. 
\end{align}

\item[(c)] If $\d(I_1, J_{I_1}^c) > \ell(I_1)^{\frac12} \ell(J_1)^{\frac12}$, then 
\begin{align}\label{eq:THH83}
|\langle T (h_{I_1} \otimes h_{I_2}),
\phi_{I_1 J_1} \otimes \phi_{I_2 J_2} \rangle| 
\lesssim \widetilde{F}_1(I_1, J_1) \rs(I_1, J_1)^{\frac{n_1}{2} + \frac{\delta_1}{2}}
\widetilde{F}_2(I_2) \rs(I_2, J_2)^{\frac{n_2}{2}}. 
\end{align}

\item[(d)] If $\d(I_i, J_{I_i}^c) > \ell(I_i)^{\frac12} \ell(J_i)^{\frac12}$ for each $i=1, 2$, then  
\begin{align}\label{eq:THH84}
|\langle T (h_{I_1} \otimes h_{I_2}),
\phi_{I_1 J_1} \otimes \phi_{I_2 J_2} \rangle| 
\lesssim \prod_{i=1}^2 \widetilde{F}_i(I_i, J_i) \rs(I_i, J_i)^{\frac{n_i}{2} + \frac{\delta_i}{2}}. 
\end{align}
\end{enumerate} 
\end{lemma}

\begin{proof}
Let us first prove \eqref{eq:THH81}. Noting that $\supp(\phi_{I_i J_i}) \subset J_{I_i}^c \subset I_i^c$ for each $i=1, 2$, we have   
\begin{align}\label{EE}
\mathscr{E} 
:= \langle T (h_{I_1} \otimes h_{I_2}),
\phi_{I_1 J_1} \otimes \phi_{I_2 J_2} \rangle 
= (\mathscr{E}_{1, 1} + \mathscr{E}_{1, 2}) 
+ (\mathscr{E}_{2, 1} + \mathscr{E}_{2, 2}) 
=: \mathscr{E}_1 + \mathscr{E}_2,   
\end{align}
where 
\begin{align*}
\mathscr{E}_{1, 1} 
&:= \langle T (h_{I_1} \otimes h_{I_2}), (\phi_{I_1 J_1} \mathbf{1}_{3I_1 \setminus I_1}) 
\otimes (\phi_{I_2 J_2} \mathbf{1}_{3I_2 \setminus I_2})\rangle, 
\\
\mathscr{E}_{1, 2}
&:= \langle T (h_{I_1} \otimes h_{I_2}), (\phi_{I_1 J_1} \mathbf{1}_{3I_1 \setminus I_1})
\otimes (\phi_{I_2 J_2} \mathbf{1}_{(3I_2)^c})\rangle, 
\\ 
\mathscr{E}_{2, 1}
&:= \langle T (h_{I_1} \otimes h_{I_2}), (\phi_{I_1 J_1} \mathbf{1}_{(3I_1)^c}) 
\otimes (\phi_{I_2 J_2} \mathbf{1}_{3I_2 \setminus I_2})\rangle, 
\\
\mathscr{E}_{2, 2}
&:= \langle T (h_{I_1} \otimes h_{I_2}), (\phi_{I_1 J_1} \mathbf{1}_{(3I_1)^c})
\otimes (\phi_{I_2 J_2} \mathbf{1}_{(3I_2)^c})\rangle. 
\end{align*}
By the compact full kernel representation, the size condition, and \eqref{def:Q}, we obtain
\begin{align}\label{EE-1}
|\mathscr{E}_{1, 1}|  
\lesssim \prod_{i=1}^2 |I_i|^{-\frac12} |J_i|^{-\frac12} \mathscr{Q}_i(I_i)
\lesssim \prod_{i=1}^2 \widetilde{F}_i(I_i) \rs(I_1, J_1)^{\frac{n_1}{2}}. 
\end{align}
It follows from the compact full kernel representation, the cancellation of $h_{I_2}$, the mixed size-H\"{o}lder condition, \eqref{def:Q}, and \eqref{def:R} that 
\begin{align}\label{EE-2}
|\mathscr{E}_{1, 2}| 
\lesssim \mathscr{Q}_1(I_1) \, \mathscr{R}_2(I_2)
\prod_{i=1}^2 |I_i|^{-\frac12} |J_i|^{-\frac12} 
\lesssim \prod_{i=1}^2 \widetilde{F}_i(I_i) \rs(I_1, J_1)^{\frac{n_1}{2}}.
\end{align}
Symmetrically, 
\begin{align}\label{EE-3}
|\mathscr{E}_{2, 1}| 
\lesssim \mathscr{R}_1(I_1) \mathscr{Q}_2(I_2) 
\prod_{i=1}^2 |I_i|^{-\frac12} |J_i|^{-\frac12}
\lesssim \prod_{i=1}^2 \widetilde{F}_i(I_i) \rs(I_1, J_1)^{\frac{n_1}{2}}.
\end{align}
By the compact full kernel representation, the H\"{o}lder condition, and \eqref{def:R}, there holds 
\begin{align}\label{EE-4}
|\mathscr{E}_{2, 2}| 
\lesssim \prod_{i=1}^2 \mathscr{R}_i(I_i) |I_i|^{-\frac12} |J_i|^{-\frac12} 
\lesssim \prod_{i=1}^2 \widetilde{F}_i(I_i) \rs(I_i, J_i)^{\frac{n_i}{2}}. 
\end{align}
Hence, \eqref{eq:THH81} immediately follows from \eqref{EE}--\eqref{EE-4}.

If $\d(I_2, J_{I_2}^c) > \ell(I_2)^{\frac12} \ell(J_2)^{\frac12} \ge \ell(I_2)$, then the compact full kernel representation, the mixed size-H\"{o}lder condition of $K$, \eqref{def:Q}, and \eqref{def:RIJ} imply  
\begin{align}\label{EE-5}
|\mathscr{E}_1| 
&\lesssim \mathscr{Q}_1(I_1) \mathscr{R}_2(I_2, J_2) 
\prod_{i=1}^2 |I_i|^{-\frac12} |J_i|^{-\frac12} 
\\ \nonumber
&\lesssim \widetilde{F}_1(I_1) \rs(I_1, J_1)^{\frac{n_1}{2}} 
\, \widetilde{F}_2(I_2, J_2) \rs(I_2, J_2)^{\frac{n_2}{2} + \frac{\delta_2}{2}}, 
\end{align}
and the H\"{o}lder condition of $K$, \eqref{def:R}, and \eqref{def:RIJ} give   
\begin{align}\label{EE-6}
|\mathscr{E}_2| 
&\lesssim \mathscr{R}_1(I_1) \mathscr{R}_2(I_2, J_2) 
\prod_{i=1}^2 |I_i|^{-\frac12} |J_i|^{-\frac12} 
\\ \nonumber
&\lesssim \widetilde{F}_1(I_1) \rs(I_1, J_1)^{\frac{n_1}{2}} 
\widetilde{F}_2(I_2, J_2) \rs(I_2, J_2)^{\frac{n_2}{2} + \frac{\delta_2}{2}}.
\end{align}
Thus, \eqref{EE} and \eqref{EE-5}--\eqref{EE-6} yield \eqref{eq:THH82} as desired. Symmetrically, one can show \eqref{eq:THH83}. 

Finally, \eqref{eq:THH84} is just a consequence of the compact full kernel representation, the H\"{o}lder condition of $K$, and \eqref{def:RIJ}.  
\end{proof}

Denote 
\begin{align*}
\mathscr{P}_3^N
&:= \bigg|\sum_{\substack{J_1 \not\in \D_1(2N) \\ \text{or } J_2 \not\in \D_2(2N)}}
\sum_{\substack{I_1 \in \D_1 \\ I_1 \subsetneq J_1}}
\sum_{\substack{I_2 \in \D_2 \\ I_2 \subsetneq J_2}}
\langle T (h_{I_1} \otimes h_{I_2}), \phi_{I_1 J_1} \otimes 1 \rangle 
\langle h_{J_2} \rangle_{I_2} \, 
f_{I_1 I_2} \, g_{J_1 J_2} \bigg|,  
\\
\mathscr{P}_4^N
&:= \bigg|\sum_{\substack{J_1 \not\in \D_1(2N) \\ \text{or } J_2 \not\in \D_2(2N)}}
\sum_{\substack{I_1 \in \D_1 \\ I_1 \subsetneq J_1}}
\sum_{\substack{I_2 \in \D_2 \\ I_2 \subsetneq J_2}}
\langle T (h_{I_1} \otimes h_{I_2}), 1 \otimes \phi_{I_2 J_2} \rangle 
\langle h_{J_1} \rangle_{I_1} \, 
f_{I_1 I_2} \, g_{J_1 J_2} \bigg|, 
\\  
\mathscr{P}_5^N
&:= \bigg|\sum_{\substack{J_1 \not\in \D_1(2N) \\ \text{or } J_2 \not\in \D_2(2N)}}
\sum_{\substack{I_1 \in \D_1 \\ I_1 \subsetneq J_1}}
\sum_{\substack{I_2 \in \D_2 \\ I_2 \subsetneq J_2}}
\langle T (h_{I_1} \otimes h_{I_2}), 1 \rangle 
\langle h_{J_1} \rangle_{I_1} \langle h_{J_2} \rangle_{I_2} \, 
f_{I_1 I_2} \, g_{J_1 J_2} \bigg|. 
\end{align*}

\begin{lemma}\label{lem:PP345}
For any $\varepsilon>0$, there exists an $N_0>1$ so that for all $N \ge N_0$, 
\begin{align*}
\mathscr{P}_i^N 
\lesssim \varepsilon \|f\|_{L^2} \|g\|_{L^2}, \qquad i=3, 4, 5. 
\end{align*}
\end{lemma}

\begin{proof}
Let $\varepsilon>0$. By Lemma \ref{lem:Pi} and the fact that $b := T^*1 \in \CMO(\R^{n_1} \times \R^{n_2})$, there exists an $N_0>1$ so that for all $N \ge N_0$, 
\begin{align*}
\mathscr{P}_5^N
= \big|\langle P_{2N}^{\perp} \Pi_b^* f, g \rangle\big| 
\le \|P_{2N}^{\perp} \Pi_b^*\|_{L^2 \to L^2} \|f\|_{L^2} \|g\|_{L^2}
\le \varepsilon \|f\|_{L^2} \|g\|_{L^2}.
\end{align*}
By symmetry between $\mathscr{P}_3^N$ and $\mathscr{P}_4^N$, it suffices to estimate $\mathscr{P}_3^N$. We split 
\begin{align*}
\mathscr{P}_3^N
\le \bigg|\sum_{\substack{J_1 \not\in \D_1(2N) \\ J_2 \in \D_2}}
\sum_{\substack{I_1 \in \D_1 \\ I_1 \subsetneq J_1}}
\sum_{\substack{I_2 \in \D_2 \\ I_2 \subsetneq J_2}} \cdots \bigg| 
+ \bigg|\sum_{\substack{J_1 \in \D_1(2N) \\ J_2 \not\in \D_2(2N)}}
\sum_{\substack{I_1 \in \D_1 \\ I_1 \subsetneq J_1}}
\sum_{\substack{I_2 \in \D_2 \\ I_2 \subsetneq J_2}} \cdots \bigg| 
=: \mathscr{P}_{3, 1}^N + \mathscr{P}_{3, 2}^N. 
\end{align*}
Let $b_{I_1 J_1} := \langle T^*(\phi_{I_1 J_1} \otimes 1), h_{I_1} \rangle$. As done in \eqref{JDN-2} and \eqref{Nbb}, we invoke \eqref{B4}, \eqref{B5}, and Lemma \ref{lem:FE} parts \eqref{list:FE-2}--\eqref{list:FE-4} to arrive at  
\begin{align*}
\mathscr{P}_{3, 1}^N 
\lesssim \sum_{i=1}^4 \mathscr{P}_{3, 1}^{N, i}, 
\end{align*}
where 
\begin{align*}
\mathscr{P}_{3, 1}^{N, 1}
&:= \varepsilon \bigg[ \sum_{J_1 \in \D_1} 
\bigg(\sum_{k_1 \ge 1}  \sum_{I_1 \in \mathcal{B}_1^{k_1}(J_1)} 
2^{-k_1 \frac{n_1}{2}}  
\|\langle f, h_{I_1}\rangle\|_{L^2(\R^{n_2})} \bigg)^2 
\bigg]^{\frac12} \, \|g\|_{L^2}, 
\\ 
\mathscr{P}_{3, 1}^{N, 2}
&:= \bigg[ \sum_{J_1 \in \D_1} 
\bigg(\sum_{k_1 \ge N}  \sum_{I_1 \in \mathcal{B}_1^{k_1}(J_1)} 
2^{-k_1 \frac{n_1}{2}}  
\|\langle f, h_{I_1}\rangle\|_{L^2(\R^{n_2})} \bigg)^2 
\bigg]^{\frac12} \, \|g\|_{L^2}, 
\\ 
\mathscr{P}_{3, 1}^{N, 3}
&:= \varepsilon \Bigg[ \sum_{J_1 \in \D_1} 
\bigg(\sum_{k_1 \ge 1} \sum_{I_1 \in \mathcal{G}_1^{k_1}(J_1)}
2^{-k_1(\frac{n_1}{2} + \frac{\delta_1}{2})}  
\|\langle f, h_{I_1}\rangle\|_{L^2(\R^{n_2})} \bigg)^2 
\Bigg]^{\frac12} \, \|g\|_{L^2}, 
\\ 
\mathscr{P}_{3, 1}^{N, 4}
&:= \Bigg[ \sum_{J_1 \in \D_1} 
\bigg(\sum_{k_1 \ge N} \sum_{I_1 \in \mathcal{G}_1^{k_1}(J_1)}
2^{-k_1(\frac{n_1}{2} + \frac{\delta_1}{2})}  
\|\langle f, h_{I_1}\rangle\|_{L^2(\R^{n_2})} \bigg)^2 
\Bigg]^{\frac12} \, \|g\|_{L^2}. 
\end{align*}
Applying the Cauchy--Schwarz inequality twice and \eqref{GB}, we obtain that for any $N \ge 1$ large enough, 
\begin{align}\label{P31N2}
\mathscr{P}_{3, 1}^{N, 2}
&\lesssim \bigg[ \sum_{J_1 \in \D_1} 
\bigg(\sum_{k_1 \ge N} 2^{-\frac{k_1}{4}} 
\Big(\sum_{I_1 \in \mathcal{B}_1^{k_1}(J_1)} 
\|\langle f, h_{I_1}\rangle\|_{L^2(\R^{n_2})}^2 \Big)^{\frac12} \bigg)^2 
\bigg]^{\frac12} \, \|g\|_{L^2}
\\ \nonumber
&\lesssim \bigg[ \sum_{J_1 \in \D_1} 
\sum_{k_1 \ge N} 2^{-\frac{k_1}{4}} 
\sum_{I_1 \in \mathcal{B}_1^{k_1}(J_1)} 
\|\langle f, h_{I_1}\rangle\|_{L^2(\R^{n_2})}^2 
\bigg]^{\frac12} \, \|g\|_{L^2}
\\ \nonumber
&\le \bigg[  \sum_{k_1 \ge N} 2^{-\frac{k_1}{4}} 
\sum_{I_1 \in \D_1} \sum_{J_1=I_1^{(k_1)}} 
\|\langle f, h_{I_1}\rangle\|_{L^2(\R^{n_2})}^2 
\bigg]^{\frac12} \, \|g\|_{L^2}
\\ \nonumber 
&\lesssim 2^{-N/8} \|f\|_{L^2} \|g\|_{L^2}
\le \varepsilon \|f\|_{L^2} \|g\|_{L^2}, 
\end{align}
and 
\begin{align}\label{P31N4}
\mathscr{P}_{3, 1}^{N, 4}
&\lesssim \bigg[ \sum_{J_1 \in \D_1} 
\bigg(\sum_{k_1 \ge N} 2^{- k_1 \frac{\delta_1}{2}} 
\Big(\sum_{I_1 \in \mathcal{G}_1^{k_1}(J_1)} 
\|\langle f, h_{I_1}\rangle\|_{L^2(\R^{n_2})}^2 \Big)^{\frac12} \bigg)^2 
\bigg]^{\frac12} \, \|g\|_{L^2}
\\ \nonumber 
&\lesssim \bigg[ \sum_{J_1 \in \D_1} 
\sum_{k_1 \ge N} 2^{- k_1 \frac{\delta_1}{2}} 
\sum_{I_1 \in \mathcal{G}_1^{k_1}(J_1)} 
\|\langle f, h_{I_1}\rangle\|_{L^2(\R^{n_2})}^2 
\bigg]^{\frac12} \, \|g\|_{L^2}
\\ \nonumber 
&\le \bigg[  \sum_{k_1 \ge N} 2^{- k_1 \frac{\delta_1}{2}}  
\sum_{I_1 \in \D_1} \sum_{J_1=I_1^{(k_1)}} 
\|\langle f, h_{I_1}\rangle\|_{L^2(\R^{n_2})}^2 
\bigg]^{\frac12} \, \|g\|_{L^2}
\\ \nonumber 
&\lesssim 2^{-N \delta_1/4} \|f\|_{L^2} \|g\|_{L^2}
\le \varepsilon \|f\|_{L^2} \|g\|_{L^2}. 
\end{align}
Similarly, 
\begin{align*}
\mathscr{P}_{3, 1}^{N, i}  
\lesssim \varepsilon \|f\|_{L^2} \|g\|_{L^2}, \qquad i=1, 3. 
\end{align*}
Gathering these estimates above, one has 
\begin{align}\label{P31N}
\mathscr{P}_{3, 1}^N 
\lesssim \sum_{i=1}^4 \mathscr{P}_{3, 1}^{N, i} 
\lesssim  \varepsilon \|f\|_{L^2} \|g\|_{L^2}.  
\end{align}

Next, let us estimate $\mathscr{P}_{3, 2}^N$. We have 
\begin{align*}
\mathscr{P}_{3, 2}^N 
&= \bigg|\sum_{J_1 \in \D_1(2N)}
\sum_{\substack{I_1 \in \D_1 \\ I_1 \subsetneq J_1}} 
\big\langle h_{J_1} \otimes P_{2N}^{\perp} \Pi_{b_{I_1 J_1}}^* (\langle f, h_{I_1}\rangle),  
g \big\rangle \bigg| 
\le \mathscr{P}_{3, 2}^{N, 1} + \mathscr{P}_{3, 2}^{N, 2},   
\end{align*}
where 
\begin{align*}
\mathscr{P}_{3, 2}^{N, 1}  
&:= \bigg|\sum_{J_1 \in \D_1(2N)}
\sum_{\substack{I_1 \in \D_1 \\ I_1 \subsetneq J_1}} 
\big\langle h_{J_1} \otimes P_{2N}^{\perp} \Pi_{P_N b_{I_1 J_1}}^* (\langle f, h_{I_1}\rangle),  
g \big\rangle \bigg|,  
\\
\mathscr{P}_{3, 2}^{N, 2}  
&:= \bigg|\sum_{J_1 \in \D_1(2N)}
\sum_{\substack{I_1 \in \D_1 \\ I_1 \subsetneq J_1}} 
\big\langle h_{J_1} \otimes P_{2N}^{\perp} \Pi_{P_N^{\perp} b_{I_1 J_1}}^* (\langle f, h_{I_1}\rangle),  
g \big\rangle \bigg|.   
\end{align*}
As argued in \eqref{30N22}, by \eqref{PB4}, \eqref{PB5}, and Lemma \ref{lem:FE} part \eqref{list:FE-4}, there exists an $N_0 \ge 1$ so that for all $N \ge N_0$,  
\begin{align}\label{P32N2}
\mathscr{P}_{3, 2}^{N, 2} 
&\lesssim \Bigg[ \sum_{J_1 \in \D_1} 
\bigg(\sum_{\substack{I_1 \in \D_1 \\ I_1 \subsetneq J_1}} 
\|P_N^{\perp} b_{I_1 J_1}\|_{\BMO_{\D_2}}  
\|\langle f, h_{I_1}\rangle\|_{L^2(\R^{n_2})} \bigg)^2 
\Bigg]^{\frac12} \, \|g\|_{L^2}
\\ \nonumber  
& \lesssim \varepsilon \bigg[ \sum_{J_1 \in \D_1} 
\bigg(\sum_{k_1 \ge 1}  \sum_{I_1 \in \mathcal{B}_1^{k_1}(J_1)} 
2^{-k_1 \frac{n_1}{2}}  
\|\langle f, h_{I_1}\rangle\|_{L^2(\R^{n_2})} \bigg)^2 \bigg]^{\frac12} \, \|g\|_{L^2}  
\\ \nonumber 
&\quad+ \varepsilon \Bigg[ \sum_{J_1 \in \D_1} 
\bigg(\sum_{k_1 \ge 1} \sum_{I_1 \in \mathcal{G}_1^{k_1}(J_1)}
2^{-k_1(\frac{n_1}{2} + \frac{\delta_1}{2})}  
\|\langle f, h_{I_1}\rangle\|_{L^2(\R^{n_2})} \bigg)^2 
\Bigg]^{\frac12} \, \|g\|_{L^2}
\\ \nonumber 
&\lesssim \varepsilon \|f\|_{L^2} \|g\|_{L^2},  
\end{align}
where the last two summation terms were estimated as in \eqref{P31N2} and \eqref{P31N4}.

On the other hand, in light of \eqref{B4}, \eqref{B5}, and \eqref{IDN-1}, we apply the same strategy as in \eqref{30N21} to conclude 
\begin{align*}
\mathscr{P}_{3, 2}^{N, 1} 
&= \bigg|\sum_{\substack{J_1 \in \D_1(2N) \\ J_2 \not\in \D_2(2N)}}
\sum_{\substack{I_1 \in \D_1 \\ I_1 \subsetneq J_1}}
\sum_{\substack{I_2 \in \D_2(N) \\ I_2 \subsetneq J_2}}
\langle b_{I_1 J_1}, h_{I_2} \rangle 
\langle h_{J_2} \rangle_{I_2} \, 
f_{I_1 I_2} \, g_{J_1 J_2} \bigg|
\lesssim \mathscr{P}_{3, 2}^{N, 1, 1} + \mathscr{P}_{3, 2}^{N, 1, 2},  
\end{align*} 
where 
\begin{align*}
\mathscr{P}_{3, 2}^{N, 1, 1} 
&:= \sum_{k_1 \ge 1} \sum_{\substack{-N \le \ell_2 \le N \\ k_2 > 2N-\ell_2}} 
\sum_{\substack{J_1 \in \D_1 \\ I_1 \in \mathcal{B}_1^{k_1}(J_1)}} 
\sum_{I_2 \in \D_2^{\ell_2}(N)} 
2^{-k_1 \frac{n_1}{2}} 2^{-k_2 \frac{n_2}{2}} 
|f_{I_1 I_2}|  |g_{J_1 I_2^{(k_2)}}|, 
\\
\mathscr{P}_{3, 2}^{N, 1, 2} 
&:= \sum_{k_1 \ge 1} \sum_{\substack{-N \le \ell_2 \le N \\ k_2 > 2N-\ell_2}} 
\sum_{\substack{J_1 \in \D_1 \\ I_1 \in \mathcal{G}_1^{k_1}(J_1)}} 
\sum_{I_2 \in \D_2^{\ell_2}(N)} 
2^{-k_1 (\frac{n_1}{2} + \frac{\delta_1}{2})} 2^{-k_2 \frac{n_2}{2}} 
|f_{I_1 I_2}|  |g_{J_1 I_2^{(k_2)}}|. 
\end{align*}
By the Cauchy--Schwarz inequality, \eqref{IDN-2}, and \eqref{GB}, we obtain that for any $N \ge 2$ large enough,  
\begin{align*}
\mathscr{P}_{3, 2}^{N, 1, 1} 
&\le \sum_{k_1 \ge 1} \sum_{\substack{-N \le \ell_2 \le N \\ k_2 > 2N-\ell_2}} 
2^{-k_1 \frac{n_1}{2}} 2^{-k_2 \frac{n_2}{2}} 
\bigg(\sum_{\substack{I_1 \in \D_1 \\ I_2 \in \D_2}}
\sum_{J_1 = I_1^{(k_1)}} |f_{I_1 I_2}|^2 \bigg)^{\frac12}
\\
&\quad\times \bigg(\sum_{\substack{J_1 \in \D_1 \\ J_2 \in \D_2}}
\sum_{\substack{I_1 \in \mathcal{B}_1^{k_1}(J_1) \\ I_2 \in \D_2^{\ell_2}(N) \\ I_2^{(k_2)} = J_2}}  
|g_{J_1 J_2}|^2 \bigg)^{\frac12}
\\
&\lesssim \sum_{k_1 \ge 1} \sum_{\substack{-N \le \ell_2 \le N \\ k_2 > 2N-\ell_2}} 
2^{-\frac{k_1}{4}} 2^{-k_2 \frac{n_2}{2}} N^{\frac{n_2}{2}} 2^{(N-\ell_2) \frac{n_2}{2}} 
\|f\|_{L^2} \|g\|_{L^2}
\\
&\lesssim N^{1 + \frac{n_2}{2}} 2^{-N \frac{n_2}{2}} \|f\|_{L^2} \|g\|_{L^2}
\le \varepsilon \|f\|_{L^2} \|g\|_{L^2}, 
\end{align*}
and 
\begin{align*}
\mathscr{P}_{3, 2}^{N, 1, 2} 
&\le \sum_{k_1 \ge 1} \sum_{\substack{-N \le \ell_2 \le N \\ k_2 > 2N-\ell_2}} 
2^{-k_1 (\frac{n_1}{2} + \frac{\delta_1}{2})} 2^{-k_2 \frac{n_2}{2}} 
\bigg(\sum_{\substack{I_1 \in \D_1 \\ I_2 \in \D_2}}
\sum_{J_1 = I_1^{(k_1)}} |f_{I_1 I_2}|^2 \bigg)^{\frac12}
\\
&\quad\times \bigg(\sum_{\substack{J_1 \in \D_1 \\ J_2 \in \D_2}}
\sum_{\substack{I_1 \in \mathcal{G}_1^{k_1}(J_1) \\ I_2 \in \D_2^{\ell_2}(N) \\ I_2^{(k_2)} = J_2}}  
|g_{J_1 J_2}|^2 \bigg)^{\frac12}
\\
&\lesssim \sum_{k_1 \ge 1} \sum_{\substack{-N \le \ell_2 \le N \\ k_2 > 2N-\ell_2}} 
2^{-k_1 \frac{\delta_1}{2}} 2^{-k_2 \frac{n_2}{2}} 
N^{\frac{n_2}{2}} 2^{(N-\ell_2) \frac{n_2}{2}} 
\|f\|_{L^2} \|g\|_{L^2}
\\
&\lesssim N^{1 + \frac{n_2}{2}} 2^{-N \frac{n_2}{2}} \|f\|_{L^2} \|g\|_{L^2}
\le \varepsilon \|f\|_{L^2} \|g\|_{L^2}. 
\end{align*}
Then, there holds 
\begin{align*}
\mathscr{P}_{3, 2}^{N, 1} 
\le \mathscr{P}_{3, 2}^{N, 1, 1}  + \mathscr{P}_{3, 2}^{N, 1, 2} 
\lesssim \varepsilon \|f\|_{L^2} \|g\|_{L^2},  
\end{align*}
which together with \eqref{P32N2} gives 
\begin{align}\label{P32N}
\mathscr{P}_{3, 2}^N 
\le \mathscr{P}_{3, 2}^{N, 1}  + \mathscr{P}_{3, 2}^{N, 2} 
\lesssim \varepsilon \|f\|_{L^2} \|g\|_{L^2}. 
\end{align}
Consequently, from \eqref{P31N} and \eqref{P32N}, we conclude that 
\begin{align*}
\mathscr{P}_3^N 
\le \mathscr{P}_{3, 1}^N  + \mathscr{P}_{3, 2}^N \lesssim \varepsilon \|f\|_{L^2} \|g\|_{L^2}.
\end{align*}
This completes the proof. 
\end{proof}

Finally, let us turn to the estimate for $\mathscr{I}_6^N$. It follows from Lemma \ref{lem:THH-81}, \eqref{TH42}, \eqref{TH43}, Lemmas \ref{lem:THH-82} and \ref{lem:PP12}, \eqref{TH4}, Lemmas \ref{lem:THH-8} and \ref{lem:PP345} that for all $N \ge N_0$ large enough,
\begin{align*}
|\mathscr{I}_6^N| 
= \Bigg|\sum_{\substack{k_1 \ge 0 \\ k_2 \ge 0}} 
\sum_{\substack{J_1 \not\in \D_1(2N) \\ \text{or } J_2 \not\in \D_2(2N)}}
\sum_{\substack{I_1 \in \D_1^{k_1}(J_1) \\ I_2 \in \D_2^{k_2}(J_2)}}
\G_{J_1 J_2}^{I_1 I_2} \, f_{I_1 I_2} \, g_{J_1 J_2} \Bigg|
\lesssim \sum_{i=1}^5 \mathscr{P}_i^N 
+ \sum_{i=1}^8 \mathscr{I}_{6, i}^N,
\end{align*}
where
\begin{align*}
\mathscr{I}_{6, 1}^N 
&:= \varepsilon \sum_{\substack{k_1 \ge 0 \\ k_2 \ge 0}} 
\sum_{\substack{J_1 \in \D_1 \\ J_2 \in \D_2}}
\sum_{\substack{I_1 \in \mathcal{B}_1^{(k_1)}(J_1) \\ I_2 \in \mathcal{B}_2^{(k_2)}(J_2)}}
2^{-k_1 \frac{n_1}{2}} 2^{-k_2 \frac{n_2}{2}} |f_{I_1 I_2}| |g_{J_1 J_2}|,
\\
\mathscr{I}_{6, 2}^N 
&:= \sum_{\substack{k_1 \ge N \\ \text{or } k_2 \ge N}} 
\sum_{\substack{J_1 \in \D_1 \\ J_2 \in \D_2}}
\sum_{\substack{I_1 \in \mathcal{B}_1^{(k_1)}(J_1) \\ I_2 \in \mathcal{B}_2^{(k_2)}(J_2)}}
2^{-k_1 \frac{n_1}{2}} 2^{-k_2 \frac{n_2}{2}} |f_{I_1 I_2}| |g_{J_1 J_2}|,
\\ 
\mathscr{I}_{6, 3}^N 
&:= \varepsilon \sum_{\substack{k_1 \ge 1 \\ k_2 \ge 0}} 
\sum_{\substack{J_1 \in \D_1 \\ J_2 \in \D_2}}
\sum_{\substack{I_1 \in \mathcal{B}_1^{(k_1)}(J_1) \\ I_2 \in \mathcal{G}_2^{(k_2)}(J_2)}}
2^{-k_1 \frac{n_1}{2}} 2^{-k_2 (\frac{n_2}{2} + \frac{\delta_2}{2})} |f_{I_1 I_2}| |g_{J_1 J_2}|,
\\ 
\mathscr{I}_{6, 4}^N 
&:= \sum_{\substack{k_1 \ge N \\ \text{or } k_2 \ge N}}  
\sum_{\substack{J_1 \in \D_1 \\ J_2 \in \D_2}}
\sum_{\substack{I_1 \in \mathcal{B}_1^{(k_1)}(J_1) \\ I_2 \in \mathcal{G}_2^{(k_2)}(J_2)}}
2^{-k_1 \frac{n_1}{2}} 2^{-k_2 (\frac{n_2}{2} + \frac{\delta_2}{2})} |f_{I_1 I_2}| |g_{J_1 J_2}|,
\\
\mathscr{I}_{6, 5}^N 
&:= \varepsilon \sum_{k_1 \ge 1, k_2 \ge 0}  
\sum_{\substack{J_1 \in \D_1 \\ J_2 \in \D_2}}
\sum_{\substack{I_1 \in \mathcal{G}_1^{(k_1)}(J_1) \\ I_2 \in \mathcal{B}_2^{(k_2)}(J_2)}}
2^{-k_1 (\frac{n_1}{2} + \frac{\delta_1}{2})} 2^{-k_2 \frac{n_2}{2}}  |f_{I_1 I_2}| |g_{J_1 J_2}|,
\\ 
\mathscr{I}_{6, 6}^N 
&:= \sum_{\substack{k_1 \ge N \\ \text{or }k_2 \ge N}}  
\sum_{\substack{J_1 \in \D_1 \\ J_2 \in \D_2}}
\sum_{\substack{I_1 \in \mathcal{G}_1^{(k_1)}(J_1) \\ I_2 \in \mathcal{B}_2^{(k_2)}(J_2)}}
2^{-k_1 (\frac{n_1}{2} + \frac{\delta_1}{2})} 2^{-k_2 \frac{n_2}{2}}  |f_{I_1 I_2}| |g_{J_1 J_2}|,
\\
\mathscr{I}_{6, 7}^N 
&:= \varepsilon \sum_{k_1 \ge 1, k_2 \ge 1}  
\sum_{\substack{J_1 \in \D_1 \\ J_2 \in \D_2}}
\sum_{\substack{I_1 \in \mathcal{G}_1^{(k_1)}(J_1) \\ I_2 \in \mathcal{G}_2^{(k_2)}(J_2)}}
2^{-k_1 \frac{n_1}{2}} 2^{-k_2 (\frac{n_2}{2} + \frac{\delta_2}{2})} |f_{I_1 I_2}| |g_{J_1 J_2}|,
\\ 
\mathscr{I}_{6, 8}^N 
&:= \sum_{\substack{k_1 \ge N \\ \text{or } k_2 \ge N}} 
\sum_{\substack{J_1 \in \D_1 \\ J_2 \in \D_2}}
\sum_{\substack{I_1 \in \mathcal{G}_1^{(k_1)}(J_1) \\ I_2 \in \mathcal{G}_2^{(k_2)}(J_2)}}
2^{-k_1 (\frac{n_1}{2} + \frac{\delta_1}{2})} 
2^{-k_2 (\frac{n_2}{2} + \frac{\delta_2}{2})} 
|f_{I_1 I_2}| |g_{J_1 J_2}|. 
\end{align*}
By the Cauchy--Schwarz inequality and \eqref{GB}, we obtain
\begin{align*}
\mathscr{I}_{6, 1}^N 
&\le \varepsilon \sum_{k_1, k_2 \ge 0} 
2^{-k_1 \frac{n_1}{2}} 2^{-k_2 \frac{n_2}{2}}
\bigg[\sum_{\substack{I_1 \in \D_1 \\ I_2 \in \D_2}}
\sum_{\substack{J_1 = I_1^{(k_1)} \\ J_2=I_2^{(k_2)}}}
|f_{I_1 I_2}|^2 \bigg]^{\frac12} 
\bigg[\sum_{\substack{J_1 \in \D_1 \\ J_2 \in \D_2}}
\sum_{\substack{I_1 \in \mathcal{B}_1^{(k_1)}(J_1) \\ I_2 \in \mathcal{B}_2^{(k_2)}(J_2)}}
|g_{J_1 J_2}|^2 \bigg]^{\frac12} 
\\
&\lesssim \varepsilon \sum_{k_1, k_2 \ge 0} 
2^{- \frac{k_1}{4}} 2^{- \frac{k_2}{4}}
\bigg[\sum_{\substack{I_1 \in \D_1 \\ I_2 \in \D_2}} 
|f_{I_1 I_2}|^2 \bigg]^{\frac12} 
\bigg[\sum_{\substack{J_1 \in \D_1 \\ J_2 \in \D_2}}
|g_{J_1 J_2}|^2 \bigg]^{\frac12} 
\\
&\lesssim \varepsilon \|f\|_{L^2} \|g\|_{L^2}, 
\end{align*}
and for any $N \ge N_0$ large enough,
\begin{align*}
\mathscr{I}_{6, 4}^N 
&\le \sum_{\substack{k_1 \ge N \\ \text{or } k_2 \ge N}} 
2^{-k_1 \frac{n_1}{2}} 2^{-k_2 (\frac{n_2}{2} + \frac{\delta_2}{2})}
\bigg[\sum_{\substack{I_1 \in \D_1 \\ I_2 \in \D_2}}
\sum_{\substack{J_1 = I_1^{(k_1)} \\ J_2=I_2^{(k_2)}}}
|f_{I_1 I_2}|^2 \bigg]^{\frac12} 
\bigg[\sum_{\substack{J_1 \in \D_1 \\ J_2 \in \D_2}}
\sum_{\substack{I_1 \in \mathcal{B}_1^{(k_1)}(J_1) \\ I_2 \in \mathcal{G}_2^{(k_2)}(J_2)}}
|g_{J_1 J_2}|^2 \bigg]^{\frac12} 
\\
&\lesssim \sum_{\substack{k_1 \ge N \\ \text{or } k_2 \ge N}}  
2^{- \frac{k_1}{4}} 2^{- k_2 \frac{\delta_2}{2}}
\bigg[\sum_{\substack{I_1 \in \D_1 \\ I_2 \in \D_2}} 
|f_{I_1 I_2}|^2 \bigg]^{\frac12} 
\bigg[\sum_{\substack{J_1 \in \D_1 \\ J_2 \in \D_2}}
|g_{J_1 J_2}|^2 \bigg]^{\frac12} 
\\
&\lesssim (2^{-N/4} + 2^{-N\delta_2/2}) \|f\|_{L^2} \|g\|_{L^2}
\le \varepsilon \|f\|_{L^2} \|g\|_{L^2}.
\end{align*}
Similarly, one can obtain the same bound for other terms $\mathscr{I}_{6, i}^N$, $i=2, 3, 5, 6, 7, 8$. Therefore, together with Lemmas \ref{lem:PP12} and \ref{lem:PP345}, we conclude that for all $N \ge N_0$ large enough,
\begin{align*}
\mathscr{I}_6^N
\lesssim \sum_{i=1}^5 \mathscr{P}_i^N 
+ \sum_{i=1}^8 \mathscr{I}_{6, i}^N
\lesssim \varepsilon \|f\|_{L^2} \|g\|_{L^2}.
\end{align*}
So far, we have already shown \eqref{reduction-3}.

\section*{Acknowledgements}
The authors would like to thank the referee for careful reading and valuable comments, which lead to the improvement of this paper.


\end{document}